\documentclass[11pt,twoside,english,reqno,a4paper]{article}

\usepackage{a4wide}
\usepackage[utf8]{inputenc}
\usepackage{amsmath}
\usepackage{amsthm}
\usepackage{amssymb}
\usepackage{mathrsfs}
\usepackage{lmodern}
\usepackage[T1]{fontenc}
\usepackage[english]{babel}
\usepackage[numbers]{natbib}
\usepackage{url}
\usepackage{tikz-cd}
\usepackage{tikz}
\usepackage{mathtools}
\usepackage{hyperref}
\usepackage{epic}
\usepackage{pict2e}
\usepackage{paralist}
\usepackage{enumerate}
\usepackage{enumitem}
\usepackage{float}
\usepackage{mathtools}
\usepackage{blindtext}
\usepackage{framed}
\usepackage{relsize}

\usepackage{chngcntr}

\counterwithin{figure}{section}

\usepackage{scalerel,stackengine}

\stackMath
\newcommand\reallywidehat[1]{
\savestack{\tmpbox}{\stretchto{
  \scaleto{
    \scalerel*[\widthof{\ensuremath{#1}}]{\kern-.6pt\bigwedge\kern-.6pt}
    {\rule[-\textheight/2]{1ex}{\textheight}}
  }{\textheight}%
}{0.5ex}}%
\stackon[1pt]{#1}{\tmpbox}%
}

\usepackage{titletoc}

\titlecontents{section}[0.5em]{}%
{\thecontentslabel.\enspace}
{}
{\titlerule*[0.6pc]{.}\contentspage}%

\titlecontents{part}[0em]{\medskip\bfseries}%
{\thecontentslabel\enspace}
{}
{\normalfont\titlerule*[1.2pc]{}\contentspage}%

\titlecontents{subsubsection}[0.5em]{\medskip}%
{\thecontentslabel.\enspace}
{}
{\titlerule*[0.6pc]{.}\contentspage}%

\DeclareMathOperator{\SL}{SL}

\DeclareMathOperator{\GL}{GL}
\DeclareMathOperator{\Orm}{O}
\DeclareMathOperator{\tr}{tr}

\DeclareMathOperator{\Hom}{{Hom}}
\DeclareMathOperator{\Ind}{Ind}
\DeclareMathOperator{\diag}{diag}
\DeclareMathOperator{\Urm}{U}

\newcommand{\zfrak}{{\mathfrak{z}}}
\newcommand{\nfrak}{\mathfrak{n}}
\newcommand{\nbar}{\bar{\mathfrak{n}}}
\newcommand{\Nbar}{\overline{N}}

\newcommand{\R}{\mathbb{R}}
\newcommand{\C}{\mathbb{C}}
\newcommand{\g}{\mathfrak{g}}

\newcommand{\Z}{\mathbb{Z}}
\newcommand{\N}{\mathbb{N}}

\renewcommand{\Im}{\operatorname{Im}}
\renewcommand{\Re}{\operatorname{Re}}
\newcommand{\F}{\mathbb{F}}

\newcommand{\Supp}{\operatorname{Supp}}
\newcommand{\T}{T}
\newcommand{\HH}{\mathbb{H}}
\newcommand{\qi}{{\bf i}}
\newcommand{\qj}{{\bf j}}
\newcommand{\qk}{{\bf k}}
\renewcommand{\setminus}{-}

\newcommand{\LocalPoles}{\;\backslash\kern-0.8em{\backslash} \;}
\newcommand{\GlobalPoles}{\; /\kern-0.8em{/} \; }
\newcommand{\IntersectionPoles}{\mathbb{X} }

\DeclarePairedDelimiter\abs{\lvert}{\rvert}

\newtheoremstyle{normal}
{10pt}
{10pt}
{\normalfont}
{}
{\bfseries}
{}
{0.8em}
{\bfseries{\thmname{#1}\thmnumber{ #2}.\thmnote{ \hspace{0.5em}(#3)\newline}}}
 
\newtheoremstyle{kursiv}
{10pt}
{10pt}
{\itshape}
{}
{\bfseries}
{}
{0.8em}
{\bfseries{\thmname{#1}\thmnumber{ #2}.\thmnote{ \hspace{0.5em}(#3)\newline}}}

\theoremstyle{plain}
\newtheorem{theorem}{Theorem}[section]
\newtheorem{lemma}[theorem]{Lemma}
\newtheorem{corollary}[theorem]{Corollary}
\newtheorem{prop}[theorem]{Proposition}
\newtheorem{theoremalph}{Theorem}

\newtheorem{fact}{Fact}

\theoremstyle{definition}

\newtheorem{remark}[theorem]{Remark}

\numberwithin{equation}{section}

\title{Symmetry breaking operators for\\real reductive groups of rank one}
\author{Jan Frahm and Clemens Weiske}
\date{}

\begin{document}

\maketitle

\abstract{For a pair of real reductive groups $G'\subset G$ we consider the space $\Hom_{G'}(\pi|_{G'},\tau)$ of intertwining operators between spherical principal series representations $\pi$ of $G$ and $\tau$ of $G'$, also called \emph{symmetry breaking operators}. Restricting to those pairs $(G,G')$ where $\dim\Hom_{G'}(\pi|_{G'},\tau)<\infty$ and $G$ and $G'$ are of real rank one, we classify all symmetry breaking operators explicitly in terms of their distribution kernels. This generalizes previous work by Kobayashi--Speh for $(G,G')=({\rm O}(1,n+1),{\rm O}(1,n))$ to the reductive pairs
$$ (G,G') = (\Urm(1,n+1;\F),\Urm(1,m+1;\F)\times F), \qquad \F=\C,\HH,\mathbb{O}\mbox{ and }F<\Urm(n-m;\F). $$

In most cases, all symmetry breaking operators can be constructed using one meromorphic family of distributions whose poles and residues we describe in detail. In addition to this family, there may occur some sporadic symmetry breaking operators which we determine explicitly.}

\tableofcontents
\thispagestyle{empty}
\newpage

\section*{Introduction}\label{sec:introduction}
\addcontentsline{toc}{section}{Introduction}

Group representations occur in various branches of mathematics and mathematical physics. One particular problem in group representation theory is the decomposition of restricted representations. Given an irreducible representation $\pi$ of a group $G$ and a subgroup $G'\subseteq G$, the restricted representation $\pi|_{G'}$ is in general not irreducible anymore, so one is interested in a decomposition of $\pi|_{G'}$ into irreducible representations of $G'$.

In the context of finite or compact groups, an irreducible representation $\pi$ is finite-dimensional and its restriction $\pi|_{G'}$ decomposes into an algebraic direct sum of irreducible representations $\tau$ of $G'$ which occur with finite multiplicities $m(\pi,\tau)\in\Z_{\geq0}$.

In the context of real reductive groups, irreducible representations are typically infinite-dimensional and their restrictions do no longer decompose into direct sums of irreducibles. However, there still exists a notion of multiplicity which describes the occurrence of a given irreducible representation $\tau$ of $G'$ inside a restricted representation $\pi|_{G'}$ as a quotient, namely
$$ m(\pi,\tau) = \dim\Hom_{G'}(\pi|_{G'},\tau). $$
The study of these multiplicities turns out to be very fruitful and has recently attracted a lot of attention.

In contrast to the case of finite or compact groups, the multiplicities for real reductive groups are not necessarily finite. Therefore, Kobayashi~\cite{kobayashi_2015} has proposed to single out settings where only finite multiplicities occur.

\begin{fact}[{see \cite[Theorem C]{KO13}}]\label{fact:FiniteMultiplicities}
Assume that $G$ and $G'$ are defined algebraically over $\R$. Then
$$ \dim\Hom_{G'}(\pi|_{G'},\tau)<\infty $$
for all smooth admissible representations $\pi$ of $G$ and $\tau$ of $G'$ if and only if the pair $(G,G')$ is strongly spherical, i.e. the homogeneous space $(G\times G')/{\rm diag}(G')$ is real spherical.
\end{fact}

For a fixed strongly spherical reductive pair $(G,G')$ one is now interested in determining the multiplicities $m(\pi,\tau)$. It turns out that for some applications such as unitary branching laws (see e.g. \cite{MO15}) or applications to partial differential equations or analytic number theory (see e.g. \cite{FraSu_2018,MOZ16b,MO17}) the mere knowledge of the multiplicities is not sufficient and one needs an explicit description of the operators in
$$ \Hom_{G'}(\pi|_{G'},\tau), $$
so-called \emph{symmetry breaking operators} as advocated by Kobayashi in his ABC program (see \cite{kobayashi_2015}). This is the main objective of this paper.

Strongly spherical real reductive pairs have recently been classified in the case of symmetric pairs by Kobayashi--Matsuki~\cite{KM14} and in general by Knop--Krötz--Pecher--Schlichtkrull \cite{KKPS17} (The complete list can also be found in \cite{Moe17}.) In what follows we restrict to the case where $G$ is of real rank one. Up to central extensions, every simple real reductive group $G$ of rank one is of the form $\Urm(1,n+1;\F)$ with $F\in\{\R,\C,\HH\}$ and $n\geq0$ or $\F=\mathbb{O}$ and $n=1$, where we use the interpretation $\Urm(1,2;\mathbb{O})=F_{4(-20)}$.

\begin{fact}[{see \cite{KM14} and \cite{KKPS17}}]\label{fact:StronglySphericalRankOnePairs}
Up to central extensions, the irreducible strongly spherical real reductive pairs $(G,G')$ with $G$ of real rank one and $G'$ non-compact are given by
$$ (G,G') = (\Urm(1,n+1;\F),\Urm(1,m+1;\F)\times F) $$
with $\F\in\{\R,\C,\HH\}$ and $0\leq m\leq n$ or $\F=\mathbb{O}$ and $0\leq m\leq n=1$ and $F<\Urm(n-m;\F)$ such that the action of $\Urm(1;\F)\times F$ on $\F^{n-m}$ by $(a,c)\cdot z=azc^{-1}$ has an open orbit on the unit sphere in $\F^{n-m}$.
\end{fact}

Having fixed the class of groups, we now have to specify a class of representations for which to study symmetry breaking operators. By the Casselman Embedding Theorem every irreducible admissible representation of a reductive group occurs as a quotient inside a principal series representation. Principal series are families of representations induced from a parabolic subgroup $P=MAN$. Here we focus on the case of spherical principal series, i.e. those induced from representations of $P$ which are trivial on $M$ and $N$. Since $A\cong\R_+$, its representations are parametrized by a complex number $\lambda\in\C$, and we denote the corresponding induced representation of $G$ by $\pi_\lambda$. In the same way the spherical principal series $\tau_\nu$ ($\nu\in\C$) of $G'$ is constructed. The topic of this paper is a full classification of $\Hom_{G'}(\pi_\lambda|_{G'},\tau_\nu)$ for all strongly spherical real reductive pairs $(G,G')$ with $G$ of rank one and all $(\lambda,\nu)\in\C$.

For the pair $(G,G')=({\rm O}(1,n+1),{\rm O}(1,n))$ of rank one orthogonal groups, symmetry breaking operators between spherical principal series were classified by Kobayashi--Speh~\cite{kobayashi_speh_2015} (see also \cite{kobayashi_speh_2018} for the vector-valued case). We therefore restrict ourselves to the strongly spherical pairs $(G,G')=(\Urm(1,n+1;\F),\Urm(1,m+1;\F)\times F)$ with $\F\in\{\C,\HH,\mathbb{O}\}$ in this paper.

\subsection*{Multiplicities}

For $\F\in\{\C,\HH,\mathbb{O}\}$ we consider the real reductive pairs
$$ (G,G') = (\Urm(1,n+1;\F),\Urm(1,m+1;\F)\times F) $$
with $0\leq m<n$ and $F<\Urm(n-m;\F)$ a closed subgroup. We assume that the pair $(G,G')$ is strongly spherical.

Let $\pi_\lambda$ ($\lambda\in\C$) denote the spherical principal series of $G$. Our normalization is chosen, such that $\pi_\lambda$ is unitary for $\lambda\in i\R$ and has a finite-dimensional quotient if and only if $\lambda\in\rho+2\Z_{\geq0}$, where $\rho=\frac{1}{2}(p+2q)$ with $p=n\cdot\dim_\R\F$ and $q=\dim_\R\F-1\in\{1,3,7\}$. In the same way we parametrize the spherical principal series $\tau_\nu$ ($\nu\in\C$) of $G'$ and let $\rho'=\frac{1}{2}(p'+2q)$ with $p'=m\cdot\dim_\R\F$ (see Section~\ref{sec:PrincipalSeries} for the precise definitions). Note that for $m=0$ we have $\mathfrak{su}(1,1;\F)=\mathfrak{so}(1,q+1)$, so that $\mathfrak{g}'$ is
the product of an indefinite orthogonal Lie algebra and a compact Lie algebra.

To state our results for the multiplicities $\dim\Hom_{G'}(\pi_\lambda|_{G'},\tau_\nu)$ we define for $m>0$
\begin{align*}
 L &= \{ (-\rho+q-1-2i,-\rho'+q-1-2j): i,j \in \Z, 0\leq j\leq i \}
\intertext{and for $m=0$}
 L &= \{ (-\rho+q-1-2i,\pm(1+2j)): i,j \in \Z, 0\leq j\leq i \},
\intertext{as well as}
 S_1 &= \{ (-\rho+q-1-2i,\rho'+2j): i,j\in\Z, 0\leq j\leq i\},\\
 S_2 &= \{ (-\rho+q-1-2i,\rho'+2i): i\in\Z_{\geq0} \},\\
 S_3 &= \{ (-\rho+q-1,\rho') \}.
\end{align*}
Note that for $m=0$ we have $S_1\subseteq L$ for $\F=\C$ and $S_2\cap L=S_3\cap L=\emptyset$ for $\F=\HH$.

\begin{theoremalph}[Multiplicities]\label{introthm:Multiplicities}
The multiplicities of symmetry breaking operators between spherical principal series of $G$ and $G'$ are given by
$$ \dim\Hom_{G'}(\pi_\lambda|_{G'},\tau_\nu) = \begin{cases}1&\mbox{for $(\lambda,\nu)\in\C^2\setminus L$,}\\2&\mbox{for $(\lambda,\nu)\in L$,}\end{cases} $$
except in the following three cases:
\begin{enumerate}[label=(\roman{*}), ref=\thetheorem(\roman{*})]
\item For $(G,G')=(\Urm(1,n+1),\Urm(1,1)\times F)$ we have
$$ \dim\Hom_{G'}(\pi_\lambda|_{G'},\tau_\nu) = \begin{cases}1&\mbox{for $(\lambda,\nu)\in\C^2\setminus L$,}\\2&\mbox{for $(\lambda,\nu)\in L \setminus S_1$,}\\3&\mbox{for $(\lambda,\nu)\in S_1$.}\end{cases} $$
\item For $(G,G')=({\rm Sp}(1,2),{\rm Sp}(1,1))$ we have
$$ \dim\Hom_{G'}(\pi_\lambda|_{G'},\tau_\nu) = \begin{cases}1&\mbox{for $(\lambda,\nu)\in\C^2\setminus (L\cup S_2)$,}\\2&\mbox{for $(\lambda,\nu)\in L$,}\\2i+4&\mbox{for $(\lambda,\nu)=(-\rho+q-1-2i,\rho'+2i)\in S_2$,}\end{cases} $$
\item For $(G,G')=({\rm Sp}(1,2),{\rm Sp}(1,1)\times\Urm(1))$ we have
$$ \dim\Hom_{G'}(\pi_\lambda|_{G'},\tau_\nu) = \begin{cases}1&\mbox{for $(\lambda,\nu)\in\C^2\setminus(L\cup S_3)$,}\\2&\mbox{for $(\lambda,\nu)\in L\cup S_3$.}\end{cases} $$
\end{enumerate}
\end{theoremalph}

\subsection*{Regular symmetry breaking operators}

We now describe the explicit construction of all symmetry breaking operators in the space $\Hom_{G'}(\pi_\lambda|_{G'},\tau_\nu)$. For this we realize the representations $\pi_\lambda$ as smooth sections of homogeneous line bundles over the flag variety $G/P$. If $\tau_\nu$ is realized in the same way as smooth sections over $G'/P'$, $P'=P\cap G'$ a parabolic subgroup of $G'$, then symmetry breaking operators in $\Hom_{G'}(\pi_\lambda|_{G'},\tau_\nu)$ are given by convolution with a $P'$-equivariant distribution section over $G/P$ followed by restriction to $G'/P'\subseteq G/P$. For our pairs of groups $(G,G')$ the distribution kernels can most easily be described on the opposite unipotent radical $\overline{N}$ which can be identified with an open dense subset of $G/P$. In this way, symmetry breaking operators can be classified in terms of certain invariant distributions on $\overline{N}$ (see Section~\ref{sec:DistributionKernelsOnOpenBruhatCell} for details).

For $G=\Urm(1,n+1;\F)$ the group $\overline{N}$ is a $2$-step nilpotent group diffeomorphic to its Lie algebra $\nbar\cong\F^n\oplus\Im\F$, where $\Im\F=\{Z\in\F:\overline{Z}=-Z\}$. We write $\mathcal{D}'(\nbar)_{\lambda,\nu}$ for those distributions on $\nbar$ which correspond to symmetry breaking operators in $\Hom_{G'}(\pi_\lambda|_{G'},\tau_\nu)$ and now describe the classification of $\mathcal{D}'(\nbar)_{\lambda,\nu}$ in detail. For this we may without loss of generality assume that $G'$ is embedded in $G$ such that $\overline{N}\cap G'$ corresponds to the subalgebra $\nbar'=\F^m\oplus\Im\F\subseteq\nbar$ with $\F^m\subseteq\F^n$ in the first $m$ coordinates.

Since the distribution kernels on $G/P$ corresponding to symmetry breaking operators are $P'$-equivariant, the support of every such kernel is a union of $P'$-orbits in $G/P$. There are three such orbits, an open dense orbit $\mathcal{O}_A$, an intermediate orbit $\mathcal{O}_B$ and a closed orbit $\mathcal{O}_C$. Their intersections with the open dense Bruhat cell $\overline{N}P/P\simeq\nbar$ are given by
$$ \mathcal{O}_A\cap\nbar = \nbar\setminus\nbar', \qquad \mathcal{O}_B\cap\nbar = \nbar'\setminus\{0\}, \qquad \mathcal{O}_C\cap\nbar = \{0\}. $$
For each of the orbits we construct a holomorphic family of distributions which is generically supported on that orbit and contains all distributions in $\mathcal{D}'(\nbar)_{\lambda,\nu}$ which are supported on that orbit, except for $(\lambda,\nu)\in S_i$ ($i=1,2,3$) in the cases (i)-(iii) in Theorem~\ref{introthm:Multiplicities}. The remaining distribution kernels for $(\lambda,\nu)\in S_i$ are all supported on $\{0\}$ and will be described at the end of the introduction. Symmetry breaking operators whose distribution kernels are supported on the whole space $\nbar$ are called \emph{regular} while operators whose kernels have smaller support are called \emph{singular}. Note that a symmetry breaking operator is a differential operator if and only if its distribution kernel is supported on $\{0\}\subseteq\nbar$.

For $\Re(\lambda\pm\nu)>-\frac{p''}{2}$ with $p''=(n-m)\cdot\dim_\R\F=p-p'=2(\rho-\rho')$ we define the following locally integrable function on $\nbar=\F^n\oplus\Im\F $:
$$ u_{\lambda,\nu}^A(X,Z) := \frac{1}{\Gamma(\frac{\lambda+\rho-\nu-\rho'}{2})\Gamma(\frac{\lambda+\rho+\nu-\rho'}{2})}N(X,Z)^{-2(\nu+\rho')}\abs{X''}^{\lambda-\rho+\nu+\rho'}, $$
where $N(X,Z)=(\abs{X}^4+\abs{Z}^2)^{\frac{1}{4}}$ is the norm function on $\nbar$ and we have used the notation $X=(X',X'')\in\F^m\oplus\F^{n-m}=\F^n$.

\begin{theoremalph}[Regular symmetry breaking operators]\label{introthm:RegularSBOs}
$u_{\lambda,\nu}^A$ extends to a family of distributions which depends holomorphically on $(\lambda,\nu)\in\C^2$ and $u_{\lambda,\nu}^A\in\mathcal{D}'(\nbar)_{\lambda,\nu}$ for all $(\lambda,\nu)\in\C^2$. Further, $u_{\lambda,\nu}^A=0$ if and only if $(\lambda,\nu)\in L$.
\end{theoremalph}

In Corollary~\ref{cor:supports} we show that $\Supp u_{\lambda,\nu}^A=\nbar$ if and only if $(\lambda,\nu)\in\C^2\setminus(\GlobalPoles\cup\LocalPoles)$, where
\begin{align*}
 \GlobalPoles &:= \{ (\lambda,\nu)\in \C^2, \lambda+\rho-\nu-\rho' \in -2\Z_{\geq 0}   \},\\
 \LocalPoles &:= \{ (\lambda,\nu)\in \C^2, \lambda+\rho+\nu-\rho' \in -2\Z_{\geq 0}   \}.
\end{align*}
Note that $L\subseteq\IntersectionPoles:=\GlobalPoles\cap\LocalPoles$ since $q$ is odd and $p'$ is even for $\F=\C,\mathbb{H},\mathbb{O}$. For $(\lambda,\nu)\in\GlobalPoles\cup\LocalPoles$ the support of $u_{\lambda,\nu}^A$ is smaller than $\nbar$ and we now describe these distributions in more detail.

\subsection*{Singular symmetry breaking operators}

Write $\nbar=\mathfrak{v}\oplus\mathfrak{z}$ with $\mathfrak{v}=\F^n$ and $\mathfrak{z}=\Im\F$. Decompose $\mathfrak{v}=\mathfrak{v}'\oplus\mathfrak{v}''$ with $\mathfrak{v}'=\F^m$ and $\mathfrak{v}''=\F^{n-m}$, so that $\nbar'=\mathfrak{v}'\oplus\mathfrak{z}$. We write
\begin{align*}
 p &= \dim\mathfrak{v}=n\cdot\dim_\R\F,\\
 p' &= \dim\mathfrak{v}'=m\cdot\dim_\R\F, & q &= \dim\mathfrak{z}=\dim_\R\F-1.\\
 p'' &= \dim\mathfrak{v}''=(n-m)\cdot\dim_\R\F,
\end{align*}
Further, let $\Delta_{\mathfrak{v}'}$, $\Delta_{\mathfrak{v}''}$ and $\Box$ denote the Euclidean Laplacians on $\mathfrak{v}'$, $\mathfrak{v}''$ and $\mathfrak{z}$, respectively.

For $(\lambda,\nu)\in\LocalPoles$ with $\lambda+\rho+\nu-\rho'=-2l\in-2\Z_{\geq 0}$  and $\Re \nu \ll 0$ let
$$ u_{\lambda,\nu}^B(X,Z):= c^B(\lambda,\nu) N(X,Z)^{-2(\nu+\rho')}\Delta_{\mathfrak{v}''}^l\delta(X'), $$
where $c^B(\lambda,\nu)$ is a meromorphic renormalization parameter defined in \eqref{eq:singular_renorm_parameter}. Further, for $(\lambda,\nu)\in\GlobalPoles$ with $\lambda+\rho-\nu-\rho'=-2k\in -2\Z_{\geq 0}$ let
$$ u_{\lambda,\nu}^C(X,Z):= \sum_{h+i+2j=k} c_{h,i,j}(\lambda,\nu) \Delta_{\mathfrak{v}'}^h\Delta_{\mathfrak{v}''}^i\square^j\delta(X,Z), $$
where $c_{h,i,j}(\lambda,\nu)$ are holomorphic functions of $(\lambda,\nu)\in\GlobalPoles$ defined in \eqref{eq:diff_op_scalars1} and \eqref{eq:diff_op_scalars2}.

\begin{theoremalph}[Singular symmetry breaking operators]\label{introthm:ResidueFormulas}
\begin{enumerate}[label=(\roman{*}), ref=\thetheorem(\roman{*})]
\item $u_{\lambda,\nu}^B$ extends to a family of distributions which depends holomorphically on $(\lambda,\nu)\in\LocalPoles$ and $u_{\lambda,\nu}^B\in\mathcal{D}'(\nbar)_{\lambda,\nu}$ for all $(\lambda,\nu)\in\LocalPoles$. The support of $u_{\lambda,\nu}^B$ is given by
$$ \Supp u_{\lambda,\nu}^B = \begin{cases}\nbar'&\mbox{for $(\lambda,\nu)\in\LocalPoles\setminus(\IntersectionPoles\setminus L)$,}\\\{0\}&\mbox{for $(\lambda,\nu)\in\IntersectionPoles\setminus L$}\end{cases} $$
and for $\lambda+\rho+\nu-\rho'=-2l$ the following residue formula holds:
$$ u_{\lambda,\nu}^A = \frac{(-1)^l\pi^{\frac{p''}{2}}}{2^l\Gamma(\frac{p''}{2}+l)}\times\begin{cases}
u_{\lambda,\nu}^B & \text{for $m>0$ and } l \leq \frac{p'}{2}, \\
\frac{\Gamma(\frac{2\nu+p'+2}{4}+ \lfloor \frac{2l-p'+2}{4}\rfloor)}{\Gamma(\frac{2\nu+p'+2}{4})}u_{\lambda,\nu}^B & \text{for $m>0$ and } l> \frac{p'}{2}, \\
\frac{\Gamma(-\frac{\nu}{2}-\lfloor \frac{l}{2} \rfloor)}{\Gamma(-\nu-l)}
u_{\lambda,\nu}^B & \text{for $m=0$.}
\end{cases} $$
\item $u_{\lambda,\nu}^C$ extends to a family of distributions which depends holomorphically on $(\lambda,\nu)\in\GlobalPoles$ and $u_{\lambda,\nu}^C\in\mathcal{D}'(\nbar)_{\lambda,\nu}$ for all $(\lambda,\nu)\in\GlobalPoles$. The support of $u_{\lambda,\nu}^C$ is given by
$$ \Supp u_{\lambda,\nu}^C = \{0\} \qquad \mbox{for all $(\lambda,\nu)\in\GlobalPoles$} $$
and for $\lambda+\rho-\nu-\rho'=-2k$ the following residue formula holds:
\begin{align*}
 u_{\lambda,\nu}^A &= \frac{(-1)^k k! \pi^{\frac{p+q}{2}}}{\Gamma(\frac{\nu+\rho'}{2})}
 \times
 \begin{cases} \frac{\Gamma(\frac{2\nu+p'}{4})}{\Gamma(\frac{2\nu+p'}{2})}u^C_{\lambda,\nu}&\mbox{for $m>0$,}\\\frac{\Gamma(\frac{\nu}{2}-\lfloor\frac{k}{2}\rfloor)}{2^k\Gamma(\nu-k)}u^C_{\lambda,\nu}&\mbox{for $m=0$.} \end{cases}
\end{align*}
\end{enumerate}
\end{theoremalph}

Using the three families $u_{\lambda,\nu}^A$, $u_{\lambda,\nu}^B$ and $u_{\lambda,\nu}^C$ of invariant distributions associated to the orbits $\mathcal{O}_A$, $\mathcal{O}_B$ and $\mathcal{O}_C$ we have the following explicit description of $\mathcal{D}'(\nbar)_{\lambda,\nu}\cong\Hom_{G'}(\pi_\lambda|_{G'},\tau_\nu)$:

\begin{theoremalph}[Classification of symmetry breaking operators]\label{introthm:Classification}
The space $\mathcal{D}'(\nbar)_{\lambda,\nu}$ is given by
$$ \mathcal{D}'(\nbar)_{\lambda,\nu} = \begin{cases}
\C u_{\lambda,\nu}^A & \text{for } (\lambda,\nu)\in \C^2 \setminus L,\\
\C u^B_{\lambda,\nu} \oplus \C u^C_{\lambda,\nu} & \text{for } (\lambda,\nu) \in L.
\end{cases} $$
except in the cases (i), (ii) and (iii) of Theorem~\ref{introthm:Multiplicities} with $(\lambda,\nu)\in S_i$.
\end{theoremalph}

Theorems~\ref{introthm:ResidueFormulas} and \ref{introthm:Classification} show that in almost all cases the holomorphic family $u^A_{\lambda,\nu}$ and its two different renormalizations $u^B_{\lambda,\nu}$ and $u^C_{\lambda,\nu}$ when $(\lambda,\nu)\in L$ span the space of all symmetry breaking operators. Note that $u^C_{\lambda,\nu}$ is supported at the origin, so the associated symmetry breaking operator is a differential operator. In the cases (i), (ii) and (iii) of Theorem~\ref{introthm:Multiplicities} with $(\lambda,\nu)\in S_i$ there appear additional differential symmetry breaking operators that cannot be obtained as renormalizations of the holomorphic family $A_{\lambda,\nu}$ of operators with distribution kernels $u^A_{\lambda,\nu}$. Following \cite{kobayashi_speh_2018} we call such operators \emph{sporadic symmetry breaking operators} and now describe them in detail.

\subsection*{Sporadic symmetry breaking operators}

To define all sporadic differential symmetry breaking operators we treat the three cases in Theorem~\ref{introthm:Multiplicities} separately. In all three cases the differential operators $\Delta_{\mathfrak{v}'}$, $\Delta_{\mathfrak{v}''}$ and $\Box$ do not generate the algebra of $M'$-invariant differential operators on $\nbar$ with constant coefficients. The additional generators give rise to sporadic differential symmetry breaking operators.

\begin{enumerate}[label=(\roman{*})]
\item $(G,G')=(\Urm(1,n+1),\Urm(1,1)\times F)$. For $(\lambda,\nu)\in\GlobalPoles$ with $\lambda+\rho-\nu-\rho'=-2k$ we have
$$ u_{\lambda,\nu}^C = \sum_{i+2j=k} \frac{2^{-i}\Gamma(\frac{\nu}{2}-j)}{i!j!\Gamma(\frac{p}{2}+i)\Gamma(\frac{\nu}{2}-\lfloor\frac{k}{2}\rfloor)} \Delta_{\mathfrak{v}}^i\left(\frac{\partial}{\partial Z}\right)^{2j}\delta(X,Z). $$
The algebra of $M'$-invariant differential operators on $\nbar$ is generated by $\Delta_{\mathfrak{v}}$ and $\frac{\partial}{\partial Z}$ and we obtain another family of distributions by formally replacing $2j$ by $2j+1$:
$$ v_{\lambda,\nu}^C := \sum_{i+2j+1=k}\frac{2^{-i}\Gamma(\frac{\nu-1}{2}-j)}{i!\Gamma(j+\frac{3}{2})\Gamma(\frac{p}{2}+i)\Gamma(\frac{\nu-1}{2}-\lfloor\frac{k-1}{2}\rfloor)}\Delta_\mathfrak{v}^i\left(\frac{\partial}{\partial Z}\right)^{2j+1}\delta(X,Z). $$
\item $(G,G')=({\rm Sp}(1,2),{\rm Sp}(1,1))$. Note that $\nbar=\HH\oplus\Im\HH$ and $\nbar'=\Im\HH$. On $\HH$ and $\Im\HH\subseteq\HH$ we consider the usual inner product $\langle X,Y\rangle=\Re(X\overline{Y})$. With respect to this inner product we define three differential operators $P_1,P_2,P_3$ on $\nbar$ in terms of their symbols
$$ p_1(X,Z) = \langle\qi,\overline{X}ZX\rangle, \qquad p_2(X,Z) = \langle\qj,\overline{X}ZX\rangle, \qquad p_3(X,Z) = \langle\qk,\overline{X}ZX\rangle. $$
Then the algebra of $M'$-invariant differential operators on $\nbar$ is generated by $\Delta_{\mathfrak{v}}$, $\Box$ and $P_1,P_2,P_3$. For $k\in\Z_{\geq0}$ let
$$ \mathcal{H}^k(P_1,P_2,P_3)\delta = \{q(P_1,P_2,P_3)\delta(X,Z):q\in\mathcal{H}^k(\R^3)\}, $$
where $\mathcal{H}^k(\R^3)$ denotes the space of homogeneous harmonic polynomials of three variables of degree $k$.
\item $(G,G')=({\rm Sp}(1,2),{\rm Sp}(1,1)\times\Urm(1))$. The subalgebra $\mathfrak{u}(1)\subseteq\mathfrak{sp}(1)$ is spanned by a single element $U$ of the Lie algebra $\mathfrak{sp}(1)$ which is identified with $\Im\HH\subseteq\HH$. We write $U=U_1\qi+U_2\qj+U_3\qk$ and note that non-trivial $\Urm (1)$-invariant differential operators in $\mathcal{H}^{k}(P_{1},P_{2},P_{3})$ only occur for $k=0$ (then $1$ is $\Urm (1)$-invariant) and for $k=1$ (then $U_{1}P_{1}+U_{2}P_{2}+U_{3}P_{3}$ is $\Urm (1)$-invariant).
\end{enumerate}

\begin{theoremalph}[Sporadic symmetry breaking operators]\label{introthm:sporadic}
\begin{enumerate}[label=(\roman{*}), ref=\thetheorem(\roman{*})]
\item For $(G,G')=(\Urm(1,n+1),\linebreak\Urm(1,1)\times F)$ and $(\lambda,\nu)\in S_1$ we have
$$ \mathcal{D}'(\nbar)_{\lambda,\nu} =\C u^B_{\lambda,\nu} \oplus \C u_{\lambda,\nu}^C\oplus\C v_{\lambda,\nu}^C. $$
\item For $(G,G')=({\rm Sp}(1,2),{\rm Sp}(1,1))$ and $(\lambda,\nu)=(-\rho+q-1-2i,\rho'+2i)\in S_2$ we have
$$ \mathcal{D}'(\nbar)_{\lambda,\nu} = \C u_{\lambda,\nu}^C\oplus\mathcal{H}^{i+1}(P_1,P_2,P_3)\delta. $$
\item For $(G,G')=({\rm Sp}(1,2),{\rm Sp}(1,1)\times\Urm(1))$ with $\mathfrak{u}(1)=\R U$, $U=U_1\qi+U_2\qj+U_3\qk\in\Im\HH=\mathfrak{sp}(1)$, and $(\lambda,\nu)=(-\rho+q-1,\rho')\in S_3$ we have
$$ \mathcal{D}'(\nbar)_{\lambda,\nu} = \C u_{\lambda,\nu}^C\oplus\C(U_1P_1+U_2P_2+U_3P_3)\delta. $$
\end{enumerate}
\end{theoremalph}

\subsection*{Functional equations}

For $\lambda\in\R$ the representations $\pi_\lambda$ are irreducible and unitarizable if and only if $\lambda\in(-\rho+q-1,\rho-q+1)$ and the corresponding unitary representions form the \emph{complementary series}. The invariant inner product can be expressed using the normalized Knapp--Stein intertwining operators $T_\lambda:\pi_\lambda\to\pi_{-\lambda}$ which are in the non-compact picture on $\nbar=\F^n\oplus\Im\F$ given by convolution with the holomorphic family of distributions
$$ \frac{1}{\Gamma(\lambda)}N(X,Z)^{2(\lambda-\rho)}. $$
Note that the operator $T_\lambda$ is a differential operator if and only if $\lambda\in-\Z_{\geq0}$ (see Section~\ref{sec:FunctionalEquations} for details).

When a complementary series representation $\pi_\lambda$ is restricted to $G'$, it decomposes into a direct integral of irreducible unitary representations of $G'$. We expect the holomorphic family $A_{\lambda,\nu}\in\Hom_{G'}(\pi_\lambda|_{G'},\tau_\nu)$ belonging to the distribution kernels $u^A_{\lambda,\nu}$ and its renormalizations to play a key role in this decomposition (see \cite{MO15} for the case $(G,G')=({\rm O}(1,n+1),{\rm O}(1,m+1)\times{\rm O}(n-m))$). To explicitly constuct the direct integral decomposition, compositions of $A_{\lambda,\nu}$ with Knapp--Stein intertwining operators $T_\lambda$ for $G$ and $T_\nu'$ for $G'$ will be important. This motivates the following functional equations:

\begin{theoremalph}[Functional equations]\label{introthm:FunctionalEquations}
For $(\lambda,\nu)\in \C^2$ we have
$$ A_{\lambda,\nu}\circ \T_{-\lambda} = \pi^{\frac{p+q}{2}}\frac{\Gamma(\frac{2\lambda+p}{4})}{\Gamma(\frac{2\lambda+p}{2})\Gamma(\frac{\lambda+\rho}{2})}A_{-\lambda,\nu} $$
and
$$ \T'_\nu \circ A_{\lambda,\nu} = \frac{\pi^{\frac{p'+q}{2}}}{\Gamma(\frac{\nu+\rho'}{2})}\times \begin{cases}
 \frac{\Gamma(\frac{2\nu+p'}{4})}{\Gamma(\frac{2\nu+p'}{2})} A_{\lambda,-\nu} & \text{if $m>0$,}\\
 A_{\lambda,-\nu} & \text{if $m=0$.}\\
 \end{cases} $$
\end{theoremalph}

\subsection*{Relation to previous work}

The systematic construction and classification of symmetry breaking operators was initiated by T.~Kobayashi~\cite{kobayashi_2015}. Together with B.~Speh~\cite{kobayashi_speh_2015} he classified symmetry breaking operators between spherical principal series for $(G,G')=({\rm O}(1,n+1),{\rm O}(1,n))$. Our work can be seen as an extension of this classification to all strongly spherical reductive pairs $(G,G')$ with $G$ and $G'$ of real rank one. The only other strongly spherical pairs where a full classification of symmetry breaking operators between spherical principal series is known are the pairs $(G,G')=({\rm O}(1,n)\times{\rm O}(1,n),{\rm O}(1,n))$ treated by Clerc~\cite{Cle16,Cle17}.

The existence of the meromorphic family of distributions $u_{\lambda,\nu}^A$ and a generic multiplicity one statement were previously obtained by M\"{o}llers--{\O}rsted--Oshima~\cite{MOO16} (see also \cite{Moe17}), but the precise multiplicities also for singular parameters as well as the detailed study of the meromorphic nature of $u_{\lambda,\nu}^A$ were missing so far. Only the differential symmetry breaking operators corresponding to the distributions $u_{\lambda,\nu}^C$ were previously constructed by M\"{o}llers--{\O}rsted--Zhang~\cite{MOZ16}, but in particular the sporadic differential operators in Theorem~\ref{introthm:sporadic} were not known before.

\subsection*{Methods of proof}

Our general strategy of classifying symmetry breaking operators follows closely Kobayashi--Speh~\cite{kobayashi_speh_2015} (see Part~\ref{part:prelim} and in particular Section~\ref{sec:ClassificationStrategy}). This reformulates the problem of classifying symmetry breaking operators into a classification problem for invariant distributions on a nilpotent Lie algebra, the nilradical of a parabolic subgroup. However, whereas in \cite{kobayashi_speh_2015} the relevant Lie algebra is abelian, in our situation we have to deal with a $2$-step nilpotent Lie algebra which adds combinatorial difficulties to the computations.

The classification strategy essentially consists of two steps: the classification of all differential symmetry breaking operators (see Part~\ref{part:diff}) and the global study of the distribution kernels of symmetry breaking operators (see Part~\ref{part:classification}). For the classification of differential symmetry breaking operators we apply the F-method (see e.g. \cite{Kob13}) which uses the Euclidean Fourier transform on the nilpotent Lie algebra. Since our Lie algebra is $2$-step nilpotent, this results in a system of polynomial differential equations of order $4$ which we solve combinatorially. In contrast to all previous applications of the F-method where equations of order $2$ were solved, we cannot use classical orthogonal polynomials, but have to systematically solve the equations by hand. In particular, we find that in three special cases sporadic differential operators occur (see Theorem~\ref{introthm:sporadic}).

In the global study of invariant distributions corresponding to symmetry breaking operators we follow a more direct approach to study the meromorphic nature of the family $u_{\lambda,\nu}^A$. For the localization of all poles and to obtain $u_{\lambda,\nu}^C$ as renormalization of $u_{\lambda,\nu}^A$, we use polar type coordinates adapted to the nilpotent Lie algebra $\nbar$ and evaluate explicitly the resulting integrals (see Theorems~\ref{lemma:holomorphic_cont_solution} and \ref{thm:u^C}). Combined with some combinatorial computations this gives a direct way to obtain differential symmetry breaking operators from regular symmetry breaking operators. We remark that, in addition to the combinatorial construction in Part~\ref{part:diff}, this gives a second and more analytic construction of the distributions $u_{\lambda,\nu}^C$.

While the (unnormalized) family $u_{\lambda,\nu}^B$ is easily obtained from $u_{\lambda,\nu}^A$, it is much harder to find its optimal renormalization and to determine its support. We resolve this problem by explicitly decomposing the relevant distributions with respect to the decomposition $\nbar=\nbar'\oplus\mathfrak{v}''$ (see Theorem~\ref{thm:singular_family}). The resulting renormalization even improves the normalization used in \cite{kobayashi_speh_2015} for the corresponding distributions in the sense that our $u_{\lambda,\nu}^B$ never vanishes. The remaining arguments for the full classification of symmetry breaking operators then work similarly as in \cite{kobayashi_speh_2015}.

Theorem~\ref{introthm:RegularSBOs} is shown in Theorem~\ref{lemma:holomorphic_cont_solution} and Theorem~\ref{introthm:ResidueFormulas} is a combination of Theorems~\ref{thm:u^C} and \ref{thm:singular_family}. Theorem~\ref{introthm:sporadic} follows from Theorems~\ref{theorem:diff_solution_complex} and \ref{thm:diff_solution_quaternionic} and Corollary~\ref{cor:diff_solution_quaternionic}, and together with Theorems~\ref{theorem:diff_solutions} and \ref{theorem:classification} it implies Theorems~\ref{introthm:Multiplicities} and \ref{introthm:Classification}. Finally, Theorem~\ref{introthm:FunctionalEquations} is Theorem~\ref{thm:functional_equations}.

\subsection*{Acknowledgements}

The second author was supported by the DFG project 325558309.

\subsection*{Notation}

For two sets $B \subseteq A$ we use the Notation $A \setminus B= \{ a\in A: a \notin B  \}$. We denote Lie groups by Roman capitals and their corresponding Lie algebras by the corresponding Fraktur lower cases. 

\newpage

\part{Preliminaries}\label{part:prelim}

In this first part we recall from \cite{kobayashi_speh_2015} the basic theory of symmetry breaking operators between principal series representations and introduce the necessary notation for rank one real reductive groups.

\section{Symmetry breaking operators between principal series representations}

We recall the basic facts about symmetry breaking operators between principal series representations from \cite{kobayashi_speh_2015}.

\subsection{Principal series representations}\label{sec:PrincipalSeries}

Let $G$ be a real reductive Lie group and $P$ a minimal parabolic subgroup of $G$ with Langlands decomposition $P=MAN$. For a finite-dimensional representation $(\xi,V)$ of $M$, a character $\lambda \in \mathfrak{a}_\C^*$ and the trivial representation $\mathbf{1}$ of $N$ we obtain a finite-dimensional representation $(\xi \otimes e^\lambda \otimes \mathbf{1}, V_{\xi,\lambda})$ of $P=MAN$. By smooth normalized parabolic induction this representation gives rise to the principal series representation
$$ \pi_{\xi,\lambda}:= \Ind_P^G(\xi \otimes e^\lambda \otimes \mathbf{1}) $$
as the left-regular representation of $G$ on the space
$$\{ \varphi \in C^\infty(G,V): \ \varphi(gman)=\xi(m)^{-1}a^{-(\lambda+\rho)}\varphi(g) \ \forall man\in MAN    \},$$
where $\rho:= \frac{1}{2} \tr \operatorname{ad}|_{\mathfrak{n}} \in \mathfrak{a}^*_\C$.
Let $\mathcal{V}_{\xi,\lambda}:= G \times_P V_{\xi,\lambda+\rho} \to G/P$ be the homogeneous vector bundle associated to $V_{\xi,\lambda+\rho}$, then $\pi_{\xi,\lambda}$ identifies with the left-regular action of $G$ on the space of smooth sections $C^\infty(G/P,\mathcal{V}_{\xi,\lambda})$.

Now let $G'<G$ be a reductive subgroup. Similarly we let $P'=M'A'N'$ be a minimal parabolic subgroup of $G'$. For a finite-dimensional representation $(\eta,W)$ of $M'$ and $\nu \in (\mathfrak{a}'_\C)^*$ we obtain a finite-dimensional representation $(\eta \otimes e^\nu \otimes \mathbf{1},W_{\eta,\lambda})$ of $P'$ and the corresponding principal series representation
$$ \tau_{\eta,\nu}:= \Ind_{P'}^{G'}(\eta \otimes e^\nu \otimes \mathbf{1}). $$
Again we identify $\tau_{\eta,\nu}$ with the smooth sections $C^\infty(G'/P',\mathcal{W}_{\eta,\nu})$ of the homogeneous vector bundle $\mathcal{W}_{\eta,\nu}:= G' \times_{P'} W_{\eta,\nu+\rho'}\to G'/P'$, where $\rho':= \frac{1}{2}\tr \operatorname{ad}|_{\nfrak'}$.

\subsection{Symmetry breaking operators}

In these realizations the space of symmetry breaking operators between $\pi_{\xi,\lambda}$ and $\tau_{\eta,\nu}$ is given by the continuous linear $G'$-maps between the smooth sections of the two homogeneous vector bundles
$$\Hom_{G'}(\pi_{\xi,\lambda}|_{G'},\tau_{\eta,\nu})=\Hom_{G'}(C^\infty(G/P,\mathcal{V}_{\xi,\lambda})),C^\infty(G'/P',\mathcal{W}_{\eta,\nu})).$$
The Schwartz Kernel Theorem implies that every such operator is given by a $G'$-invariant distribution section of the tensor bundle $\mathcal{V}_{\xi^*,-\lambda}\boxtimes\mathcal{W}_{\eta,\nu}$ over $G/P\times G'/P'$, where $\xi^*$ is the representation contradigent to $\xi$. Since $G'$ acts transitively on $G'/P'$ we can consider these distributions as sections on $G/P$ with a certain $P'$-invariance:

\begin{theorem}[{\cite[Proposition 3.2]{kobayashi_speh_2015}}]
\label{theorem:KS_kernel}
There is a natural bijection
$$\Hom_{G'}(\pi_{\xi,\lambda}|_{G'},\tau_{\eta,\nu}) \stackrel{\sim}{\longrightarrow} (\mathcal{D}'(G/P,\mathcal{V}_{\xi^*,-\lambda})\otimes W_{\eta,\nu+\rho'})^{P'}, \quad T\mapsto u^T
.$$
\end{theorem}

\subsection{Restriction to the open Bruhat cell}\label{sec:DistributionKernelsOnOpenBruhatCell}

From now on assume $M'=M\cap G'$, $A'=A\cap G'$ and $N'=N\cap G'$. Since $\Nbar$ is unipotent we obtain a parameterization of the open Bruhat cell $\Nbar P/P \subseteq G/P$ in terms of the Lie algebra $\nbar$ by the map
$$\nbar\stackrel{\exp}{\longrightarrow}  \Nbar \lhook \joinrel \longrightarrow G \longrightarrow G/P
,$$ 
so that we can consider $\nbar$ as an open dense subset of $G/P$. Then the restriction
$$\mathcal{D}'(G/P,\mathcal{V}_{\xi^*,-\lambda})\longrightarrow \mathcal{D}'(\nbar,\mathcal{V}_{\xi^*,-\lambda}|_{\nbar})$$ can be used to define a $\g$-action on $\mathcal{D}'(\nbar,\mathcal{V}_{\xi^*,-\lambda}|_{\nbar})\cong\mathcal{D}'(\nbar)\otimes V_{\xi^*,-\lambda+\rho}$ by vector fields. Moreover, since ${\rm Ad}(M'A')$ leaves $\nbar$ invariant, the restriction is further $M'A'$-equivariant.
If we assume $P'\Nbar P=G$, i.e. every $P'$-orbit in $G/P$ meets the open Bruhat cell $\Nbar P$, then symmetry breaking operators can be described in terms of $(M'A',\mathfrak{n}')$-invariant distributions on $\nbar$:

\begin{theorem}[{\cite[Theorem 3.16]{kobayashi_speh_2015}}]
\label{theorem:KS_kernel_open_cell}
	Assume $P'\Nbar P=G$, then there is a natural bijection
	$$\Hom_{G'}(\pi_{\xi,\lambda}|_{G'},\tau_{\eta,\nu}) \stackrel{\sim}{\longrightarrow} (\mathcal{D}'(\nbar)\otimes V_{\xi^*,-\lambda+\rho}\otimes W_{\eta, \nu+\rho'})^{M'A',\mathfrak{n}'}.$$
\end{theorem}

Given a distribution kernel $u^T$, the corresponding operator $T \in \Hom_{G'}(\pi_{\xi,\lambda}|_{G'},\tau_{\eta,\nu})$ is given by
\begin{equation}
\label{eq:distribution_operator}
T \varphi(h)=\langle u^T, \varphi(h\exp(\;\cdot\;)) \rangle.
\end{equation}

In this paper we will only be concerned with the case of spherical principal series representations. We therefore let $\xi={\bf 1}$ and $\eta={\bf 1}$ be the trivial representations and put $\pi_\lambda=\pi_{{\bf 1},\lambda}$ and $\tau_\nu=\tau_{{\bf 1},\nu}$.

For the classification of symmetry breaking operators it will be convenient to also consider invariant distributions on open subsets of $\nbar$. For an open $M'A'$-invariant subset $\Omega \subseteq \nbar$ we therefore use the notation
$$ \mathcal{D}'(\Omega)_{\lambda,\nu}:=(\mathcal{D}'(\Omega)\otimes V_{{\bf 1},-\lambda+\rho}\otimes W_{{\bf 1}, \nu+\rho'})^{M'A',\mathfrak{n}'}, $$
so that $\Hom(\pi_\lambda|_{G'},\tau_\nu)\cong\mathcal{D}'(\nbar)_{\lambda,\nu}$.

\begin{remark}
\label{remark:diff_equations}
The space $\mathcal{D}'(\Omega)_{\lambda,\nu}$ is given by all distributions $u\in\mathcal{D}'(\Omega)$ satisfying
\begin{align}\label{eq:diff_equation_gerneral}
 u({\rm Ad}(m'a')X) &= a'^{-\lambda+\rho+\nu+\rho'}u(X) && \forall\,m'\in M',a'\in A',\\
 d\pi_{-\lambda}(Y)u(X) &= 0 && \forall\,Y\in\nfrak'.\label{eq:diff_equation_gerneral_z}
\end{align}
\end{remark}

\section{Principal series representations of rank one reductive groups}

For $\F\in\{\C,\mathbb{H},\mathbb{O}\}$ let $G= \Urm(1,n+1;\F)$ denote the group of $(n+2)\times(n+2)$ matrices over $\F$ preserving the quadratic form
$$(z_0,z_1,\ldots,z_{n+1})\mapsto-|z_0|^2+|z_1|^2+\cdots+|z_{n+1}|^2. $$
Here we assume that $n\geq1$ for $\F=\C,\mathbb{H}$ and $n=1$ for $\F=\mathbb{O}$, using the interpretation $\Urm(1,2;\mathbb{O})\cong F_{4(-20)}$.

Let $P$ be the minimal parabolic subgroup of $G$ with Langlands decomposition $P=MAN$ given by
\begin{align*}
 M &= \left\{ \begin{pmatrix}a & & \\& a & \\& & b\end{pmatrix}: a \in \Urm(1;\F), b \in \Urm(n;\F)\right\},\\
 A &= \exp (\mathfrak{a}) \qquad \text{where } \mathfrak{a}=\R H, \quad H=\begin{pmatrix}0 & 1 & \\1 & 0 & \\& & \mathbf{0}_n\end{pmatrix},\\
 N &= \exp (\mathfrak{n}) \qquad \text{where } \mathfrak{n}=\left\{ \begin{pmatrix}Z & -Z & X \\Z & -Z & X \\X^* & -X^* & \mathbf{0}_n\end{pmatrix}: X\in \F^n, Z\in \Im \F \right\}.
\end{align*}
Here, for $X=(X_1,\ldots ,X_n) \in \F^n$ we write $X^*:= (\overline{X}_1, \ldots \overline{X}_n)^T$. Further, we use the notation $\Im\F=\{Z\in\F:\overline{Z}=-Z\}$. Note that $\mathfrak{n}$ is the direct sum of the eigenspaces of ${\rm ad}(H)$ to the eigenvalues $+1$ and $+2$. We abbreviate the real dimensions of these eigenspaces by
$$ p=n\cdot\dim_\R \F \quad \mbox{and} \quad q= \dim_\R \F -1. $$
We identify $\mathfrak{a}_\C^*\cong\C$ by $\lambda\mapsto\lambda(H)$. Then in particular
$$ \rho= \frac{1}{2} \tr \operatorname{ad}|_\mathfrak{n}(H)= \frac{p+2q}{2}. $$

For $\lambda\in\C$ we consider the spherical principal series representation $\pi_\lambda$ in the notation of Section~\ref{sec:PrincipalSeries}. Due to Johnson and Wallach the composition series of these representations are well known.

\begin{theorem}[{\cite{Johnson_Wallach_1977,Johnson_1976}}]\label{theorem:composition_series}
$\pi_\lambda$ is irreducible if and only if $\lambda \notin \pm(\rho-q+1+2\Z_{\geq 0})$. $\pi_\lambda$ contains an irreducible finite-dimensional submodule if and only if $\lambda \in -\rho-2\Z_{\geq 0}$.
\end{theorem}

\subsection{The non-compact picture}

Let $\Nbar$ be the nilradical of the parabolic subgroup opposite to $P$. Since $\Nbar$ is unipotent, we identify it with its Lie algebra $\nbar\cong \F^n \oplus \Im \F $ in terms of the exponential map:
$$
\F^n\oplus\Im\F\to\Nbar, \quad (X,Z)\mapsto\bar{n}_{(X,Z)}:=\exp \begin{pmatrix}
\frac{Z}{2} & \frac{Z}{2} & X \\
-\frac{Z}{2} & -\frac{Z}{2} & -X \\
X^* & X^* & \mathbf{0}_n
\end{pmatrix}.
$$

Since $\Nbar P$ is open and dense in $G$, the restriction of $\pi_\lambda$ to functions on $\Nbar$ is one-to-one. The resulting realization $I_\lambda\subseteq C^\infty(\nbar)$ of $\pi_\lambda$ is called the \emph{non-compact picture of $\pi_\lambda$}. For $g \in \Nbar MAN$ we write $g=\bar{n}(g) m(g) a(g) n(g)$ for the obvious decomposition. Then the $G$-action in the non-compact picture is given by
\begin{equation}
\label{eq:action_non_compact_picture}
\pi_\lambda(g)f(X,Z)=a(g^{-1} \bar{n}_{(X,Z)})^{-(\lambda+\rho)}f(\log\bar{n}(g^{-1}\bar{n}_{(X,Z)})),
\end{equation}
whenever $g^{-1}\overline{n}_{(X,Z)}\in\overline{N}MAN$.

The Lie bracket of $\nbar\cong\F^n\oplus\Im\F$ is given by
$$[(X,Z),(Y,W)]=(0,4\Im(XY^*))$$
and the group multiplication is given by
$$(X,Z)\cdot (Y,W)=(X,Z)+(Y,W)+\frac{1}{2}[(X,Z),(Y,W)] = (X+Y,Z+W+2 \Im(XY^*)).$$
We write $\nbar=\mathfrak{v}\oplus\zfrak$ with $\mathfrak{v}:=\F^n$ and $\zfrak:=\Im\F$, noting that $\zfrak$ is the center of $\nbar$. We endow $\mathfrak{v}$ and $\zfrak$ with the usual inner product given by
$$\langle X,Y \rangle = \Re(XY^*)$$
and extend it to $\nbar$ so that $\mathfrak{v}$ and $\zfrak$ are orthogonal. Then for all $Z \in \zfrak$ we obtain a linear skew-symmetric map $J_Z\in{\rm End}(\mathfrak{v})$ characterized by
$$ \langle J_Z X , Y \rangle = \langle Z,[X,Y]\rangle. $$
We note that $J_ZX=-4Z\cdot X$, the scalar multiplication of $\F$ on $\F^n$. In particular we have $J^2_Z=-16\abs{Z}^2\cdot{\rm id}_{\mathfrak{v}}$.

\begin{remark}
\begin{enumerate}[label=(\roman{*}), ref=\thetheorem(\roman{*})]
\item
Two-step nilpotent Lie algebras and their corresponding groups which satisfy the condition $J^2_Z=-16\abs{Z}^2\cdot{\rm id}_{\mathfrak{v}}$ are said to be of \emph{H-type} and many constructions in this paper could be carried out in the more general setting of H-type groups.
\item
Those H-type groups which occur as unipotent radicals of semisimple Lie groups of rank one are exactly the H-type groups satisfying the $J^2$-condition, namely for all $Z,Z' \in \zfrak$ with $\langle Z,Z'\rangle =0$ there exists $Z'' \in \zfrak$ such that $J_ZJ_{Z'}=J_{Z''}$ (see \cite{cowling_1998}).
\end{enumerate}
\end{remark}

We collect some facts that follow directly from the definitions above:

\begin{lemma}
\label{lemma:H-type_rules}
Let $X,X' \in \mathfrak{v}$ and $Z,Z' \in \zfrak$.
\begin{enumerate}[label=(\roman{*}), ref=\thetheorem(\roman{*})]
\item 
\label{lemma:H-type_rules:i}
$\langle J_Z X,X \rangle =0$,
\item 
\label{lemma:H-type_rules:ii}
$\langle J_Z X, X'   \rangle= -\langle X, J_Z X' \rangle$,
\item 
\label{lemma:H-type_rules:iii}
$\abs{J_Z X}^2=16 \abs{Z}^2\abs{X}^2$,
\item 
\label{lemma:H-type_rules:iv}
$J_Z J_{Z'}+J_{Z'} J_Z=-32 \langle Z , Z'\rangle\cdot {\rm id}_{\mathfrak{v}}$.
\end{enumerate}
\end{lemma}

We now compute the representation $\pi_\lambda$ on $\overline{P}=MA\Nbar$ and on the representative $\tilde{w}_0=\diag(-1,1,\mathbf{1}_n)$ of the longest Weyl group element of $G$ with respect to $A$. For this we first state some matrix decompositions which are easily verified.

For $(X,Z)\in\nbar$ let
$$ N(X,Z):=(\abs{X}^4+\abs{Z}^2)^{\frac{1}{4}} $$
denote the \emph{norm function} of $\nbar$.

\begin{lemma}
\label{prop:matrix_decompositions}
\begin{enumerate}[label=(\roman{*}), ref=\thetheorem(\roman{*})]
\item\label{prop:matrix_decompositions:iii} Let $m=\diag(a,a,b^{-1}) \in M$ with $a \in \Urm(1;\F)$ and $b\in \Urm(n;\F)$, then
$$ m \bar{n}_{(X,Z)}m^{-1}=\bar{n}_{(aXb,aZa^{-1})}. $$
\item\label{prop:matrix_decompositions:ii} Let $t\in \R$ and $a=\exp(tH)$, then
$$ a\bar{n}_{(X,Z)}a^{-1}=\bar{n}_{(e^{-t}X,e^{-2t}Z)}. $$
\item\label{prop:matrix_decompositions:i} Let $(X,Z)\neq (0,0)$, then  $\tilde{w}_0 \bar{n}_{(X,Z)}=\bar{n}_{(U,V)}m a n$ with $m\in M$, $n\in N$ and
\begin{equation*}
 U=\frac{(\frac{1}{4}J_Z-\abs{X}^2)X}{N(X,Z)^4}, \qquad V=\frac{-Z}{N(X,Z)^4}, \qquad a = \exp(2\log(N(X,Z))H).
\end{equation*}
\end{enumerate}
\end{lemma}

These decompositions immediately imply the following formulas for the action of $P$ and $\tilde{w}_0$:

\begin{prop}\label{prop:GroupAction}
\begin{enumerate}[label=(\roman{*}), ref=\thetheorem(\roman{*})]
\item\label{prop:GroupActionM} For $m=\diag(a^{-1},a^{-1},b)\in M$ with $a\in\Urm(1;\F)$ and $b\in\Urm(n;\F)$:
$$ \pi_\lambda(m)u(X,Z) = u(aXb,aZa^{-1}). $$
\item For $t\in\R$ and $a=\exp(tH)$:
$$ \pi_\lambda(a)u(X,Z) = e^{(\lambda+\rho)t}u(e^tX,e^{2t}Z). $$
\item\label{prop:GroupActionNbar} For $(S,T)\in\nbar$:
$$ \pi_\lambda(\bar{n}_{(S,T)})u(X,Z) = u(X-S,Z-T-\tfrac{1}{2}[S,X]). $$
\item
\label{prop:longest_weyl_group_element_action}
For the action of $\tilde{w}_0$ we have
\begin{equation*}
	\pi_{\lambda}(\tilde{w}_0)u(X,Z)=N(X,Z)^{-2(\lambda+\rho)}u(\sigma(X,Z)).
\end{equation*}
where $\sigma: \nbar \setminus \{0 \} \to \nbar \setminus \{0 \}$ is the inversion given by
$$ \sigma(X,Z)=\left(\frac{(\frac{1}{4}J_Z-\abs{X}^2)X}{N(X,Z)^4}, \frac{-Z}{N(X,Z)^4} \right). $$
\end{enumerate}
\end{prop}

Note that $X\in\F^n$ is a row vector, so that matrix multiplication is from the right.

\begin{remark}
The map $\sigma$ is the inversion defining the Kelvin-transform for $H$-type groups (see \cite[Chapter 4]{cowling_1991}).
\end{remark}

We now use Proposition~\ref{prop:GroupAction} to compute derived representation $d\pi_\lambda$ of $\mathfrak{g}$ on $I_\lambda$. To state the action of the Lie algebra $\mathfrak{g}$ let $E_\mathfrak{v}$ denote the Euler operator on $\mathfrak{v}$ and let $E_\zfrak$ be the Euler operator on $\zfrak$. Then
$$ E:=E_\mathfrak{v}+2E_\mathfrak{z} $$
defines the \emph{weighted Euler operator} on $\nbar$ which takes into account homogeneity with respect to the action of $A$ (see Lemma~\ref{prop:matrix_decompositions:ii}). For instance, the function $N(X,Z)$ is homogeneous of degree one with respect to the weighted Euler operator, i.e. $EN(X,Z)=N(X,Z)$.

\begin{prop}\label{prop:LieAlgebraAction}
\begin{enumerate}[label=(\roman{*}), ref=\thetheorem(\roman{*})]
\item The element $H\in\mathfrak{a}$ acts by
$$ d\pi_\lambda(H) = E+\lambda+\rho. $$
\item For $S\in\mathfrak{v}$:
\begin{multline*}
 d\pi_\lambda({\rm Ad}(\tilde{w}_0S)) = -2\langle S,X\rangle(E+\lambda+\rho)+\abs{X}^2\partial_S+\frac{1}{4}\partial_{J_ZS}-\frac{1}{2}\abs{X}^2\partial_{[S,X]}\\
 +\frac{1}{8}\partial_{[S,J_ZX]}-\frac{1}{8}\partial_{J_{[S,X]}X}.
\end{multline*}
\item For $T\in\mathfrak{z}$:
\begin{multline*}
 d\pi_\lambda({\rm Ad}(\tilde{w}_0T)) = -\langle T,Z\rangle(E+E_{\mathfrak{v}}+\lambda+\rho)+N(X,Z)^4\partial_T+\frac{1}{4}\abs{X}^2\partial_{J_TX}\\
 -\frac{1}{16}\partial_{J_TJ_ZX}.
\end{multline*}
\end{enumerate}
\end{prop}

\begin{proof}
(i) follows immediately from Lemma~\ref{prop:matrix_decompositions:ii} by differentiation. 

Ad (ii): Let $S \in \mathfrak{v}$. By Proposition~\ref{prop:GroupActionNbar} the derived representation acts by the right-invariant vector fields $d\pi_\lambda(S)=-\partial_S-\frac{1}{2}\partial_{[S,X]}$. Conjugation with $\pi_\lambda(\tilde{w}_0)=\pi_\lambda(\tilde{w}_0^{-1})$ gives, using Proposition~\ref{prop:longest_weyl_group_element_action}:
\begin{align}
\begin{split}
 & \pi_{\lambda}(\tilde{w}_0)d\pi_{\lambda}(-S)\pi_{\lambda}(\tilde{w}_0^{-1})u(X,Z)=\Bigg(\langle S,X \rangle 2(E+E_\mathfrak{v}+\lambda+\rho)-\partial_{ (\abs{X}^2+\frac{1}{4}J_Z)S } \\
 & \hspace{3.5cm} + \frac{1}{2} \partial_{[S,(\abs{X}^2-\frac{1}{4}J_Z)X]}-\frac{2\left\langle (\abs{X}^2-\frac{1}{4}J_Z)X,S \right\rangle}{N(X,Z)^4}\partial_{(\abs{X}^2-\frac{1}{4}J_Z)X} \\
 & \hspace{3.5cm} + \frac{1}{8N(X,Z)^4}  \partial_{J_{[S,(\abs{X}^2-\frac{1}{4}J_Z)X]}(\abs{X}^2-\frac{1}{4}J_Z)X} \Bigg)u(X,Z).
\end{split}\label{eq:differential_equation_non_central_0}
\end{align}
Note that for $U,S \in \F^n$ we have
$$\langle S,U \rangle U- \frac{1}{16}J_{[S,U]}U=(\Re(SU^*)+\Im(SU^*))U=SU^*U, $$
which is for $U=(\abs{X}^2-\frac{1}{4}J_ZX)X$ equal to
\begin{multline*}
 SX^*\left(\abs{X}^2+\frac{1}{4}J_Z\right)\left(\abs{X}^2-\frac{1}{4}J_Z\right)X=N(X,Z)^4SX^*X\\
 =N(X,Z)^4\left(\langle S,X \rangle - \frac{1}{16}J_{[S,X]}X \right)X,
\end{multline*}
so that \eqref{eq:differential_equation_non_central_0} is
\begin{equation*}
\left(2\langle S,X \rangle(\rho+\lambda+ E)-\abs{X}^2\partial_{S} -\frac{1}{4}\partial_{J_ZS }+ \frac{1}{2} \abs{X}^2\partial_{[S,X]}-\frac{1}{8}\partial_{[S,J_ZX]}+ \frac{1}{8}\partial_{J_{[S,X]}X}\right)
\end{equation*}
applied to $u(X,Z)$.

Ad (iii): As in the case before, for $T \in \zfrak$ the derived representation acts by the right-invariant vector field $d\pi_{\lambda}(T)= -\partial_T$. Then
\begin{multline*}
\pi(\tilde{w}_0)d\pi(-T)\pi(\tilde{w}_0)u(X,Z)\\=\left(\langle T,Z \rangle(E+E_\mathfrak{v}+\lambda+\rho)- \frac{1}{4}\partial_{  J_T(\abs{X}^2-\frac{1}{4}J_Z)X  }  -N(X,Z)^4\partial_T\right)u(X,Z).\qedhere
\end{multline*}
\end{proof}

\subsection{Strongly spherical reductive pairs}

In view of Facts~\ref{fact:FiniteMultiplicities} and \ref{fact:StronglySphericalRankOnePairs} we consider strongly spherical real reductive pairs of the form
$$ (G,G') = (\Urm(1,n+1;\F),\Urm(1,m+1;\F)\times F), $$
where $F<\Urm(n-m;\F)$. To avoid arbitrarily large finite groups, we assume throughout the paper that $F$ is connected. The results for non-connected $F$ can easily be deduced from our classification using classical invariant theory for the component group $F/F_0$. We further assume $0\leq m<n$, excluding the case $m=n$ which leads to the classical theory of intertwining operators between principal series for the group $G$. Note that for $\F = \mathbb{O}$ we have $n=1$ which implies $m=0$ and $G'={\rm Spin}_0(1,8)$.

The intersection $P':= G'\cap P$ is a minimal parabolic subgroup of $G'$ with Langlands decomposition $M'A'N'$ given by $M'=G' \cap M$, $A'=A$ and $N'=G' \cap N$. The opposite parabolic $\overline{P}'=M'A'\Nbar'$ has unipotent radical $\Nbar'$ with Lie algebra $\nbar'\cong \F^m \oplus \Im \F$. Again we write $\nbar'=\mathfrak{v}' +\zfrak$ with $\mathfrak{v}'=\F^m$ and $\zfrak=\Im\F$ and we denote by $\mathfrak{v}''=\F^{n-m}$ the orthogonal complement of $\nbar'$ in $\nbar$, or equivalently $\mathfrak{v}'$ in $\mathfrak{v}$. Then
$$ p' = \dim_\R\mathfrak{v}' = m\cdot\dim_\R \F, \qquad \mbox{and} \qquad p'' = \dim_\R\mathfrak{v}'' = (n-m)\cdot\dim_\R \F. $$
Note that for $m=0$ the Lie algebra $\nbar'=\zfrak$ is abelian and hence $\mathfrak{g}'\cong\mathfrak{so}(1,q+1)$.

\begin{remark}\label{rem:TransUnitSphere}
The classification in \cite{KM14} and \cite{KKPS17} (see also \cite[Theorem 1.6]{Moe17} for a complete table) shows that the pair $(G,G')$ is strongly spherical if and only if $M'A'=\Urm(1;\F)\times\Urm(m;\F)\times F$ has an open orbit on the unit sphere in $\mathfrak{v}''=\F^{n-m}$. We remark that in this case $M'$ always acts transitively on the unit sphere in $\mathfrak{v}''$. In fact, $\mathfrak{v}''=\F^{n-m}$ is of real dimension $\geq2$ and therefore its unit sphere is connected. Hence, an open orbit of the compact group $M'$ on the unit sphere has to be the entire sphere. Moreover, the action of $\Urm(m;\F)$ on $\F^{n-m}$ is trivial, so the group $\Urm(1;\F)\times F$ acts transitively on the unit sphere in $\F^{n-m}$.
\end{remark}

In this paper we classify symmetry breaking operators between the spherical principal series representations $\pi_\lambda$ of $G$ and $\tau_\nu$ of $G'$.

\subsection{Orbit structure of $G/P$}

The $P'$-orbits in $G/P$ will be important in the following since Theorem~\ref{theorem:KS_kernel} reduces the classification of symmetry breaking operators to the classification of $P'$-invariant distributions on $G/P$. Note that we have $N'=\tilde{w}_0\Nbar'\tilde{w}_0^{-1}$, so that $N'$ acts on $\tilde{w}_0\Nbar'$ by left multiplication.

\begin{prop}
\label{cor:orbits}
The $P'$-orbits in $G/P$ and their closure relations are
$$ \mathcal{O}_A \stackrel{p''}{\rule[.5ex]{3em}{.4pt}} \mathcal{O}_B \stackrel{p'+q}{\rule[.5ex]{3em}{.4pt}} \mathcal{O}_C,$$
where
\begin{align*}
 \mathcal{O}_A &= P'\cdot\tilde{w}_0\bar{n} P = \tilde{w}_0(\Nbar\setminus\Nbar')P,\\
 \mathcal{O}_B &= P'\cdot\tilde{w}_0P = \tilde{w}_0\Nbar'P,\\
 \mathcal{O}_C &= P'\cdot\mathbf{1}_{n+2}P,
\end{align*}
for some $\bar{n}\in \Nbar \setminus \Nbar'$. Here $X\stackrel{k}{\rule[.5ex]{2em}{.4pt}}Y$ means that $Y$ is a subvariety of $\bar{X}$ of co-dimension $k$.
\end{prop}

For the proof of Proposition~\ref{cor:orbits} we write $X=(X',X'')\in \F^{m}\oplus \F^{n-m}$ instead of $X$ whenever convenient. 

\begin{proof}[Proof of Proposition~\ref{cor:orbits}]
First note that by the Bruhat decomposition $G=\tilde{w}_0\Nbar P \sqcup P$. The closed Bruhat cell $P$ clearly becomes the closed orbit $\mathcal{O}_C$ in the quotient $G/P$. For the open Bruhat cell $\tilde{w}_0\Nbar P$ we note that since $P'\tilde{w}_0=\tilde{w}_0\overline{P'}$ it suffices to describe the $\overline{P}'$-orbits in $\Nbar P/P$. The nilpotent group $\Nbar$ decomposes as $\Nbar=\Nbar'\exp(\mathfrak{v}'')$. If we write $\overline{u}=\overline{u}'\exp(X)\in\Nbar$ with $\overline{u}'\in\Nbar'$ and $X\in\mathfrak{v}''$, then
$$ \overline{n}'a'm'\cdot\big[\overline{u}'\exp(X)P\big] = (\overline{n}'\overline{u}'^{(m'a')})\exp({\rm Ad}(m'a')X)P, $$
where we use the notation $\overline{u}'^{(m'a')}=(m'a')\overline{u}'(m'a')^{-1}\in\Nbar'$. This shows that the $\overline{P}'$-orbits in $\Nbar P/P$ are of the form $\Nbar'\exp({\rm Ad}(M'A')X)P/P$ with $X\in\mathfrak{v}''$ and therefore in one-to-one correspondence with the ${\rm Ad}(M'A')$-orbits in $\mathfrak{v}''$. By Remark~\ref{rem:TransUnitSphere} the group ${\rm Ad}(M')$ acts transitively on the unit sphere in $\mathfrak{v}''$ and by Lemma~\ref{prop:matrix_decompositions:ii} the group ${\rm Ad}(A')$ acts on $\mathfrak{v}''$ by dilations, so the ${\rm Ad}(M'A')$-orbits in $\mathfrak{v}''$ are $\mathfrak{v}''\setminus\{0\}$ and $\{0\}$. Therefore, the $P'$-orbits in the open Bruhat cell $\tilde{w}_0\Nbar P$ are $\mathcal{O}_A=\tilde{w}_0\Nbar'\exp(\mathfrak{v}''\setminus\{0\})P=\tilde{w}_0(\Nbar\setminus\Nbar')P$ and $\mathcal{O}_B=\tilde{w}_0\Nbar'\exp(\{0\})P=\tilde{w}_0\Nbar'P$. Finally, the codimensions are easily determined.
\end{proof}

The orbit structure of $G/P$ implies the following:

\begin{corollary}
\label{cor:P'NbarP=G}
	$P'\Nbar P=G$.
	More precisely,
	$$ \mathcal{O}_A\cap\Nbar = \Nbar\setminus\Nbar', \qquad \mathcal{O}_B\cap\Nbar = \Nbar'\setminus\{{\bf 1}\}, \qquad \mathcal{O}_C\cap\Nbar = \{{\bf 1}\}. $$
\end{corollary}

\begin{proof}
It is clear that $\mathcal{O}_C\cap\Nbar=P\cap\Nbar=\{{\bf 1}\}$. The remaining two identities are easily verified if one observes that by Lemma~\ref{prop:matrix_decompositions:i} the map $xP\mapsto\tilde{w}_0xP$ maps $(\Nbar\setminus\{{\bf 1}\})P$ and $(\Nbar'\setminus\{{\bf 1}\})P$ to itself and it further maps $\tilde{w}_0P$ to ${\bf 1}P$ and vice versa.
\end{proof}

Theorem~\ref{theorem:KS_kernel_open_cell} and Corollary~\ref{cor:P'NbarP=G} allow us to reduce the classification of symmetry breaking operators to the classification of certain invariant distributions on the Lie algebra $\nbar$.

\section{Invariant distribution kernels}\label{sec:differential_equations}

In this section we give a characterization of the invariant distribution kernels on $\nbar$ describing symmetry breaking operators in terms of a set of differential equations and describe the strategy that is used in the remaining part of the paper to find all solutions to these equations.

\subsection{Differential equations satisfied by symmetry breaking operators}

With Remark~\ref{remark:diff_equations} as well as Proposition~\ref{prop:longest_weyl_group_element_action} and Proposition~\ref{prop:LieAlgebraAction} we immediately obtain the following description of the space $\mathcal{D}'(\nbar)_{\lambda,\nu}$ of distribution kernels of symmetry breaking operators:

\begin{lemma}
\label{lemma:diff_equations}
The space $\mathcal{D}'(\nbar)_{\lambda,\nu}$ is given by all $u \in \mathcal{D}'(\nbar)$ satisfying:
\begin{enumerate}[label=(\roman{*}), ref=\thetheorem(\roman{*})]
\item 
\label{lemma:diff_equations:ii}
$u(aX'b,aX''c,aZa^{-1})=u(X',X'',Z)$ for all $a \in \Urm(1;\F)$, $b\in \Urm(m;\F)$ and $c \in F$,
\item 
\label{lemma:diff_equations:i}
$(E-\lambda+\rho+\nu+\rho')u(X,Z)=0$,
\item 
\label{lemma:diff_equations:iii}
$D_{\mathfrak{v}}(S)u(X,Z)=0$ for all $S \in \mathfrak{v}'$,
\item 
\label{lemma:diff_equations:iv}
$D_\zfrak(T)u(X,Z)=0 $ for all $T \in \zfrak$,
\end{enumerate}
where for $S\in\mathfrak{v}$ and $T\in\mathfrak{z}$ we put
\begin{align*}
 D_{\mathfrak{v}}(S) &= 2\langle S,X \rangle (\nu+\rho')+\abs{X}^2\partial_S-\frac{1}{2}\abs{X}^2\partial_{[S,X]}+\frac{1}{4}\partial_{J_ZS}+\frac{1}{8}\partial_{[S,J_ZX]}-\frac{1}{8}\partial_{J_{[S,X]}X},\\
 D_\zfrak(T) &= \langle T,Z \rangle (\nu+\rho'-E_\mathfrak{v}) + N(X,Z)^4 \partial_T +\frac{1}{4}\abs{X}^2\partial_{J_TX}- \frac{1}{16}\partial_{J_TJ_ZX}.
\end{align*}
\end{lemma}
 
\begin{remark}
In case $m>0$ the subspace $\mathfrak{v}'$ generates $\mathfrak{n}'$, so that the invariance property in Lemma~\ref{lemma:diff_equations:iv} follows from the one in Lemma~\ref{lemma:diff_equations:iii}. For $m=0$ the property in Lemma~\ref{lemma:diff_equations:iii} is trivial since $\mathfrak{v'}=\{0\}$.
\end{remark}

\subsection{The classification strategy}\label{sec:ClassificationStrategy}

Let $\Omega \subseteq \nbar$ be open and $M'$-invariant.
For a closed subset $\Omega'\subseteq \Omega$ we write $\mathcal{D}'_{\Omega'}(\Omega)_{\lambda,\nu} \subseteq \mathcal{D}'(\Omega)_{\lambda,\nu} $ for the subspace of distributions with support contained in $\Omega'$. In particular if $\Omega'=\nbar \setminus \Omega$
the following sequence is exact:
	$$
	\begin{tikzcd}
	0 \arrow[r] & \mathcal{D}'_{\Omega'}(\nbar)_{\lambda,\nu} \arrow[r] & \mathcal{D}'(\nbar)_{\lambda,\nu} \arrow[r]     &\mathcal{D}'(\nbar \setminus \Omega')_{\lambda,\nu}
	\end{tikzcd}
	$$
For $\Omega'=\{0\}$ the space $\mathcal{D}'_{\{0\}}(\nbar)_{\lambda,\nu}$ consists of those distribution kernels which define differential symmetry breaking operators. The strategy to determine $\mathcal{D}'(\nbar)_{\lambda,\nu}$ will be to classify $\mathcal{D}'_{\{0\}}(\nbar)_{\lambda,\nu}$ and
$\mathcal{D}'(\nbar\setminus \{0\})_{\lambda,\nu}$ separately and finally to study the restriction map $\mathcal{D}'(\nbar)_{\lambda,\nu} \to \mathcal{D}'(\nbar\setminus \{0\})_{\lambda,\nu}$.
This strategy was proposed by Kobayashi--Speh in \cite{kobayashi_speh_2015} and successfully applied in the case $\F=\R$.
Since the classification of differential symmetry breaking operators is a question of invariant theory and combinatorics, while the classification of of symmetry breaking operators outside the origin and the continuation of such operators into the origin is a question of distribution theory, we divide the classification into two major parts: the classification of differential symmetry breaking operators (Part~\ref{part:diff}) and the study of meromorphic families of symmetry breaking operators which eventually leads to the full classification (Part~\ref{part:classification}).

\newpage

\part{Differential symmetry breaking operators}\label{part:diff}

The main result of this part is the full classification of differential symmetry breaking operators between principal series representations for strongly spherical pairs of the form $(G,G')=(\Urm(1,n+1;\F),\Urm(1,m+1;\F)\times F)$ with $0\leq m<n$ and $F<\Urm(n-m;\F)$. For this we use a Euclidean Fourier transform on $\nbar$ which gives an equivalent formulation of the classification problem in terms of polynomial solutions to a certain system of differential equations. This approach is essentially the F-method proposed by Kobayashi~\cite{Kob13} which was previously applied in several situations where the nilpotent radical is abelian (see e.g. \cite{KOSS15,KKP16,KP16}).

\section{The Fourier transformed picture}

We use the following normalization of the Fourier transform:
$$ \mathcal{F}:\mathcal{D}'_{\{0\}}(\nbar) \to \C[\nbar], \quad \mathcal{F}u(X,Z) = \int_{\nbar} e^{-\langle X,X'\rangle-\langle Z,Z'\rangle}u(X',Z')\,d(X',Z'). $$
For a differential operator $D$ on $\nbar$ with polynomial coefficients there exists a unique differential operator $\mathcal{F}(D)$ with polynomial coefficients on the same space such that
$$ \mathcal{F}(Du) = \mathcal{F}(D)\mathcal{F}(u) \qquad \mbox{for $u\in\mathcal{D}'_{\{0\}}(\nbar)$.} $$
This map has the following properties:
$$ \mathcal{F}(\langle X,S\rangle)=-\partial_S, \qquad \mathcal{F}(\partial_S)=\langle X,S\rangle $$
and similar for $Z$. We define
$$ \C[\nbar]_{\lambda,\nu}:= \mathcal{F}(\mathcal{D}'_{\{0\}}(\nbar)_{\lambda,\nu}). $$

To characterize the space $\C[\nbar]_{\lambda,\nu}$ we have to transform the differential equations in Lemma~\ref{lemma:diff_equations}. For this let $S_1,\dots,S_p$ be an orthonormal basis of $\mathfrak{v}$ such that $S_1, \ldots S_{p'}$ is an orthonormal basis of $\mathfrak{v}'$ and $S_{p'+1},\ldots,S_p$ an orthonormal basis of $\mathfrak{v}''$ and let $T_1,\dots T_q$ be an orthonormal basis of $\mathfrak{z}$. Then we define the Euclidean Laplacians on $\mathfrak{v}$, $\mathfrak{v}'$, $\mathfrak{v}''$ by
$$ \Delta_\mathfrak{v} := \partial_{S_1}^2+\cdots +\partial_{S_p}^2, \qquad \Delta_{\mathfrak{v}'} := \partial_{S_1}^2+\cdots +\partial_{S_{p'}}^2, \qquad \Delta_{\mathfrak{v}''}:=\Delta_\mathfrak{v}-\Delta_{\mathfrak{v}'} $$
and the Euclidean Laplacian on $\mathfrak{z}$:
$$ \square = \partial_{T_1}^2+ \cdots+ \partial_{T_q}^2. $$

\begin{prop}\label{prop:fourier_equations}
For $S\in\mathfrak{v}$ and $T\in\zfrak$ we have
\begin{align*}
 -\mathcal{F}(D_{\mathfrak{v}}(S)) &= 2(\nu+\rho'-1)\partial_S-\langle X,S\rangle\Delta_{\mathfrak{v}}-\frac{1}{2}\partial_{J_ZS}\Delta_{\mathfrak{v}}+\frac{1}{4}\partial_{[S,X]}+2P_S-2Q_S,\\
 -\mathcal{F}(D_\zfrak (T)) &= (E_{\mathfrak v}+\nu+\rho'-2)\partial_T-\langle Z,T\rangle(\Delta_{\mathfrak{v}}^2+\Box)-\frac{1}{4}\partial_{J_TX}\Delta_{\mathfrak{v}}+R_T\\
\end{align*}
with
$$ P_S = \frac{1}{16} \sum_{j=1}^q \partial_{J_{T_j}J_ZS}\partial_{T_j}, \qquad Q_S = \frac{1}{16}\sum_{i=1}^{p}\partial_{J_{[S,S_i]}X}\partial_{S_i}, \qquad R_T = \frac{1}{16}\sum_{j=1}^q \partial_{J_{T_j}J_TX}\partial_{T_j}. $$
\end{prop}

To show the proposition, we first prove some technical identities.

\begin{lemma}\label{lem:LaplacianCommutators}
For $S\in\mathfrak{v}$ and $T\in\zfrak$ we have
\begin{enumerate}[label=(\roman{*}), ref=\thetheorem(\roman{*})]
\item $[\Delta_{\mathfrak{v}},\partial_{J_TX}]=0$,
\item $[\Delta_{\mathfrak{v}},\partial_{J_ZS}]=0$,
\item $[\Box,\partial_{[J_ZS,X]}]=-32\langle S,X\rangle\Box$.
\end{enumerate}
\end{lemma}

\begin{proof}
Ad (i): Since $J_T$ is skew-symmetric we have
$$ [\Delta_{\mathfrak{v}},\partial_{J_TX}] = 2\sum_{i=0}^p\partial_{J_TS_i}\partial_{S_i} = 0. $$

Ad (ii): This is clear since $J_ZS$ is independent of $X\in\mathfrak{v}$.

Ad (iii): We have
\begin{align*}
 [\Box,\partial_{[J_ZS,X]}] &= 2\sum_{j=1}^q\partial_{[J_{T_j}S,X]}\partial_{T_j}\\
 &= 2\sum_{j,k=1}^q\langle T_k,[J_{T_j}S,X]\rangle\partial_{T_j}\partial_{T_k}\\
 &= 2\sum_{j,k=1}^q\langle J_{T_k}J_{T_j}S,X\rangle\partial_{T_j}\partial_{T_k}\\
 &= \sum_{j,k=1}^q\langle(J_{T_j}J_{T_k}+J_{T_k}J_{T_j})S,X\rangle\partial_{T_j}\partial_{T_k}\\
 &= -32\sum_{j,k=1}^q\langle T_j,T_k\rangle\langle S,X\rangle\partial_{T_j}\partial_{T_k}\\
 &= -32\langle S,X\rangle\Box.\qedhere
\end{align*}
\end{proof}

\begin{lemma}\label{lem:J_[S,Xi]}
For $X\in\mathfrak{v}$:
$$ \sum_{i=1}^pJ_{[X,S_i]}S_i = -16qX. $$
\end{lemma}

\begin{proof}
Since $\langle X,Y\rangle=\Re(XY^*)$, $[X,Y]=4\Im(XY^*)$ and $J_ZX=-4ZX$ we have
\begin{align*}
 \sum_{i=1}^{p}J_{[X,S_i]}S_i &= -16\sum_{i=1}^{p}\Im(XS_i^*)S_i\\
 &= -16\sum_{i=1}^{p}(XS_i^*)S_i+16\sum_{i=1}^{p}\langle X,S_i\rangle S_i\\ 
 &= -16X\sum_{i=1}^{p}S_i^*S_i + 16X\\
 &= -16(q+1)X + 16X\\
 &= -16qX,
\end{align*}
where we have used $\sum_{i=1}^{p}S_i^*S_i=\dim_\R\F\cdot{\bf 1}_n$.
\end{proof}

\begin{proof}[Proof of Proposition~\ref{prop:fourier_equations}]
Using Lemma~\ref{lem:LaplacianCommutators} it is easily seen that
\begin{align*}
 \mathcal{F}(\abs{X}^2\partial_S) &= \langle S,X\rangle\Delta_{\mathfrak{v}}+2\partial_S,\\
 \mathcal{F}(\abs{X}^2\partial_{[S,X]}) &= -\partial_{J_ZS}\Delta_{\mathfrak{v}},\\
 \mathcal{F}(\partial_{J_ZS}) &= -\partial_{[S,X]}.
\end{align*}
We further have
\begin{align*}
 \mathcal{F}(\partial_{[S,J_ZX]}) &= \mathcal{F}\left(\sum_{i=1}^p\sum_{j=1}^q\langle X,S_i\rangle\langle Z,T_j\rangle\partial_{[S,J_{T_j}S_i]}\right)\\
 &= \sum_{i=1}^p\sum_{j=1}^q\partial_{S_i}\partial_{T_j}\langle Z,[S,J_{T_j}S_i]\rangle\\
 &= \sum_{i=1}^p\sum_{j=1}^q\left(\langle Z,[S,J_{T_j}S_i]\rangle\partial_{T_j}+\langle T_j,[S,J_{T_j}S_i]\rangle\right)\partial_{S_i}\\
 &= -\sum_{i=1}^p\sum_{j=1}^q\left(\langle J_{T_j}J_ZS,S_i\rangle\partial_{T_j}+\langle J_{T_j}^2S,S_i\rangle\right)\partial_{S_j}\\
 &= -16P_S+16q\partial_S
\intertext{and}
 \mathcal{F}(\partial_{J_{[S,X]}X}) &= \mathcal{F}\left(\sum_{i,j=1}^p\langle X,S_i\rangle\langle X,S_j\rangle\partial_{J_{[S,S_i]}S_j}\right)\\
 &= \sum_{i,j=1}^p\partial_{S_i}\partial_{S_j}\langle X,J_{[S,S_i]}S_j\rangle\\
 &= \sum_{i,j=1}^p\partial_{S_i}\langle X,J_{[S,S_i]}S_j\rangle\partial_{S_j}\\
 &= -\sum_{i,j=1}^p\partial_{S_i}\langle J_{[S,S_i]}X,S_j\rangle\partial_{S_j}\\
 &= -\sum_{i,j=1}^p\left(\langle J_{[S,S_i]}X,S_j\rangle\partial_{S_i}+\langle J_{[S,S_i]}S_i,S_j\rangle\right)\partial_{S_j}\\
 &= -16Q_S+16q\partial_S,
\end{align*}
where Lemma~\ref{lem:J_[S,Xi]} was used in the last step. Then the first identity follows.

For the second identity a short computation shows that
\begin{align*}
 \mathcal{F}(E_{\mathfrak{v}}) &= -E_{\mathfrak{v}}-p, & \mathcal{F}(\abs{X}^4\partial_T) &= \langle Z,T\rangle\Delta_{\mathfrak{v}}^2,\\
 \mathcal{F}(\abs{X}^2\partial_{J_TX}) &= \partial_{J_TX}\Delta_{\mathfrak{v}}, & \mathcal{F}(\abs{Z}^2\partial_T) &= \langle Z,T\rangle\Box+2\partial_T.
\end{align*}
Further, we have
\begin{align*}
 \mathcal{F}(\partial_{J_TJ_ZX}) &= \mathcal{F}\left(\sum_{i=1}^p\sum_{j=1}^q\langle X,S_i\rangle\langle Z,T_j\rangle\partial_{J_TJ_{T_j}S_i}\right)\\
 &= \sum_{i=1}^p\sum_{j=1}^q\partial_{S_i}\partial_{T_j}\langle X,J_TJ_{T_j}S_i\rangle\\
 &= \sum_{i=1}^p\sum_{j=1}^q\left(\langle X,J_TJ_{T_j}S_i\rangle\partial_{S_i}+\langle S_i,J_TJ_{T_j}S_i\rangle\right)\partial_{T_j}\\
 &= \sum_{i=1}^p\sum_{j=1}^q\langle J_{T_j}J_TX,S_i\rangle\partial_{S_i}\partial_{T_j} + \frac{1}{2}\sum_{i=1}^p\sum_{j=1}^q\langle S_i,(J_TJ_{T_j}+J_{T_j}J_T)S_i\rangle\partial_{T_j}\\
 &= 16R_T - 16p\partial_T,
\end{align*}
where Lemma~\ref{lemma:H-type_rules:iv} was used in the last step. Putting these identities together shows the second formula.
\end{proof}

Now applying the Fourier transform to the differential equations of Lemma~\ref{lemma:diff_equations} we obtain:

\begin{lemma}\label{lemma:fourier_equations}
The space $\C[\nbar]_{\lambda,\nu}$ is given by all $u \in \C[\nbar]$ satisfying:
\begin{enumerate}[label=(\roman{*}), ref=\thetheorem(\roman{*})]
\item 
\label{lemma:fourier_equations:i}
$(E+\lambda+\rho-\nu-\rho')u(X,Z)=0$,
\item 
\label{lemma:fourier_equations:ii}
$u(aX'b,aX''c,aZa^{-1})=u(X',X'',Z)$ for all $a \in \Urm(1;\F)$, $b\in \Urm(m;\F)$ and $c \in F$,
\item 
\label{lemma:fourier_equations:iii}
$\left(2(\nu+\rho'+q-2)\partial_S-(X,S)\Delta_{\mathfrak{v}}-\frac{1}{2}\partial_{J_ZS}\Delta_{\mathfrak{v}}+\frac{1}{4}\partial_{[S,X]}+2P_S-2Q_S\right)u(X,Z)=0$,
 for all $S \in \mathfrak{v}'$,
\item 
\label{lemma:fourier_equations:iv}
$\left((E_{\mathfrak v}+\nu+\rho'-2)\partial_T-(Z,T)(\Delta_{\mathfrak{v}}^2+\Box)+\frac{1}{4}\partial_{J_TX}\Delta_{\mathfrak{v}}+R_T\right)u(X,Z)=0$,
 for all $T \in \zfrak$.
\end{enumerate}
\end{lemma}

\section{Polynomial Invariants}\label{sec:Invariants}

Consider the ring $\C[\nbar]^{M'}$ of polynomials which are invariant under the $M'$ action given in Lemma~\ref{lemma:fourier_equations:ii}, then $\C[\nbar]_{\lambda,\nu}\subseteq\C[\nbar]^{M'}$. The first step in the classification of $\C[\nbar]_{\lambda,\nu}$ is to find generators of $\C[\nbar]^{M'}$.

\begin{lemma}\label{lem:Transitivity}
\begin{enumerate}[label=(\roman{*}), ref=\thetheorem(\roman{*})]
\item For $\F=\C$ the group $M'$ acts transitively on $S^{p'-1}\times S^{p''-1}\times\{1\}$.
\item For $\F=\mathbb{H}$ with either $m<n-1$ or $m=n-1$ and $F={\rm Sp}(1)=\Urm(1;\mathbb{H})$, the group $M'$ acts transitively on $S^{p'-1}\times S^{p''-1}\times S^{q-1}$.
\item For $\F=\mathbb{O}$ the group $M'$ acts transitively on $S^{p'-1}\times S^{p''-1}\times S^{q-1}$.
\end{enumerate}
\end{lemma}

Here we use the convention $S^{-1}=\{0\}\subseteq\F^0$.

\begin{proof}
Recall that by Remark~\ref{rem:TransUnitSphere} the group $M'=\Urm(1;\F)\times\Urm(m;\F)\times F$ acts transitively on the unit sphere in $\mathfrak{v}''=\F^{n-m}$.
\begin{enumerate}[label=(\roman{*}), ref=\thetheorem(\roman{*})]
\item For $\F=\C$ we have $M'=\Urm(1)\times\Urm(m)\times F$ with the action on $\nbar=\C^m\oplus\C^{n-m}\oplus i\R$ given by $(a,b,c)\cdot(X',X'',Z)=(aX'b,aX''c,Z)$. Since $\Urm(m)$ acts transitively on $S^{p'-1}=S^{2m-1}$ and $\Urm(1)\times F$ acts transitively on $S^{p''-1}$ by Remark~\ref{rem:TransUnitSphere}, the claim follows.
\item For $\F=\mathbb{H}$ we have $M'={\rm Sp}(1)\times{\rm Sp}(m)\times F$ with the action on $\nbar=\mathbb{H}^m\oplus\mathbb{H}^{n-m}\oplus\Im\mathbb{H}$ given by $(a,b,c)\cdot(aX'b,aX''c,aZa^{-1})$. Assume first that $m<n-1$, then ${\rm Sp}(1)$ cannot act transitively on $S^{p''-1}$, so $F$ has to act transitively on it. Then ${\rm Sp}(m)$ acts transitively on $S^{p'-1}=S^{4m-1}$, $F$ acts transitively on $S^{p''-1}$ and ${\rm Sp}(1)$ acts transitively on $S^{q-1}=S^2$ and the claim follows. For $m=n-1$ and $F={\rm Sp}(1)=\Urm(1;\mathbb{H})$, the argument is similar.
\item For $\F=\mathbb{O}$ the group $M'={\rm Spin}(7)$ acts on $\mathfrak{v}=\mathfrak{v}''=\R^8$ by the spin representation and on $\mathfrak{z}=\R^7$ by the covering map ${\rm Spin}(7)\to{\rm SO}(7)$. Clearly the action of ${\rm SO}(7)$ on $S^6\subseteq\R^7$ is transitive. The stabilizer of a point in $S^6$ is isomorphic to ${\rm Spin}(6)\simeq{\rm SU}(4)$ which acts on $\R^8\simeq\C^4$ by the standard representation and hence transitively on its unit sphere.\qedhere
\end{enumerate}
\end{proof}

We now determine the ring $\C[\nbar]^{M'}$ of invariants. The most complicated situation occurs for $\F=\mathbb{H}$ and $m=n-1$. For $(X,Z)\in\mathfrak{v}''\oplus\mathfrak{z}=\mathbb{H}\oplus\Im\mathbb{H}$ we write
$$ p_1(X,Z) := \langle\qi,\overline{X}ZX\rangle, \qquad p_2(X,Z) := \langle\qj,\overline{X}ZX\rangle, \qquad p_3(X,Z) := \langle\qk,\overline{X}ZX\rangle. $$

\begin{lemma}\label{lemma:poly_invariants}
\begin{enumerate}[label=(\roman{*}), ref=\thetheorem(\roman{*})]
\item For $\F=\C$, the ring $\C[\nbar]^{M'}$ is generated by $\abs{X'}^2$, $\abs{X''}^2$ and $Z$.
\item For $\F=\mathbb{H}$, $m=n-1$ and $F=\{1\}$, the ring $\C[\nbar]^{M'}$ is generated by $\abs{X'}^2$, $\abs{X''}^2$, $|Z|^2$ and $p_1(X'',Z)$, $p_2(X'',Z)$, $p_3(X'',Z)$.
\item For $\F=\mathbb{H}$, $m=n-1$ and $F=\exp(\R U)\simeq\Urm(1)$, $U\in\Im\mathbb{H}=\mathfrak{sp}(1)=\mathfrak{u}(1;\mathbb{H})$, the ring $\C[\nbar]^{M'}$ is generated by $\abs{X'}^2$, $\abs{X''}^2$, $\abs{Z}^2$ and $\langle U,\overline{X''}ZX''\rangle$.
\item In all other cases $\C[\nbar]^{M'}$ is generated by $\abs{X'}^2$, $\abs{X''}^2$ and $\abs{Z}^2$.
\end{enumerate}
\end{lemma}

\begin{proof}
Statements (i) and (iv) follow immediately from Lemma~\ref{lem:Transitivity}, so we only consider the case $\F=\mathbb{H}$ with $m=n-1$. Here $M'={\rm Sp}(1)\times{\rm Sp}(n-1)\times F$ acts on $\mathbb{H}^{n-1}\oplus\mathbb{H}\oplus\Im\mathbb{H}$ by $(a,b,c)\cdot(X',X'',Z)=(aX'b,aX''c,aZa^{-1})$. Assume first that $F=\{1\}$, then clearly $\abs{X'}^2$, $\abs{X''}^2$, $|Z|^2$ and $p_1(X'',Z)$, $p_2(X'',Z)$, $p_3(X'',Z)$ are $M'$-invariant. Further, the transitivity of ${\rm Sp}(n-1)$ on the unit sphere in $\mathbb{H}^{n-1}$ implies that the only invariant polynomial in $X'$ is $\abs{X'}^2$. It remains to show that the ${\rm Sp}(1)$-invariants in $\mathbb{H}\oplus\Im\mathbb{H}$ are generated by $\abs{X''}^2$, $|Z|^2$ and $p_1(X'',Z)$, $p_2(X'',Z)$, $p_3(X'',Z)$.

The complexification of ${\rm Sp}(1)$ is $\SL(2,\C)$. The complexification of $\mathbb{H}$ is the direct sum of two copies of the $2$-dimensional complex representation $V$ of $\SL(2,\C)$. The complexification of $\Im\mathbb{H}$ is the $3$-dimensional representation of $\SL(2,\C)$. The latter representation can be viewed as the representation of $\SL(2,\C)$ on the space $F_2$ of homogeneous forms on $V$ of degree $2$. By \cite[Exercise 7 in \S4]{KP96} the ring $\C[V\oplus V\oplus F_2]^{\GL(2,\C)}$ of $\GL(2,\C)$-invariant polynomials on $V\oplus V\oplus F_2$ is generated by $\varepsilon_0$, $\varepsilon_1$ and $\varepsilon_2$ which are given by $\varepsilon_i(v,w,f)=f_i(v,w)$ with $f(sv+tw)=\sum_{i=0}^2s^it^{2-i}f_i(v,w)$. An easy computation shows that $\varepsilon_0,\varepsilon_1,\varepsilon_2$ correspond to $p_1,p_2,p_3$. Note that an element $\lambda\cdot I$ of the center $\C^\times\cdot I$ of $\GL(2,\C)$ acts on $V$ by $\lambda$ and on $F_2$ by $\lambda^{-2}$. It follows that $p_1$, $p_2$ and $p_3$ generate the ring $\C[\mathbb{H}\oplus\Im\mathbb{H}]^{{\rm Sp}(1)\cdot\R^\times}$, where $\lambda\in\R^\times$ acts on $\mathbb{H}$ by $\lambda$ and on $\Im\mathbb{H}$ by $\lambda^{-2}$.\\
Now let $f\in\C[\mathbb{H}\oplus\Im\mathbb{H}]^{{\rm Sp}(1)}$ be merely ${\rm Sp}(1)$-invariant. Since the action of ${\rm Sp}(1)$ and $\R^\times$ commute, we may assume that $\lambda\in\R^\times$ acts on $f$ by $\lambda^m$, $m\in\Z$. If $m$ is odd then $\lambda=-1\in\R^\times$ acts on $f$ by $-1$, but on the other hand $f$ is invariant under $-1\in{\rm Sp}(1)$, so $f=0$. If $m$ is even we write $m=-2k+4\ell$ with $k,\ell\in\Z_{\geq0}$, then $f\cdot\abs{X''}^{2k}\abs{Z}^{2\ell}$ is $\R^\times$-invariant and therefore contained in $\C[\mathbb{H}\oplus\Im\mathbb{H}]^{{\rm Sp}(1)\cdot\R^\times}$ which is generated by $p_1$, $p_2$ and $p_3$. This shows (ii).

Finally, (iii) follows by inspecting which polynomial expression in $p_1$, $p_2$ and $p_3$ is invariant under $F=\exp(\R U)$.
\end{proof}

In particular Lemma~\ref{lemma:poly_invariants} implies that the degree of every generator of $\C[\nbar]^{M'}$ is even. Hence the degree of every element of $\C[\nbar]_{\lambda,\nu}$ is even, so that Lemma~\ref{lemma:fourier_equations:i} implies:

\begin{corollary}\label{cor:DiffOpsOnlyInGlobalPoles}
If $\C[\nbar]_{\lambda,\nu}\neq\{0\}$, then $(\lambda,\nu)\in\GlobalPoles$.
\end{corollary}

\section{Classification of differential symmetry breaking operators: holomorphic families}\label{sec:DiffOpClass1}

We first treat the case of polynomials in $\C[\nbar]_{\lambda,\nu}$ which only depend on $\abs{X'}^2$, $\abs{X''}^2$ and $\abs{Z}^2$. Note that if $\C[\nbar]^{M'}$ is generated by $\abs{X'}^2$, $\abs{X''}^2$ and $\abs{Z}^2$, then these are in fact all polynomials in $\C[\nbar]_{\lambda,\nu}$.

Let $(\lambda,\nu)\in\GlobalPoles$ and write $\lambda+\rho-\nu-\rho'=-2k \in -2\Z_{\geq 0}$. We define the distribution
\begin{equation}
\label{eq:definition_u^C_hat}
\widehat{u}^C_{\lambda,\nu} := \sum_{h+i+2j=k} c_{h,i,j}(\lambda,\nu) \abs{X'}^{2h}\abs{X''}^{2i}\abs{Z}^{2j},
\end{equation}
where the scalars $c_{h,i,j}(\lambda,\nu)$ are for $m>0$ given by
\begin{equation}
 c_{h,i,j}(\lambda,\nu) = \frac{2^{-2i-2h}\Gamma(\frac{2\nu+p'+2}{4})\Gamma(\frac{\lambda+\rho+\nu-\rho'}{2}+i)}{h!i!j!\Gamma(\frac{2\nu+p'+2}{4}-j)\Gamma(\frac{p''}{2}+i)\Gamma(\frac{\lambda+\rho+\nu-\rho'}{2})}\label{eq:diff_op_scalars1}
\end{equation}
and for $m=0$ by
\begin{equation}
 c_{h,i,j}(\lambda,\nu) = \frac{2^{-i}\Gamma(\frac{\nu}{2}-j)}{i!j!\Gamma(\frac{p}{2}+i)\Gamma(\frac{\nu}{2}-\lfloor\frac{k}{2}\rfloor)}.\label{eq:diff_op_scalars2}
\end{equation}
A close inspection of the gamma factors shows that $c_{h,i,j}(\lambda,\nu)$ is holomorphic in $\lambda$ and $\nu$ and hence $u_{\lambda,\nu}^C$ depends holomorphically on $(\lambda,\nu)\in\GlobalPoles$. In the case $m=0$ we have $\mathfrak{v}'=0$ and hence $\abs{X'}^{2h}=0$ for $h>0$, so the summation is over $i+2j=k$ with $h=0$.

\begin{theorem}
\label{theorem:diff_solutions}
Let $(\lambda,\nu)\in\GlobalPoles$, then $\C[\nbar]_{\lambda,\nu}\cap\C[\abs{X'}^2,\abs{X''}^2,\abs{Z}^2]=\C \widehat{u}^C_{\lambda,\nu}$. In particular, if $\C[\nbar]^{M'}$ is generated by $\abs{X'}^2$, $\abs{X''}^2$ and $\abs{Z}^2$, then
$$ \C[\nbar]_{\lambda,\nu} = \C \widehat{u}^C_{\lambda,\nu}. $$
\end{theorem}

The proof of this statement uses the following combinatorial description of differential symmetry breaking operators:

\begin{prop}
\label{theorem:diff_SBO_equations}
Let $\lambda+\rho-\nu-\rho'=-2k \in  -2\Z_{\geq 0}$ and
	let $\widehat{u}\in\C[\nbar]$ be of the form
	$$\widehat{u}=\sum_{h+i+2j=k}c_{h,i,j}\abs{X'}^{2h}\abs{X''}^{2i}\abs{Z}^{2j},$$
	with scalars $c_{h,i,j}\in \C$. Then $\widehat{u}\in\C[\nbar]_{\lambda,\nu}$ if and only if the following relations are satisfied:
\begin{enumerate}[label=(\roman{*})]
\item
	If $m>0$ and $h+i+2j=k$:
	\begin{equation}
	\label{eq:R2}
	h(\lambda+\rho+\nu-\rho'+2i)c_{h,i,j} = (i+1)(2i+p'')c_{h-1,i+1,j},
	\end{equation}
	for $h>0$ and
	\begin{equation}
	\label{eq:R3}
	jc_{h,i,j} = 4(h+1)(h+2)(2h+p'+2)c_{h+2,i,j-1} + 4(h+1)(i+1)(2i+p'')c_{h+1,i+1,j-1},
	\end{equation}
	for $j>0$.
	\item
	If $m=0$ and $h+i+2j=k$:
	\begin{equation}
	\label{eq:R4}
	j(\nu+\rho'-q-2j)c_{h,i,j} = 2(i+1)(i+2)(2i+p)(2i+p+2)c_{h,i+2,j-1}
	\end{equation}
	for $j>0$.
	\end{enumerate}
\end{prop}

To prove the proposition we need some technical identities.

\begin{lemma}\label{lemma:diff_calculations2}
Let $h,i,j\geq 0$.

\begin{enumerate}[label=(\roman{*}), ref=\thetheorem(\roman{*})]
\item \label{lemma:diff_calculations2:i} 
$
\Delta_{\mathfrak{v}} \abs{X'}^{2h}\abs{X''}^{2i}=2h(2h+p'-2)\abs{X'}^{2h-2}\abs{X''}^{2i}+2i(2i+p''-2)\abs{X'}^{2h}\abs{X''}^{2i-2}.
$
\end{enumerate}
For $S\in\mathfrak{v}'$ we have
\begin{enumerate}[resume, label=(\roman{*}), ref=\thetheorem(\roman{*})]
\item \label{lemma:diff_calculations2:ii}  $\partial_{[S,X]}\abs{X'}^{2h}\abs{X''}^{2i}\abs{Z}^{2j}=\frac{j}{h+1}\partial_{J_ZS}\abs{X'}^{2h+2}\abs{X''}^{2i}\abs{Z}^{2j-2},$
\item \label{lemma:diff_calculations2:iii}  $P_S\abs{X'}^{2h}\abs{X''}^{2i}\abs{Z}^{2j}=-2j\partial_S\abs{X'}^{2h}\abs{X''}^{2i}\abs{Z}^{2j},$
\item \label{lemma:diff_calculations2:iv}  $Q_S\abs{X'}^{2h}\abs{X''}^{2i}=q\partial_S\abs{X'}^{2h}\abs{X''}^{2i}.$
\end{enumerate}
For $T\in \zfrak$  we have
\begin{enumerate}[resume, label=(\roman{*}), ref=\thetheorem(\roman{*})]
\item \label{lemma:diff_calculations2:v} $\partial_{J_TX} \abs{X'}^{2h}\abs{X''}^{2i}=0,$
\item \label{lemma:diff_calculations2:vi} $R_T\abs{X'}^{2h}\abs{X''}^{2i}\abs{Z}^{2j}=-4j(h+i)\langle Z,T \rangle \abs{X'}^{2h}\abs{X''}^{2i}\abs{Z}^{2j-2}.$
\end{enumerate}
\end{lemma}

\begin{proof}
Ad (i): This follow from $\Delta_{\mathfrak{v}}=\Delta_{\mathfrak{v}'}+\Delta_{\mathfrak{v}''}$ and the well known identity
$$\sum_{i=1}^{p} \frac{\partial^2}{\partial x_i^2} (x_1^2+ \dots + x_p^2)^l=l(l+p-2)(x_1^2+ \dots + x_p^2)^{l-2}.$$

Ad (ii): Clearly
\begin{align*}
\partial_{[S,X]}\abs{X'}^{2h}\abs{X''}^{2i}\abs{Z}^{2j} &= 2j \langle [S,X], Z \rangle \abs{X'}^{2h}\abs{X''}^{2i}\abs{Z}^{2j-2} 
\\ &= 2j \langle 
J_ZS, X \rangle \abs{X'}^{2h}\abs{X''}^{2i}\abs{Z}^{2j-2}
\\ & = \frac{j}{h+1} \partial_{J_ZS} \abs{X'}^{2h+2}\abs{X''}^{2i}\abs{Z}^{2j-2}.
\end{align*}

Ad (iii): We have $$P_S\abs{X'}^{2h}\abs{X''}^{2i}\abs{Z}^{2j}=\frac{j}{8}  \partial_{J_ZJ_ZS} \abs{X'}^{2h}\abs{X''}^{2i}\abs{Z}^{2j-2}=-2j\partial_S\abs{X'}^{2h}\abs{X''}^{2i}\abs{Z}^{2j}.$$

Ad (iv): The statement follows immediately by Lemma~\ref{lem:J_[S,Xi]}.

Ad (v): Since $\partial_{J_TX}=\partial_{J_TX'}+\partial_{J_TX''}$ this follows from $J_T|_{\mathfrak{v}'}\in\mathfrak{so}(\mathfrak{v}')$ and $J_T|_{\mathfrak{v}''}\in\mathfrak{so}(\mathfrak{v}'')$.

Ad (vi): We have
\begin{align*}
R_T\abs{X'}^{2h}\abs{X''}^{2i}\abs{Z}^{2j} ={}& \frac{j}{8}\partial_{J_ZJ_TX} \abs{X'}^{2h}\abs{X''}^{2i}\abs{Z}^{2j-2}\\
 ={}&-\frac{j}{4} \Big(h \langle J_ZX', J_T,X' \rangle \abs{X'}^{2h-2}\abs{X''}^{2i}\abs{Z}^{2j-2}\\
 & \hspace{2cm}+i\langle J_ZX'', J_T,X'' \rangle\abs{X'}^{2h}\abs{X''}^{2i-2}\abs{Z}^{2j-2} \Big).
\end{align*}
Now
\begin{align*}
 \langle J_ZX',J_TX' \rangle &= \frac{1}{2}\left( \langle J_{Z+T}X',J_{Z+T}X' \rangle - \langle J_{Z}X',J_{Z}X' \rangle-\langle J_{T}X',J_{T}X' \rangle \right)\\
 &=16\abs{X'}^2\langle Z,T\rangle
\end{align*}
and similarly
$$ \langle J_ZX'',J_TX'' \rangle=16\abs{X''}^2\langle Z,T\rangle, $$
which implies the statement.
\end{proof}

\begin{proof}[Proof of Proposition~\ref{theorem:diff_SBO_equations}]
First let $m>0$ and $S \in \mathfrak{v}'$. Using Lemma~\ref{lemma:diff_calculations2} we obtain that 
\begin{align*}
-\mathcal{F}(D_{\mathfrak{v}}(S))u(X,Z)={}&\sum_{h+i+2j=k}c_{h,i,j}\Big(-2i(2i+p''-2)\langle S , X \rangle \abs{X'}^{2h}\abs {X''}^{2i-2} \abs{Z}^{2j}\\
&\hspace{.5cm}+2h(2\nu+2\rho'-4j-2h-p'-2q)\langle S, X \rangle \abs{X'}^{2h-2}\abs{X''}^{2i}\abs{Z}^{2j} \\
&\hspace{.5cm}-2h(h-1)(2h+p'-2)\langle J_ZS,X \rangle\abs{X'}^{2h-4}\abs{X''}^{2i}\abs{Z}^{2j} \\
&\hspace{.5cm}-2ih(2i+p''-2)\langle J_ZS,X \rangle \abs{X'}^{2h-2}\abs{X''}^{2i-2}\abs{Z}^{2j} \\
&\hspace{.5cm}+\frac{j}{2}\langle J_ZS,X \rangle\abs{X'}^{2h}\abs{X''}^{2i}\abs{Z}^{2j-2}
\Big).
\end{align*}
After rearrangement the coefficients of $\langle S,X\rangle$ and $\langle J_ZS,X\rangle$ have to vanish separately since $\langle S,J_ZS\rangle=0$, so \eqref{eq:R2} and \eqref{eq:R3} follow.

For $m=0$, $D_{\mathfrak{v}}(S)u(X,Z)=0$, since $\mathfrak{v}'=\{0\}$. For $T \in \zfrak$ we have by Lemma~\ref{lemma:diff_calculations2}
\begin{multline*}
 -\mathcal{F}(D_\zfrak (T))u(X,Z)= \sum_{h+i+2j=k}\langle Z,T \rangle c_{h,i,j}\Big(2j(\nu+\rho'-2j-q) \abs{X}^{2i}\abs{Z}^{2j-2}\\
 -4i(i-1)(2i+p-2)(2i+p-4)\abs{X}^{2i-4}\abs{Z}^{2j}\Big).
\end{multline*}
As in the previous case, \eqref{eq:R4} follows after regrouping the summands and comparing coefficients.
\end{proof}

\begin{proof}[Proof of Theorem~\ref{theorem:diff_solutions}]
Clearly \eqref{eq:R2} implies that $$c_{h,i,j}= \frac{\Gamma(\frac{\lambda+\rho+\nu-\rho'}{2}+i)}{h!i!\Gamma(\frac{p''}{2}+i)\Gamma(\frac{\lambda+\rho+\nu-\rho'}{2})}c_{k-2j,0,j}.$$
Then \eqref{eq:R3} becomes
$$ j c_{k-2j,0,j}= 16\left(\frac{2\nu+p'-2}{4}-j\right)c_{k-2j+2,0,j-1}, $$
which implies
$$ c_{k-2j,0,j} = 2^{4j} \frac{\Gamma(\frac{2\nu+p'-2}{4})}{j!\Gamma(\frac{2\nu+p'-2}{4}-j)}c_{k,0,0}.$$
If $m=0$ the proof works analogously.
\end{proof}

\section{Classification of differential symmetry breaking operators: sporadic operators}

We now treat the cases where $\C[\nbar]^{M'}$ is not generated by $\abs{X'}^2$, $\abs{X''}^2$ and $\abs{Z}^2$.

\subsection{The complex case}

Let $\F=\C$, then $\C[\nbar]^{M'}$ is generated by $\abs{X'}^2$, $\abs{X''}^2$ and $Z$. To simplify the formulas, we identify $Z\in\Im\C=\qi\R$ with $\Im Z\in\R$. Let $\widehat{u}^C_{\lambda,\nu}$ be defined as in \eqref{eq:definition_u^C_hat}. We define an additional polynomial $\widehat{v}_{\lambda,\nu}^C\in\C[\nbar]$ by
$$ \widehat{v}_{\lambda,\nu}^C(X,Z) := \sum_{i+2j+1=k}\frac{2^{-i}\Gamma(\frac{\nu-1}{2}-j)}{i!\Gamma(j+\frac{3}{2})\Gamma(\frac{p}{2}+i)\Gamma(\frac{\nu-1}{2}-\lfloor\frac{k-1}{2}\rfloor)}\abs{X}^{2i}Z^{2j+1}. $$

\begin{theorem}
\label{theorem:diff_solution_complex}
Let $(G,G')=(\Urm(1,n+1),\Urm(1,m+1)\times F)$ be strongly spherical and $(\lambda,\nu)\in\GlobalPoles$ with $\lambda+\rho-\nu-\rho'=-2k\in -2\Z_{\geq 0}$. Then 
$$\C[\nbar]_{\lambda,\nu}=\begin{cases}
\C \widehat{u}^C_{\lambda,\nu} \oplus \C \widehat{v}_{\lambda,\nu}^C & \text{for $m=0$ and $\nu\in 1+2\Z_{\geq 0}$, $0 < \nu \leq k$,}\\
\C \widehat{u}^C_{\lambda,\nu} & \text{otherwise.}
\end{cases}$$
\end{theorem}

To show the theorem, we first prove an analog of Proposition~\ref{theorem:diff_SBO_equations} in this case.

\begin{prop}
\label{prop:com_equationions_C}
Let $\lambda+\rho-\nu-\rho'=-2k \in -2\Z_{\geq 0}$ and let $\widehat{u}\in\C[\nbar]$ be of the form
$$\widehat{u}=\sum_{h+i+j=k} c_{h,i,j}\abs{X'}^{2h} \abs{X''}^{2i} Z^j$$
with scalars $c_{h,i,j}\in \C$. Then $\widehat{u}\in\C[\nbar]_{\lambda,\nu}$ if and only if the following relations are satisfied:
\begin{enumerate}[label=(\roman{*}), ref=\thetheorem(\roman{*})]
\item If $m>0$ and $h+i+j=k$:
\begin{equation}
\label{eq:R1_C}
2h(\lambda+\rho+\nu-\rho'+2i)c_{h,i,j}=(i+1)(2i+p'')c_{h-1,i+1,j}
\end{equation}
for $h>0$ and 
\begin{equation}
\label{eq:R2_C}
jc_{h,i,j}=8(h+1)(h+2)(2h+p'+2)c_{h+2,i,j-2}+8(h+1)(i+1)(2i+p'')c_{h+1,i+1,j-2}
\end{equation}
for $j>0$, where we set $c_{h,i,-1}=0$.
\item If $m=0$ and $h+i+j=k$:
\begin{equation}
\label{eq:R3_C}
j(\nu+\rho'-q-j)c_{h,i,j}=4(i+1)(i+2)(2i+p)(2i+p+2)c_{h,i+2,j-2}
\end{equation}
for $j>0$, where we set $c_{h,i,-1}=0$.
\end{enumerate}
\end{prop}

Again we first state some technical identities, which follow in the same way as the ones in Lemma~\ref{lemma:diff_calculations2}.

\begin{lemma}
\label{lemma:diff_calculations_C}
\begin{enumerate}[label=(\roman{*}), ref=\thetheorem(\roman{*})]
\item $\partial_{[S,X]}\abs{X'}^{2h} \abs{X''}^{2i} Z^j = \frac{j}{2(h+1)}\partial_{J_ZS}\abs{X'}^{2h+2} \abs{X''}^{2i} Z^{j-2}$
\item $P_S\abs{X'}^{2h} \abs{X''}^{2i} Z^j=-j\partial_S \abs{X'}^{2h} \abs{X''}^{2i} Z^{j}$
\item $R_T\abs{X'}^{2h} \abs{X''}^{2i} Z^j =-2j(h+i) \abs{X'}^{2h} \abs{X''}^{2i} Z^{j-1}$
\end{enumerate}
\end{lemma}

\begin{proof}[Proof of Proposition~\ref{prop:com_equationions_C}]
First let $m>0$.
Combining Lemma~\ref{lemma:diff_calculations2} and Lemma~\ref{lemma:diff_calculations_C} we obtain

\begin{align*}
-\mathcal{F}(D_{\mathfrak{v}}(S))u(X,Z) ={}& \sum_{h+i+j=k}c_{h,i,j}\Big(-2i(2i+p''-2)\langle S,X \rangle \abs{X'}^{2h} \abs{X''}^{2i-2} Z^j\\
&\hspace{.8cm}+2h(2\nu+2\rho'-2j-2h-p'-2q)\langle S,X \rangle \abs{X'}^{2h-2} \abs{X''}^{2i} Z^j\\
&\hspace{.8cm}-2h(h-1)(2h+p'-2)\langle J_TS,X \rangle \abs{X'}^{2h-4} \abs{X''}^{2i} Z^{j+1} \\
&\hspace{.8cm}-2ih (2i+p''-2) \langle J_TS,X \rangle \abs{X'}^{2h-2} \abs{X''}^{2i-2} Z^{j+1} \\
&\hspace{.8cm}+\frac{j}{4}\langle J_TS ,X \rangle \abs{X'}^{2h} \abs{X''}^{2i} Z^{j-1} \Big).
\end{align*}
Since $S$ and $J_TS$ are orthogonal this implies the statement in this case.

Now let $m=0$, then combining Lemma~\ref{lemma:diff_calculations2} and Lemma~\ref{lemma:diff_calculations_C} we obtain
\begin{multline*}
-\mathcal{F}(D_\zfrak(T))u(X,Z)=\sum_{h+i+j=k}c_{h,i,j}\Big(
j(\nu+\rho'-1-j)\abs{X}^{2i} Z^{j-1} \\
-4i(i+1)(2i+p''-2)(2i+p''-4)\abs{X}^{2i-4} Z^{j+1}
\Big),
\end{multline*}
which proves the proposition in this case.
\end{proof}

Now we can prove Theorem~\ref{theorem:diff_solution_complex}.

\begin{proof}[Proof of Theorem~\ref{theorem:diff_solution_complex}]
First let $m>0$. Then \eqref{eq:R2_C} implies that $c_{h,i,1}=0$ for all possible $h,i$, which implies that all $c_{h,i,j}=0$ for odd $j$. Now the statement follows from Theorem~\ref{theorem:diff_solutions}.

Let us now consider the case $m=0$ and write $c_{i,j}=c_{0,i,j}$. First it is clear that \eqref{eq:R3_C} separates odd and even degrees in $Z$. For the even degrees in $Z$ the theorem follows in the same way as Theorem~\ref{theorem:diff_solutions}. 
Assume $\nu \neq 2l+1$ for $0\leq 2l \leq k-1$. Then \eqref{eq:R3_C} implies that $c_{i,j}=0$ first for $j=1$ and then recursively for all odd $j$. Now if $\nu = 2l+1$ for some $0\leq 2l \leq k-1$, then \eqref{eq:R3_C} implies $c_{i,2j+1}=0$ for all $j < l$ and further
\begin{align*}
c_{0,i,2j+1} &= \frac{\Gamma(\lfloor\frac{k-1}{2}\rfloor+\frac{3}{2})\Gamma(\frac{\nu-1}{2}-j)\Gamma(\frac{p}{2}+k-1-2\lfloor\frac{k-1}{2}\rfloor)}{4^{\lfloor\frac{k-1}{2}\rfloor-j}i!\Gamma(j+\frac{3}{2})\Gamma(\frac{\nu-1}{2}-\lfloor\frac{k-1}{2}\rfloor)\Gamma(\frac{p}{2}+i)}c_{0,k-1-2\lfloor\frac{k-1}{2}\rfloor,2\lfloor\frac{k-1}{2}\rfloor+1}\\
&= {\rm const}\times\frac{2^{2j}\Gamma(\frac{\nu-1}{2}-j)}{i!\Gamma(j+\frac{3}{2})\Gamma(\frac{p}{2}+i)\Gamma(\frac{\nu-1}{2}-\lfloor\frac{k-1}{2}\rfloor)}.\qedhere
\end{align*}
\end{proof}

\subsection{The quaternionic case}

Let $\F=\HH$ and $m=n-1$, then $(G,G')=({\rm Sp}(1,n+1),{\rm Sp}(1,n)\times F)$ is real spherical for all $F<{\rm Sp}(1)$. The only connected subgroups of ${\rm Sp}(1)$ are (up to conjugation) $F=\{1\}$, $\Urm(1)$ and ${\rm Sp}(1)$. The case $F={\rm Sp}(1)$ was already treated by Theorem~\ref{theorem:diff_solutions}. We now discuss the case $F=\{1\}$, then the statement for $F=\Urm(1)$ will follow easily.

We use the orthonormal basis
$$ S_1 = 1, \quad S_2 = \qi, \quad S_3 = \qj, \quad S_4 = \qk $$
of $\HH$ and the orthonormal basis
$$ T_1 = \qi, \quad T_2 = \qj, \quad T_3 = \qk $$
of $\Im\HH$, so that an element $(X,Z)\in \nbar$ is given by
$$ (X,Z)=\sum_{i=1}^4 X_iS_i + \sum_{j=1}^{3}Z_jT_j, $$
for real variables $X_i$ and $Z_j$.

We recall the ${\rm Sp}(1)$-invariants $p_j(X,Z)=\langle T_j,\overline{X}ZX\rangle$ on $\HH\oplus\Im\HH$ from Section~\ref{sec:Invariants} and note the following derivatives for $S\in\HH$, $T\in\Im\HH$:
$$ \partial_Sp_j(X,Z) = 2\langle T_j,\overline{X}ZS\rangle, \qquad \partial_Tp_j(X,Z) = \langle T_j,\overline{X}TX\rangle. $$

For a polynomial $q\in\C[p_1,p_2,p_3]$ in three variables we denote by $q(p_1,p_2,p_3)$ the polynomial on $\HH^n\oplus\Im\HH$ given by
$$ q(p_1,p_2,p_3)(X,Z) = q(p_1(X'',Z),p_2(X'',Z),p_3(X'',Z)). $$
For $l \in \Z_{\geq 0}$ define the set
$$ \mathcal{H}^l(p_1,p_2,p_3):= \{ q(p_1,p_2,p_3), \; q \in \mathcal{H}^l(\R^3)  \}. $$

\begin{theorem}\label{thm:diff_solution_quaternionic}
Let $(G,G')=({\rm Sp}(1,n+1),{\rm Sp}(1,n))$ and $(\lambda,\nu)\in\GlobalPoles$ with $\lambda+\rho-\nu-\rho'=-2k\in -2\Z_{\geq 0}$. Then
$$\C[\nbar]_{\lambda,\nu}=\begin{cases}
\C\widehat{u}^C_{\lambda,\nu} \oplus \mathcal{H}^\frac{k}{2}(p_1,p_2,p_3) & \text{for $(n,m)=(1,0)$, $k>0$ even and $\nu+\rho'=k+4$,}\\
\C \widehat{u}^C_{\lambda,\nu} & \text{otherwise.}
\end{cases}$$
\end{theorem}

The proof is split into Propositions~\ref{prop:DiffSBOsSp(1,n)} and \ref{prop:DiffSBOs_SP(1,2)}.

\begin{remark}
Since $\dim\mathcal{H}^\ell(\R^3)=2\ell+1$ we have for $n=1$
$$ \dim\mathcal{D}'_{\{0\}}(\nbar)_{\lambda,\nu} = k+2 \qquad \mbox{for $(\lambda,\nu)=(-(k+1),k+1)$.} $$
In particular, the dimension can become arbitrarily large while both $\pi_\lambda$ and $\tau_\nu$ have at most $4$ composition factors.
\end{remark}

We now deduce from Theorem~\ref{thm:diff_solution_quaternionic} (the case $F=\{1\}$) the remaining case $F=\Urm(1)$. Realizing ${\rm Sp}(1)$ as the group of unit quaternions we have $\mathfrak{sp}(1)=\Im\HH$. Then the Lie algebra of $F=\Urm(1)$ is generated by a single element $U$ which we write as $U=U_1\qi+U_2\qj+U_3\qk\in\Im\HH=\mathfrak{sp}(1)$.

\begin{corollary}\label{cor:diff_solution_quaternionic}
Let $(G,G')=({\rm Sp}(1,n+1),{\rm Sp}(1,n)\times\Urm(1))$, where the Lie algebra $\mathfrak{u}(1)$ of the $\Urm(1)$-factor is generated by $U=U_1\qi+U_2\qj+U_3\qk\in\Im\HH=\mathfrak{sp}(1)$. Then for $(\lambda,\nu)\in\GlobalPoles$ with $\lambda+\rho-\nu-\rho'=-2k\in -2\Z_{\geq 0}$ we have
$$\C[\nbar]_{\lambda,\nu}=\begin{cases}
\C\widehat{u}^C_{\lambda,\nu} \oplus \C(U_1p_1+U_2p_2+U_3p_3) & \text{for $(n,m)=(1,0)$ and $\nu+\rho'=6$,}\\
\C \widehat{u}^C_{\lambda,\nu} & \text{otherwise.}
\end{cases}$$
\end{corollary}

\begin{proof}
Since $\widehat{u}_{\lambda,\nu}^C$ is $\Urm(1)$-invariant, by Theorem~\ref{thm:diff_solution_quaternionic} it suffices to show that
$$ \mathcal{H}^\ell(p_1,p_2,p_3)^{\Urm(1)} = \begin{cases}\C1&\mbox{for $\ell=0$,}\\\C(U_1p_1+U_2p_2+U_3p_3)&\mbox{for $\ell=1$,}\\\{0\}&\mbox{else.}\end{cases} $$
Clearly $\C[p_1,p_2,p_3]^{\Urm(1)}$ is generated by $U_1p_1+U_2p_2+U_3p_3$, so every polynomial in $\mathcal{H}^\ell(p_1,p_2,p_3)$ is of the form $f(U_1p_1+U_2p_2+U_3p_3)$ with $f\in\C[t]$ homogeneous of degree $\ell$. Since $\Delta_p[f(U_1p_1+U_2p_2+U_3p_3)]=(U_1^2+U_2^2+U_3^2)f''(U_1p_1+U_2p_2+U_3p_3)$ this implies $f''=0$, so either $f=0$ or $\ell\in\{0,1\}$. For $\ell=0$ the polynomial $f$ is constant and for $\ell=1$ it is a constant multiple of $t$ and the statement follows.
\end{proof}

\subsubsection{The case $n>1$}

To apply the terms occurring in Proposition~\ref{prop:fourier_equations} to a general invariant polynomial, we first prove some basic identities for the invariants $p_1$, $p_2$ and $p_3$.

\begin{lemma}
Let $q\in\C[p_1,p_2,p_3]$ and write $\Delta_p=\sum_{j=1}^3\frac{\partial^2}{\partial p_j^2}$, then
\label{lemma:SP(1,2)_basic_calc}
\begin{enumerate}[label=(\roman{*}), ref=\thetheorem(\roman{*})]
\item\label{lemma:SP(1,2)_basic_calc:i} $\Delta_{\mathfrak{v}}(q(p_1,p_2,p_3))=4\abs{X}^2\abs{Z}^2(\Delta_pq)(p_1,p_2,p_3)$,
\item\label{lemma:SP(1,2)_basic_calc:ii} $\Box(q(p_1,p_2,p_3))=\abs{X}^4(\Delta_pq)(p_1,p_2,p_3)$.
\end{enumerate}
\end{lemma}

\begin{proof}
Ad (i): By the product rule we have
\begin{multline*}
 \Delta_{\mathfrak{v}}(q(p_1,p_2,p_3)) = \sum_{i=1}^{4}\sum_{j,k=1}^3 \frac{\partial^2q}{\partial p_j\partial p_k}(p_1,p_2,p_3)\frac{\partial p_j}{\partial X_i}(X,Z)\frac{\partial p_k}{\partial X_i}(X,Z)\\
 +\sum_{j=1}^3\frac{\partial q}{\partial p_j}(p_1,p_2,p_3)\Delta_{\mathfrak{v}}(p_j(X,Z)).
\end{multline*}
For $j,k\in\{1,2,3\}$ we have
\begin{align*}
 \sum_{i=1}^4 \frac{\partial p_j}{\partial X_i}(X,Z)\frac{\partial p_k}{\partial X_i}(X,Z) &= 4\sum_{i=1}^4 \langle T_j,\overline{X}ZS_i\rangle\langle T_k,\overline{X}ZS_i\rangle\\
 &= -4\sum_{i=1}^4 \langle T_j,\overline{X}ZS_i\rangle\langle ZXT_k,S_i\rangle\\
 &= -4\langle T_j,\overline{X}ZZXT_k\rangle\\
 &= 4|X|^2|X|^2\langle T_j,T_k\rangle\\
 &= 4|X|^2|Z|^2\delta_{j,k}.
\end{align*}
Further, for every $j\in\{1,2,3\}$
$$ \Delta_{\mathfrak{v}} p_j(X,Z) = 2\sum_{i=1}^4 \langle T_j,\overline{S}_iZS_i\rangle = 0, $$
where the last step follows from the identity $\qi Z\qi+\qj Z\qj+\qk Z\qk=Z$ for $Z\in\Im\HH$ which is easily verified.

Ad (ii): This is similar to (i) using for $j,k\in\{1,2,3\}$ the identity
\begin{align*}
 \sum_{i=1}^3 \frac{\partial p_j}{\partial Z_i}(X,Z)\frac{\partial p_k}{\partial Z_i}(X,Z) &= \sum_{i=1}^3 \langle T_j,\overline{X}T_iX\rangle\langle T_k,\overline{X}T_iX\rangle\\
 &= \sum_{i=1}^3 \langle T_j,\overline{X}T_iX\rangle\langle XT_k\overline{X},T_i\rangle\\
 &= \langle T_j,\overline{X}XT_k\overline{X}X\rangle\\
 &= |X|^4\langle T_j,T_k\rangle\\
 &= |X|^4\delta_{j,k},
\end{align*}
as well as $\Box(p_j(X,Z))=0$ which is clear since $p_j(X,Z)$ is linear in $Z$.
\end{proof}

\begin{lemma}
\label{lemma:SP(1,n)_complex_calc}
Let $q\in\mathcal{H}^\ell(\R^3)$, then for $S\in\mathfrak{v}'$ we have
\begin{enumerate}[label=(\roman{*}), ref=\thetheorem(\roman{*})]
\item $\partial_S\abs{X'}^{2h} = 2h\langle X',S\rangle\abs{X'}^{2h-2},$
\item\label{lemma:xLaplaceOnGeneralTerm}
$\begin{aligned}[t]
 \Delta_{\mathfrak{v}}(\abs{X'}^{2h}\abs{X''}^{2i}q(p_1,p_2,p_3)) ={}& 2h(2h+p'-2)\abs{X'}^{2h-2}\abs{X''}^{2i}q(p_1,p_2,p_3)\\
 & \hspace{1.05cm}+4i(i+2\ell+1)\abs{X'}^{2h}\abs{X''}^{2i-2}q(p_1,p_2,p_3),
\end{aligned}$
\item $\begin{aligned}[t]
 & \partial_{J_ZS}\Delta_{\mathfrak{v}}(\abs{X'}^{2h}\abs{X''}^{2i}q(p_1,p_2,p_3))\\
 & \hspace{2.9cm} = 4h(h-1)(2h+p'-2)\langle Z,[S,X']\rangle\abs{X'}^{2h-4}\abs{X''}^{2i}q(p_1,p_2,p_3)\\
 & \hspace{3.7cm} +8ih(i+2\ell+1)\abs{X'}^{2h-2}\abs{X''}^{2i-2}\langle Z,[S,X']\rangle q(p_1,p_2,p_3),
\end{aligned}$
\item $\begin{aligned}[t]
 & \partial_{[S,X]}(\abs{Z}^{2j}q(p_1,p_2,p_3)) = 2j\langle Z,[S,X']\rangle\abs{Z}^{2j-2}q(p_1,p_2,p_3)\\
 & \hspace{8.5cm} + \abs{Z}^{2j}\partial_{[S,X']}q(p_1,p_2,p_3),
\end{aligned}$
\item $\begin{aligned}[t]
 P_S(\abs{X'}^{2h}\abs{Z}^{2j}q(p_1,p_2,p_3)) ={}& -4hj\abs{X'}^{2h-2}\abs{Z}^{2j}\langle X',S\rangle q(p_1,p_2,p_3)\\
 & \hspace{1.75cm}+\frac{1}{8}h\abs{X'}^{2h-2}\abs{Z}^{2j}\partial_{[J_ZS,X']}q(p_1,p_2,p_3),
\end{aligned}$
\item $\begin{aligned}[t]
 & Q_S(\abs{X'}^{2h}\abs{X''}^{2i}q(p_1,p_2,p_3)) = 2hq\abs{X'}^{2h-2}\abs{X''}^{2i}\langle X',S\rangle q(p_1,p_2,p_3)\\
 & \hspace{6cm} +\frac{1}{8}h\abs{X'}^{2h-2}\abs{X''}^{2i}\partial_{J_{[S,X']}X''}q(p_1,p_2,p_3).
\end{aligned}$
\end{enumerate}
\end{lemma}

\begin{proof}
For the whole proof we remark that $q(p_1,p_2,p_3)$ is homogeneous in $X$ of degree $2\ell$ and homogeneous in $Z$ of degree $\ell$.

Ad (i): This is clear.

Ad (ii): This follows by an easy application of the product rule, Lemma~\ref{lemma:SP(1,2)_basic_calc:i} and the identities $\Delta_{\mathfrak{v}}\abs{X'}^{2h}=2h(2h+p'-2)\abs{X'}^{2h-2}$ and $\Delta_{\mathfrak{v}}\abs{X''}^{2i}=4i(i+1)\abs{X''}^{2i-2}$.

Ad (iii): Here $\partial_{J_ZS}$ is applied to (ii) using
$$ \partial_{J_ZS}\abs{X'}^{2h} = 2h\langle J_ZS,X'\rangle\abs{X'}^{2h-2} = 2h\langle Z,[S,X']\rangle\abs{X'}^{2h-2}. $$

Ad (iv): This is the product rule.

Ad (v): We compute
\begin{align*}
 & P_S(\abs{X'}^{2h}\abs{Z}^{2j}q(p_1,p_2,p_3)\\
 ={}& \frac{1}{16}\sum_{a=1}^3\partial_{J_{T_a}J_ZS}(\abs{X'}^{2h})\partial_{T_a}(\abs{Z}^{2j}q(p_1,p_2,p_3))\\
 ={}& \frac{h}{8}\abs{X'}^{2h-2}\sum_{a=1}^3\langle J_{T_a}J_ZS,X'\rangle\Big[2j\langle T_a,Z\rangle\abs{Z}^{2j-2}q(p_1,p_2,p_3)+\abs{Z}^{2j}\partial_{T_a}q(p_1,p_2,p_3)\Big]\\
 ={}& \frac{hj}{4}\langle J_Z^2S,X'\rangle\abs{X'}^{2h-2}\abs{Z}^{2j-2}q(p_1,p_2,p_3)\\
 & \hspace{4.3cm}+ \frac{h}{8}\abs{X'}^{2h-2}\abs{Z}^{2j}\sum_{a=1}^3\langle T_a,[J_ZS,X']\rangle\partial_{T_a}q(p_1,p_2,p_3)\\
 ={}& -4hj\langle S,X'\rangle\abs{X'}^{2h-2}\abs{Z}^{2j}q(p_1,p_2,p_3) + \frac{h}{8}\abs{X'}^{2h-2}\abs{Z}^{2j}\partial_{[J_ZS,X']}q(p_1,p_2,p_3).
\end{align*}

Ad (vi): We first note that $[S,S_i]=0$ whenever $S_i\in\mathfrak{v}''$, so that
$$ Q_S = \frac{1}{16}\sum_{i=1}^{p'}\partial_{J_{[S,S_i]}X}\partial_{S_i}. $$
Then clearly $\partial_{J_{[S,S_i]}X}\abs{X''}^{2i}=0$. Further, since $J_{[S,S_i]}|_{\mathfrak{v}'}\in\mathfrak{so}(\mathfrak{v}')$ we have $\partial_{J_{[S,S_i]}X}\abs{X'}^{2h}=0$. Hence
\begin{align*}
 & Q_S(\abs{X'}^{2h}\abs{X''}^{2i}q(p_1,p_2,p_3))\\
 ={}& \frac{1}{8}h\sum_{a=1}^p\partial_{J_{[S,S_a]}X}\big(\langle S_a,X'\rangle\abs{X'}^{2h-2}\abs{X''}^{2i}q(p_1,p_2,p_3)\big)\\
 ={}& \frac{1}{8}h\abs{X'}^{2h-2}\abs{X''}^{2i}\sum_{a=1}^{p'}\Big(\langle S_a,J_{[S,S_a]}X'\rangle q(p_1,p_2,p_3)+\langle S_a,X'\rangle\partial_{J_{[S,S_a]}X''}q(p_1,p_2,p_3)\Big)\\
 ={}& \frac{1}{8}h\abs{X'}^{2h-2}\abs{X''}^{2i}\left(-\left\langle X',\sum_{a=1}^{p'}J_{[S,S_a]}S_a\right\rangle q(p_1,p_2,p_3)+\partial_{J_{[S,X']}X''}q(p_1,p_2,p_3)\right)
\end{align*}
and the claimed formula follows with Lemma~\ref{lem:J_[S,Xi]}.
\end{proof}

\begin{prop}\label{prop:DiffSBOsSp(1,n)}
For $(G,G')=({\rm Sp}(1,n+1),{\rm Sp}(1,n))$ with $n>1$ and $(\lambda,\nu)\in\GlobalPoles$ we have
$$ \C[\overline{\mathfrak{n}}]_{\lambda,\nu} = \C\widehat{u}_{\lambda,\nu}^C. $$
\end{prop}

\begin{proof}
Let $\widehat{u}\in\C[\nbar]_{\lambda,\nu}$, then $\widehat{u}$ is an ${\rm Sp}(n-1)\times{\rm Sp}(1)$-invariant polynomial on $\HH^{n-1}\oplus(\HH\oplus\Im\HH)$ which is homogeneous of degree $2k$. By Lemma~\ref{lemma:poly_invariants} it can therefore be written in the form
$$ \widehat{u}(X',X'',Z) = \sum_{h+i+2j+2\ell=k}\abs{X'}^{2h}\abs{X''}^{2i}\abs{Z}^{2j}q_\ell^{h,i,j}(p_1,p_2,p_3) $$
with $q_\ell^{h,i,j}\in\C[p_1,p_2,p_3]$ homogeneous of degree $\ell$ and $\lambda+\rho-\nu-\rho'=-2k$. Since $p_1^2+p_2^2+p_3^2=\abs{X''}^4\abs{Z}^2$ we can, thanks to the Fischer decomposition, without loss of generality assume that $q_\ell^{h,i,j}\in\mathcal{H}^\ell(\R^3)$. This makes the above decomposition unique.

By Lemma~\ref{lemma:SP(1,n)_complex_calc} the differential operator $-\mathcal{F}(D_{\mathfrak{v}}(S))$ in Proposition~\ref{prop:fourier_equations} applied to $\widehat{u}$ takes the following form:
\begin{align*}
 & \sum_{h+i+2j+2\ell=k}\Bigg[4h(\nu+\rho'-1)\abs{X'}^{2h-2}\abs{X''}^{2i}\abs{Z}^{2j}\langle X',S\rangle\\
 & \hspace{.3cm} -2h(2h+p'-2)\abs{X'}^{2h-2}\abs{X''}^{2i}\abs{Z}^{2j}\langle X',S\rangle\\
 & \hspace{.3cm} -4i(i+2\ell+1)\abs{X'}^{2h}\abs{X''}^{2i-2}\abs{Z}^{2j}\langle X',S\rangle\\
 & \hspace{.3cm} -2h(h-1)(2h+p'-2)\abs{X'}^{2h-4}\abs{X''}^{2i}\abs{Z}^{2j}\langle Z,[S,X']\rangle\\
 & \hspace{.3cm} -4hi(i+2\ell+1)\abs{X'}^{2h-2}\abs{X''}^{2i-2}\abs{Z}^{2j}\langle Z,[S,X']\rangle\\
 & \hspace{.3cm} +\tfrac{1}{2}j\abs{X'}^{2h}\abs{X''}^{2i}\abs{Z}^{2j-2}\langle Z,[S,X']\rangle+\tfrac{1}{4}\abs{X'}^{2h}\abs{X''}^{2i}\abs{Z}^{2j}\partial_{[S,X']}\\
 & \hspace{.3cm} -8hj\abs{X'}^{2h-2}\abs{X''}^{2i}\abs{Z}^{2j}\langle X',S\rangle+\tfrac{1}{4}h\abs{X'}^{2h-2}\abs{X''}^{2i}\abs{Z}^{2j}\partial_{[J_ZS,X']}\\
 & \hspace{.3cm} -4hq\abs{X'}^{2h-2}\abs{X''}^{2i}\abs{Z}^{2j}\langle X',S\rangle-\tfrac{1}{4}h\abs{X'}^{2h-2}\abs{X''}^{2i}\abs{Z}^{2j}\partial_{J_{[S,X']}X''}\Bigg]q_\ell^{h,i,j}(p_1,p_2,p_3).
\end{align*}
After rearrangement we obtain
\begin{align}
 & \sum_{h+i+2j+2\ell=k}\frac{1}{4}\abs{X'}^{2h}\abs{X''}^{2i}\abs{Z}^{2j}\partial_{[S,X']}q_\ell^{h,i,j}(p_1,p_2,p_3)\notag\\
 & +\sum_{h+i+2j+2\ell=k-1}\abs{X'}^{2h}\abs{X''}^{2i}\abs{Z}^{2j}\Bigg[4(h+1)(\nu-h-2j-1)\langle X',S\rangle q_\ell^{h+1,i,j}(p_1,p_2,p_3)\notag\\
 & \hspace{2.3cm} -4(i+1)(i+2\ell+2)(X',S)q_\ell^{h,i+1,j}(p_1,p_2,p_3)\notag\\
 & \hspace{2.3cm} +\frac{1}{4}(h+1)\big(\partial_{[J_ZS,X']}-\partial_{J_{[S,X']}X''}\big)q_\ell^{h+1,i,j}(p_1,p_2,p_3)\Bigg]\notag\\
 & +\sum_{h+i+2j+2\ell=k-2}\abs{X'}^{2h}\abs{X''}^{2i}\abs{Z}^{2j}\Bigg[\frac{1}{2}(j+1)\langle Z,[S,X']\rangle q_\ell^{h,i,j+1}(p_1,p_2,p_3)\notag\\
 & \hspace{2.3cm} -2(h+1)(h+2)(2h+p'+2)\langle Z,[S,X']\rangle q_\ell^{h+2,i,j}(p_1,p_2,p_3)\notag\\
 & \hspace{2.3cm} -4(h+1)(i+1)(i+2\ell+2)\langle Z,[S,X']\rangle q_\ell^{h+1,i+1,j}(p_1,p_2,p_3)\Bigg]=0.\label{eq:ProofSp(1,n)1}
\end{align}
We claim that if $q\in\C[p_1,p_2,p_3]$ is harmonic, then the polynomials $\partial_{[S,X']}q(p_1,p_2,p_3)$, $\langle X',S\rangle q(p_1,p_2,p_3)$ and $(\partial_{[J_ZS,X']}-\partial_{J_{[S,X']}X''})q_\ell^{h+1,i,j}(p_1,p_2,p_3)$ are all harmonic in $X'$, $X''$ and $Z$. In fact, all polynomials are linear in $X'$ and hence harmonic in $X'$. They are further harmonic in $X''$ since the operators $\partial_{[S,X']}$, $\langle X',S\rangle$, $\partial_{[J_ZS,X']}$ and $\partial_{J_{[S,X']}X''}$ commute with the Laplacian $\Delta_{\mathfrak{v}''}$ on $\mathfrak{v}''$. This is obvious for the first three operators and follows from Lemma~\ref{lem:LaplacianCommutators} for the fourth one. Finally, the operators $\partial_{[S,X']}$, $\langle X',S\rangle$ and $\partial_{J_{[S,X']}X''}$ also commute with the Laplacian $\Box$ on $\mathfrak{z}$ whence the corresponding polynomials are harmonic in $Z$. For the remaining polynomial we have by Lemma~\ref{lem:LaplacianCommutators}
$$ \Box\partial_{[J_ZS,X']}q(p_1,p_2,p_3) = \partial_{[J_ZS,X']}\Box q(p_1,p_2,p_3) -32 \langle S,X'\rangle\Box q(p_1,p_2,p_3) = 0. $$
The polynomial $\langle Z,[S,X']\rangle q(p_1,p_2,p_3)$ is also harmonic in $X'$ and $X''$ for the same reasons as above. It is in general not harmonic in $Z$, so we decompose it as
\begin{multline*}
 \langle Z,[S,X']\rangle q(p_1,p_2,p_3) = \underbrace{\langle Z,[S,X']\rangle q(p_1,p_2,p_3)-\frac{\abs{Z}^2}{2\ell+1}\partial_{[S,X']}q(p_1,p_2,p_3)}_{q(p_1,p_2,p_3)_+:=}\\
 + \frac{\abs{Z}^2}{2\ell+1}\partial_{[S,X']}q(p_1,p_2,p_3)
\end{multline*}
with $q(p_1,p_2,p_3)_+$ and $\partial_{[S,X']}q(p_1,p_2,p_3)$ harmonic in $X'$, $X''$ and $Z$. Inserting this into \eqref{eq:ProofSp(1,n)1} gives
\begin{align}
 & \sum_{h+i+2j+2\ell=k}\abs{X'}^{2h}\abs{X''}^{2i}\abs{Z}^{2j}\frac{1}{2\ell+1}\Bigg[\frac{1}{4}(2j+2\ell+1)\partial_{[S,X']}q_\ell^{h,i,j}(p_1,p_2,p_3)\notag\\
 & \hspace{4cm} -2(h+1)(h+2)(2h+p'+2)\partial_{[S,X']}q_\ell^{h+2,i,j-1}(p_1,p_2,p_3)\notag\\
 & \hspace{4cm} -4(h+1)(i+1)(i+2\ell+2)\partial_{[S,X']}q_\ell^{h+1,i+1,j-1}(p_1,p_2,p_3)\Bigg]\notag\\
 & +\sum_{h+i+2j+2\ell=k-1}\abs{X'}^{2h}\abs{X''}^{2i}\abs{Z}^{2j}\Bigg[4(h+1)(\nu-h-2j-1)\langle X',S\rangle q_\ell^{h+1,i,j}(p_1,p_2,p_3)\notag\\
 & \hspace{4cm} -4(i+1)(i+2\ell+2)\langle X',S\rangle q_\ell^{h,i+1,j}(p_1,p_2,p_3)\notag\\
 & \hspace{4cm} +\frac{1}{4}(h+1)\big(\partial_{[J_ZS,X']}-\partial_{J_{[S,X']}X''}\big)q_\ell^{h+1,i,j}(p_1,p_2,p_3)\Bigg]\notag\\
 & +\sum_{h+i+2j+2\ell=k-2}\abs{X'}^{2h}\abs{X''}^{2i}\abs{Z}^{2j}\Bigg[-2(h+1)(h+2)(2h+p'+2)q_\ell^{h+2,i,j}(p_1,p_2,p_3)_+\notag\\
 & \hspace{4cm} -4(h+1)(i+1)(i+2\ell+2)q_\ell^{h+1,i+1,j}(p_1,p_2,p_3)_+\notag\\
 & \hspace{4cm} +\frac{1}{2}(j+1)q_\ell^{h,i,j+1}(p_1,p_2,p_3)_+\Bigg]=0,\label{eq:ProofSp(1,n)2}
\end{align}
where we use the convention $q_\ell^{h,i,-1}=0$. Now all terms in square brackets are harmonic in $X'$, $X''$ and $Z$, so uniqueness of the Fischer decomposition implies that for fixed $h$, $i$ and $j$ the sum of the three square brackets with $2\ell=k-h-i-2j$, $2\ell=k-h-i-2j-1$ and $2\ell=k-h-i-2j-2$, respectively, vanishes. Since $k-h-i-2j$ is either even or odd, this sum contains either only the first and the third square bracket or only the second one. Further, the first square bracket is of degree $2\ell=k-h-i-2j$ in $X''$ whereas the third square bracket is of degree $2\ell=k-h-i-2j-2$ in $X''$, so the square brackets have to vanish separately. For the first square bracket this implies
\begin{multline*}
 \!\!\!\!\partial_{[S,X']}\Bigg(\frac{1}{4}(2j+2\ell+1)q_\ell^{h,i,j}(p_1,p_2,p_3) - 2(h+1)(h+2)(2h+p'+2)q_\ell^{h+2,i,j-1}(p_1,p_2,p_3)\\
 -4(h+1)(i+1)(i+2\ell+2)q_\ell^{h+1,i+1,j-1}(p_1,p_2,p_3)\Bigg) = 0.
\end{multline*}
Now $[S,X']$ ranges over $[\mathfrak{v'},\mathfrak{v}']=\mathfrak{z}$, so the term in the brackets is independent of $T\in\mathfrak{z}$. This implies that for $\ell>0$ the term in the brackets vanishes. For $j=0$ this means that $q_\ell^{h,i,0}=0$ whenever $\ell>0$ since $q_\ell^{h,i,-1}=0$. Inductively, the same equation gives $q_\ell^{h,i,j}=0$ whenever $\ell>0$. Since $q_0^{h,i,h}=c_{h,i,j}\in\mathcal{H}^0(\R^3)=\C$ are scalars, $\widehat{u}$ is of the form
$$ \widehat{u} = \sum_{h+i+2j=k}c_{h,i,j}\abs{X'}^{2h}\abs{X''}^{2i}\abs{Z}^{2j} $$
and the rest follows from Theorem~\ref{theorem:diff_solutions}.
\end{proof}

\subsubsection{The case $n=1$}

For $n=1$ we have $\mathfrak{v}'=\{0\}$ and therefore we only have the differential equation for $T\in\mathfrak{z}$ in Proposition~\ref{prop:fourier_equations}. We apply the terms occurring in this equation separately to a general invariant polynomial.

\begin{lemma}
\label{lemma:SP(1,2)_complex_calc}
Let $q\in\mathcal{H}^\ell(\R^3)$, then for $T\in\mathfrak{z}$ we have
\begin{enumerate}[label=(\roman{*}), ref=\thetheorem(\roman{*})]
\item $\partial_T(\abs{Z}^{2j}q(p_1,p_2,p_3)) = 2j\abs{Z}^{2j-2}\langle T,Z\rangle q(p_1,p_2,p_3) + \abs{Z}^{2j}\partial_Tq(p_1,p_2,p_3),$
\item $\partial_{J_TX}\Delta_{\mathfrak{v}}(\abs{X}^{2i}q(p_1,p_2,p_3)) = 4i(i+2\ell+1)\abs{X}^{2i-2}\partial_{J_{T}X}q(p_1,p_2,p_3),$
\item $\begin{aligned}[t]
 R_T(\abs{X}^{2i}\abs{Z}^{2j}q(p_1,p_2,p_3)) ={}& -4ij\abs{X}^{2i}\abs{Z}^{2j-2}\langle Z,T\rangle q(p_1,p_2,p_3)\\
 & \hspace{1.2cm}-2(i+2j+2\ell+1)\abs{X}^{2i}\abs{Z}^{2j}\partial_Tq(p_1,p_2,p_3),
\end{aligned}$
\item $\Delta_{\mathfrak{v}}^2(\abs{X}^{2i}q(p_1,p_2,p_3)) = 16(i-1)i(i+2\ell)(i+2\ell+1)\abs{X}^{2i-4}q(p_1,p_2,p_3),$
\item $\square(\abs{Z}^{2j}q(p_1,p_2,p_3)) = 2j(2j+2\ell+1)\abs{Z}^{2j-2}q(p_1,p_2,p_3).$
\end{enumerate}
\end{lemma}

\begin{proof}
Ad (i): This is the product rule.

Ad (ii): Since $J_T|_{\mathfrak{v}'}\in\mathfrak{so}(\mathfrak{v}')$ and $J_T|_{\mathfrak{v}''}\in\mathfrak{so}(\mathfrak{v}'')$ we have $\partial_{J_TX}\abs{X'}^{2h}=\partial_{J_TX}\abs{X''}^{2i}=0$ and the formula follows from Lemma~\eqref{lemma:xLaplaceOnGeneralTerm}.

Ad (iii): We first compute $R_Tq(p_1,p_2,p_3)$:
\begin{align*}
 R_Tq(p_1,p_2,p_3) ={}& \frac{1}{16}\sum_{a=1}^3\partial_{J_{T_a}J_TX}\sum_{b=1}^3\partial_{T_a}p_b(X,Z)\frac{\partial q}{\partial p_b}(p_1,p_2,p_3)\\
 ={}& \frac{1}{16}\sum_{a,b=1}^3\partial_{J_{T_a}J_TX}\partial_{T_a}p_b(X,Z)\frac{\partial q}{\partial p_b}(p_1,p_2,p_3)\\
 & \hspace{1.5cm} + \frac{1}{16}\sum_{a,b,c=1}^3\partial_{J_{T_a}J_TX}p_c(X,Z)\partial_{T_a}p_b(X,Z)\frac{\partial^2 q}{\partial p_b\partial p_c}(p_1,p_2,p_3)\\
 ={}& 2\sum_{a,b=1}^3\langle T_b,\overline{X}T_aT_aTX\rangle\frac{\partial q}{\partial p_b}(p_1,p_2,p_3)\\
 & \hspace{2.1cm} + 2\sum_{a,b,c=1}^3\langle T_c,\overline{X}ZT_aTX\rangle\langle T_b,\overline{X}T_aX\rangle\frac{\partial^2 q}{\partial p_b\partial p_c}(p_1,p_2,p_3)\\
 ={}& -6\sum_{b=1}^3\partial_Tp_b(X,Z)\frac{\partial q}{\partial p_b}(p_1,p_2,p_3)\\
 & \hspace{3.7cm} + 2\sum_{b,c=1}^3\langle T_c,\overline{X}ZXT_b\overline{X}TX\rangle\frac{\partial^2 q}{\partial p_b\partial p_c}(p_1,p_2,p_3)\\
 ={}& -6\partial_Tq(p_1,p_2,p_3) + 2\sum_{b,c=1}^3\langle T_b\overline{X}ZXT_c,\overline{X}TX\rangle\frac{\partial^2 q}{\partial p_b\partial p_c}(p_1,p_2,p_3).
\end{align*}
Now note that $Z_1Z_2Z_3=\langle Z_1,Z_3\rangle Z_2-\langle Z_1,Z_2\rangle Z_3-\langle Z_3,Z_2\rangle Z_1$ modulo $\R{\bf 1}$ for $Z_1,Z_2,Z_3\in\Im\HH$, so that
$$ \langle T_b\overline{X}ZXT_c,\overline{X}TX\rangle = \delta_{b,c}\abs{X}^4\langle Z,T\rangle - \langle T_b,\overline{X}ZX\rangle\langle T_c,\overline{X}TX\rangle - \langle T_b,\overline{X}TX\rangle\langle T_c,\overline{X}ZX\rangle. $$
This gives
\begin{multline*}
 \sum_{b,c=1}^3\langle T_b\overline{X}ZXT_c,\overline{X}TX\rangle\frac{\partial^2 q}{\partial p_b\partial p_c}(p_1,p_2,p_3)\\
 =\langle Z,T\rangle\abs{X}^4(\Delta_pq)(p_1,p_2,p_3) - 2\partial_Z\partial_Tq(p_1,p_2,p_3) = -2(\ell-1)\partial_Tq(p_1,p_2,p_3)
\end{multline*}
and hence
$$ R_Tq(p_1,p_2,p_3) = -2(2\ell+1)\partial_Tq(p_1,p_2,p_3). $$
Further, since $J_{T_a}J_T=J_{T'}-16\langle T,T_a\rangle{\rm id}_{\mathfrak{v}}$ for some $T'\in\Im\F$ and $J_{T'}\in\mathfrak{so}(\mathfrak{v})$, we have
$$ \partial_{J_{T_a}J_TX}\abs{X}^{2i} = -32i\langle T,T_a\rangle\abs{X}^{2i}. $$
Then
\begin{align*}
 & R_T(\abs{X}^{2i}\abs{Z}^{2j}q(p_1,p_2,p_3))\\
 ={}& \frac{1}{16}\sum_{a=1}^3\partial_{J_{T_a}J_TX}\Big(2j\langle Z,T_a\rangle\abs{X}^{2i}\abs{Z}^{2j-2}q(p_1,p_2,p_3)+\abs{X}^{2i}\abs{Z}^{2j}\partial_{T_a}q(p_1,p_2,p_3)\Big)\\
 ={}& -2i\abs{X}^{2i}\sum_{a=1}^3\langle T,T_a\rangle\Big(2j\langle Z,T_a\rangle\abs{Z}^{2j-2}q(p_1,p_2,p_3)+\abs{Z}^{2j}\partial_{T_a}q(p_1,p_2,p_3)\Big)\\
 & +4j\abs{X}^{2i}\abs{Z}^{2j-2}\sum_{a,b=1}^3\langle Z,T_a\rangle\langle T_b,\overline{X}ZT_aTX\rangle\frac{\partial q}{\partial p_b}(p_1,p_2,p_3) + \abs{X}^{2i}\abs{Z}^{2j}R_Tq(p_1,p_2,p_3)\\
 ={}& -4ij\abs{X}^{2i}\abs{Z}^{2j-2}\langle Z,T\rangle q(p_1,p_2,p_3) - 2(i+2j+2\ell+1)\abs{X}^{2i}\abs{Z}^{2j}\partial_Tq(p_1,p_2,p_3),
\end{align*}
where we have used $\sum_{a=1}^3\langle Z,T_a\rangle\langle T_b,\overline{X}ZT_aTX\rangle=\langle T_b,\overline{X}Z^2TX\rangle=-\abs{Z}^2\partial_Tp_b(X,Z)$ in the last step.

Ad (iv): This follows simply by applying Lemma~\ref{lemma:xLaplaceOnGeneralTerm} twice.

Ad (v): Apply the product rule and Lemma~\ref{lemma:SP(1,2)_basic_calc:ii}.
\end{proof}

\begin{lemma}
\label{lemma:SP(1,2)_key_point}
Let $0\neq q\in\mathcal{H}^\ell(\R^3)$, then $\partial_{J_{T}X}q(p_1,p_2,p_3)=0$ for all $T\in\Im\HH$ if and only if $\ell=0$.
\end{lemma}

\begin{proof}
We have
\begin{align*}
 J_{T_1}X &= 4(X_2S_1-X_1S_2+X_4S_3-X_3S_4)\\
 J_{T_2}X &= 4(X_3S_1-X_4S_2-X_1S_3+X_2S_4)\\
 J_{T_3}X &= 4(X_4S_1+X_3S_2-X_2S_3-X_1S_4).
\end{align*}
Since $\Delta_pq=0$ implies $\Delta_{\mathfrak{v}} q(p_1,p_2,p_3)=0$ by Lemma~\ref{lemma:SP(1,2)_basic_calc:i} we have
\begin{align*}
0 &= \frac{1}{4}\left(\frac{\partial}{\partial X_2}\partial_{J_{T_1}X}+\frac{\partial}{\partial X_3}\partial_{J_{T_2}X}+\frac{\partial}{\partial X_4}\partial_{J_{T_3}X}\right)q(p_1,p_2,p_3)\\
 &=\bigg(-X_1\left(\frac{\partial^2}{\partial X_2^2}+\frac{\partial^2}{\partial X_3^2}+\frac{\partial^2}{\partial X_4^2}\right)+\left(X_2\frac{\partial}{\partial X_2}+X_3\frac{\partial}{\partial X_3}+X_4\frac{\partial}{\partial X_4}\right)\frac{\partial}{\partial X_1}\\
 & \hspace{9.5cm}+3\frac{\partial}{\partial X_1}\bigg)q(p_1,p_2,p_3)\\
 &= \left(E_\mathfrak{v}+3\right)\frac{\partial}{\partial X_1}q_j(p_1,p_2,p_3) = (2\ell+2)\frac{\partial}{\partial X_1}q_j(p_1,p_2,p_3).
\end{align*}
Hence $\frac{\partial}{\partial X_1}q(p_1,p_2,p_3)$ must vanish, which implies $q=0$ or $\ell=0$.
\end{proof}

\begin{prop}
\label{prop:DiffSBOs_SP(1,2)}
For $(G,G')=({\rm Sp}(1,2),{\rm Sp}(1,1))$ and $(\lambda,\nu)\in\LocalPoles$ with $\lambda+\rho-\nu-\rho'=-2k$ we have
$$ \C[\nbar]_{\lambda,\nu} = \begin{cases}\C\widehat{u}_{\lambda,\nu}^C\oplus\mathcal{H}^{\frac{k}{2}}(p_1,p_2,p_3) & \text{for $k>0$ even and $\nu+\rho'=k+4$,}\\\C\widehat{u}_{\lambda,\nu}^C & \text{otherwise.}\end{cases} $$
\end{prop}

Note that for $k=0$ we have $\C\widehat{u}_{\lambda,\nu}^C=\mathcal{H}^0(p_1,p_2,p_3)=\C 1$.

\begin{proof}
As in the proof of Proposition~\ref{prop:DiffSBOsSp(1,n)} we can write $\widehat{u}\in\C[\nbar]_{\lambda,\nu}$ as
$$ \widehat{u} = \sum_{i+2j+2\ell=k} \abs{X}^{2i}\abs{Z}^{2j}q_\ell^{i,j}(p_1,p_2,p_3) $$
with $q_\ell^{i,j}\in\mathcal{H}^\ell(\R^3)$. By Lemma~\ref{lemma:SP(1,2)_complex_calc} the differential operator $-\mathcal{F}(D_{\mathfrak{z}}(T))$ in Proposition~\ref{prop:fourier_equations} applied to $\widehat{u}$ takes the following form:
\begin{align*}
 & \sum_{i+2j+2\ell=k}\Bigg[2j(\nu+\rho'+2i+2\ell-2)\abs{X}^{2i}\abs{Z}^{2j-2}\langle Z,T\rangle\\
 & \hspace{3cm} +(\nu+\rho'+2i+2\ell-2)\abs{X}^{2i}\abs{Z}^{2j}\partial_T\\
 & \hspace{3cm} -i(i+2\ell+1)\abs{X}^{2i-2}\abs{Z}^{2j}\partial_{J_TX}-4ij\abs{X}^{2i}\abs{Z}^{2j-2}\langle Z,T\rangle\\
 & \hspace{3cm} -2(i+2j+2\ell+1)\abs{X}^{2i}\abs{Z}^{2j}\partial_T\\
 & \hspace{3cm} -16i(i-1)(i+2\ell)(i+2\ell+1)\abs{X}^{2i-4}\abs{Z}^{2j}\langle Z,T\rangle\\
 & \hspace{3cm} -2j(2j+2\ell+1)\abs{X}^{2i}\abs{Z}^{2j-2}\langle Z,T\rangle\Bigg]q_\ell^{i,j}(p_1,p_2,p_3).
\end{align*}
After rearrangement we obtain
\begin{align}
 & \sum_{i+2j+2\ell=k}(\nu+\rho'-4j-2\ell-4)\abs{X}^{2i}\abs{Z}^{2j}\partial_Tq_\ell^{i,j}(p_1,p_2,p_3)\notag\\
 & -\sum_{i+2j+2\ell=k-1}(i+1)(i+2\ell+2)\abs{X}^{2i}\abs{Z}^{2j}\partial_{J_TX}q_\ell^{i+1,j}(p_1,p_2,p_3)\notag\\
 & +\sum_{i+2j+2\ell=k-2}\abs{X}^{2i}\abs{Z}^{2j}\Bigg[2(j+1)(\nu+\rho'-2j-5)\langle Z,T\rangle q_\ell^{i,j+1}(p_1,p_2,p_3)\notag\\
 & \hspace{2.5cm} -16(i+1)(i+2)(i+2\ell+2)(i+2\ell+3)\langle Z,T\rangle q_\ell^{i+2,j}(p_1,p_2,p_3)\Bigg].\label{eq:ProofSp(1,2)1}
\end{align}
As in the proof of Proposition~\ref{prop:DiffSBOsSp(1,n)} we note that for $q\in\mathcal{H}^\ell(\R^3)$ the two polynomials $\partial_Tq(p_1,p_2,p_3)$ and $\partial_{J_TX}q(p_1,p_2,p_3)$ are harmonic in $X$ and $Z$, the latter one due to Lemma~\ref{lem:LaplacianCommutators}. We further decompose
$$ \langle Z,T\rangle q(p_1,p_2,p_3) = \underbrace{\langle Z,T\rangle q(p_1,p_2,p_3)-\frac{\abs{Z}^2}{2\ell+1}\partial_Tq(p_1,p_2,p_3)}_{q(p_1,p_2,p_3)_+:=} + \frac{\abs{Z}^2}{2\ell+1}\partial_Tq(p_1,p_2,p_3), $$
then \eqref{eq:ProofSp(1,2)1} rewrites as
\begin{align}
 & \sum_{i+2j+2\ell=k}\abs{X}^{2i}\abs{Z}^{2j}\frac{1}{2\ell+1}\Bigg[(2\ell+1)(\nu+\rho'-4j-2\ell-4)\partial_Tq_\ell^{i,j}(p_1,p_2,p_3)\notag\\
 & \hspace{3cm} +2j(\nu+\rho'-2j-3)\partial_Tq_\ell^{i,j}(p_1,p_2,p_3)\notag\\
 & \hspace{3cm} -16(i+1)(i+2)(i+2\ell+2)(i+2\ell+3)\partial_Tq_\ell^{i+2,j-1}(p_1,p_2,p_3)\Bigg]\notag\\
 & -\sum_{i+2j+2\ell=k-1}(i+1)(i+2\ell+2)\abs{X}^{2i}\abs{Z}^{2j}\partial_{J_TX}q_\ell^{i+1,j}(p_1,p_2,p_3)\notag\\
 & +\sum_{i+2j+2\ell=k-2}\abs{X}^{2i}\abs{Z}^{2j}\Bigg[2(j+1)(\nu+\rho'-2j-5)q_\ell^{i,j+1}(p_1,p_2,p_3)_+\notag\\
 & \hspace{3cm} -16(i+1)(i+2)(i+2\ell+2)(i+2\ell+3)q_\ell^{i+2,j}(p_1,p_2,p_3)_+\Bigg].\label{eq:ProofSp(1,2)2}
\end{align}
Now all terms inside square brackets are both harmonic in $X$ and $Z$ and we can apply the uniqueness of the Fischer decomposition. Further, a similar argument as in the proof of Proposition~\ref{prop:DiffSBOsSp(1,n)} comparing degrees then shows that for fixed $i$ and $j$ each square bracket has to vanish separately for the respective value of $\ell$. The second square brackets vanish if and only if for all $i$ and $j$
$$ i(i+2\ell+1)\partial_{J_TX}q_\ell^{i,j}(p_1,p_2,p_3) = 0. $$
By Lemma~\ref{lemma:SP(1,2)_key_point} this implies that $q_\ell^{i,j}=0$ whenever $i>0$ and $\ell>0$. If we assume that $\ell=0$, then $q_\ell^{i,j}=c_{i,j}\in\mathcal{H}^0(\R^3)=\C$ is scalar and we obtain the distributions $\widehat{u}_{\lambda,\nu}^C$. We therefore assume $\ell>0$ for the rest of the proof. Since $q(p_1,p_2,p_3)_+=0$ if and only if $q=0$, the last square bracket vanishes if and only if for all $i$ and $j$:
\begin{equation}
 2(j+1)(\nu+\rho'-2j-5)q_\ell^{i,j+1} = 16(i+1)(i+2)(i+2\ell+2)(i+2\ell+3)q_\ell^{i+2,j}.\label{eq:ProofSp(1,2)3}
\end{equation}
Plugging this into the first square bracket further gives
\begin{equation}
 (\nu+\rho'-4j-2\ell-4)\partial_Tq_\ell^{i,j}(p_1,p_2,p_3) = 0.\label{eq:ProofSp(1,2)4}
\end{equation}
Now for $i=2$ we have $q_\ell^{2,j}=0$ by our previous considerations and the right hand side of \eqref{eq:ProofSp(1,2)3} vanishes. Thus, either $q_\ell^{0,j+1}=0$ or $\nu+\rho'-2j-5=0$. In the latter case, \eqref{eq:ProofSp(1,2)4} implies that $(2j+2\ell+3)\partial_Tq_\ell^{0,j+1}(p_1,p_2,p_3)=0$ and hence $q_\ell^{0,j+1}=0$ as $\ell>0$.

Summarizing, we have shown that for $\ell>0$ we have $q_\ell^{i,j}=0$ whenever $i>0$ or $j>0$. The only possible solution in addition to $\widehat{u}_{\lambda,\nu}^C$ is therefore
$$ \widehat{u} = q_\ell^{0,0}(p_1,p_2,p_3) $$
with $2\ell=k>0$. This polynomial indeed satisfies \eqref{eq:ProofSp(1,2)3}, but it only satisfies \eqref{eq:ProofSp(1,2)4} if $\nu+\rho'=2\ell+4=k+4$. 
\end{proof}

\newpage
\part{Classification of symmetry breaking operators}\label{part:classification}

Using the classification of differential symmetry breaking operators from the previous part, we now complete the classification of all symmetry breaking operators between spherical principal series representations for strongly spherical pairs of the form $(G,G')=(\Urm(1,n+1;\F),\Urm(1,m+1;\F)\times F)$ with $0\leq m<n$ and $F<\Urm(n-m;\F)$.

\section{Solutions outside the origin}\label{sec:solutions_outside_origin}

We first determine the space $\mathcal{D}'(\nbar \setminus \{0\})_{\lambda,\nu}$ of invariant distributions defined outside the origin.

\begin{prop}
\label{prop:solutions}
For any $(\lambda,\nu) \in \C^2$, the space $\mathcal{D}'(\nbar \setminus \{0\})_{\lambda,\nu}$ is spanned by the non-zero distribution
$$ \frac{1}{\Gamma(\frac{\lambda+\rho+\nu-\rho'}{2})} N(X,Z)^{-2(\nu+\rho')}\abs{X''}^{\lambda-\rho+\nu+\rho'}.$$
\end{prop}

\begin{remark}
\label{remark:residue_solution_outside_origin}
By Lemma~\ref{lemma:homogenous_distributions}, we have for $(\lambda,\nu)\in \LocalPoles$ and $l \in \Z_{\geq 0}$ given by $\lambda+\rho+\nu-\rho'=-2l$: $$\frac{1}{\Gamma(\frac{\lambda+\rho+\nu-\rho'}{2})}\abs{X''}^{\lambda-\rho+\nu+\rho'}= (-1)^l \frac{\pi^{\frac{p''}{2}}}{2^{2l}\Gamma(\frac{p''+2l}{2})}  (\Delta_{\mathfrak{v}''}^l \delta)(X'').  $$
Otherwise $$\Supp \frac{1}{\Gamma(\frac{\lambda+\rho+\nu-\rho'}{2})}\abs{X''}^{\lambda-\rho+\nu+\rho'} =\mathfrak{v}''.$$
\end{remark}

The strategy to prove this proposition is to use the action $\pi_\lambda(\tilde{w}_0)$ of the longest Weyl group element $\tilde{w}_0$. On the group level this corresponds to switching between the open dense subsets $\Nbar MAN$ and $\tilde{w}_0\Nbar MAN$ of $G$. The advantage of working on $\tilde{w}_0\Nbar MAN$ is that the left-action of $N'$ on $\tilde{w}_0\Nbar$ is by translations on $\Nbar$.

\begin{proof}[Proof of Proposition~\ref{prop:solutions}]
First note that since $N(X,Z)\neq0$ for all $(X,Z)\in\nbar\setminus\{0\}$, the involution $\pi_{-\lambda}(\tilde{w}_0)$ from Proposition~\ref{prop:longest_weyl_group_element_action} defines a bijection from $\mathcal{D}'(\nbar\setminus\{0\})$ onto itself. We first determine the image of $\mathcal{D}'(\nbar\setminus\{0\})_{\lambda,\nu}$ under $\pi_{-\lambda}(\tilde{w}_0)$. It is easy to see that every $v\in\pi_{-\lambda}(\tilde{w}_0)(\mathcal{D}'(\nbar\setminus\{0\})_{\lambda,\nu})$ is $M'$-invariant and homogeneous of degree $\lambda-\rho+\nu+\rho'$. Further, on homogeneous distributions the differential operators $D_{\mathfrak{v}}(S)$ and $D_{\mathfrak{z}}(T)$ are given by
$$ D_{\mathfrak{v}}(S) = \pi_{-\lambda}(\tilde{w}_0)d\pi_{-\lambda}(S)\pi_{-\lambda}(\tilde{w}_0) \qquad \mbox{and} \qquad D_{\mathfrak{z}}(T) = \pi_{-\lambda}(\tilde{w}_0)d\pi_{-\lambda}(T)\pi_{-\lambda}(\tilde{w}_0) $$
with $d\pi_{-\lambda}(S)=-\partial_S-\frac{1}{2}\partial_{[S,X]}$ and $d\pi_{-\lambda}(T)=-\partial_T$ (see the proof of Proposition~\ref{prop:LieAlgebraAction}). Hence, every $v\in\pi_{-\lambda}(\tilde{w}_0)(\mathcal{D}'(\nbar\setminus\{0\})_{\lambda,\nu})$ satisfies $\partial_Tv=0$ for all $T\in\mathfrak{z}$ and $(\partial_S+\frac{1}{2}\partial_{[S,X]})v=0$ for all $S\in\mathfrak{v}'$. This implies that $v(X,Z)=v(X'')$ is independent of $X'$ and $Z$. By Remark~\ref{rem:TransUnitSphere} the group $M'$ acts transitively on the unit sphere in $\mathfrak{v}''$, so the $M'$-invariance and the homogeneity condition imply by Lemma~\ref{lemma:homogenous_distributions} that
$$ \pi_{-\lambda}(\tilde{w}_0)(\mathcal{D}'(\nbar\setminus\{0\})_{\lambda,\nu}) = \C \frac{\abs{X''}^{\lambda-\rho+\nu+\rho'}}{\Gamma(\frac{\lambda+\rho+\nu-\rho'}{2})} . $$
Applying $\pi_{-\lambda}(\tilde{w}_0)$ to this distribution shows the claim.
\end{proof}

\section{Analytic continuation of invariant distributions}

For $\Re(\nu)\ll 0$ and $\Re(\lambda+\nu)\gg0$ we define the following locally integrable function on $\nbar$:
$$u_{\lambda,\nu}^A(X,Z):=\frac{1}{\Gamma(\frac{\lambda+\rho-\nu-\rho'}{2})\Gamma(\frac{\lambda+\rho+\nu-\rho'}{2})} N(X,Z)^{-2(\nu+\rho')}\abs{X''}^{\lambda-\rho+\nu+\rho'}.$$
In this section we will prove:

\begin{theorem}
\label{lemma:holomorphic_cont_solution}
$u_{\lambda,\nu}^A$ extends to a family of distributions that depends holomorphically on $(\lambda,\nu)\in \C^2$ and $u_{\lambda,\nu}^A\in \mathcal{D}'(\nbar)_{\lambda,\nu}$ for all $(\lambda,\nu)\in\C^2$.
\end{theorem}

To prove this statement we use polar coordinates on the H-type Lie algebra $\nbar$. For this let
$$ \mathbb{S} := \{(\omega,\eta)\in\mathfrak{v}\oplus\mathfrak{z}=\nbar:N(\omega,\eta)=1\}. $$
By Lemma~\ref{lemma:polar_coordinates_H_type} there exists a unique smooth measure $d\mathbb{S}$ on $\mathbb{S}$ so that for all $\varphi\in C_c(\nbar)$ we have
$$ \int_{\nbar}\varphi(X,Z)\,d(X,Z) = \int_{\R_+}\int_{\mathbb{S}}\varphi(r\omega,r^2\eta)r^{2\rho-1}\,d\mathbb{S}(\omega,\eta)\,dr, $$
where $d(X,Z)$ denotes the Lebesgue measure on $\nbar$ normalized by the inner product.

For $(\omega,\eta)\in\mathbb{S}$ we write $\omega''=(\omega_{p'+1},\ldots,\omega_p)\in\mathfrak{v}''$.

\begin{lemma}
	\label{lemma:poles_on_sphere}
	$\abs{\omega''}^\lambda$ defines a meromorphic family of generalized functions on $\mathbb{S}$ with simple poles for $\lambda\in-p''-2\Z_{\geq 0}$.
\end{lemma}

\begin{proof}
	Let $\varphi \in C_c^\infty(\mathbb{S})$ be a test function. Further, choose $\chi \in C_c^\infty(\R_+)$ with $$\int_{\R_+}\chi(r)r^{p+2q-1}\,dr=1.$$
	Then we can write
	$$
	\int_{\mathbb{S}}\abs{\omega''}^\lambda \varphi(\omega,\eta)\,d\mathbb{S}(\omega,\eta)=
	\int_{\R_+}r^{p+2q-1}\int_{\mathbb{S}}\abs{r\omega''}^\lambda\cdot r^{-\lambda}\chi(r)\varphi(\omega, \eta)\,d\mathbb{S}(\omega,\eta)\,dr.
	$$
	The formula $\tilde{\varphi}_\lambda(r\omega,r^2\eta):= r^{-\lambda}\chi(r)\varphi(\omega,\eta)$ defines a smooth compactly supported function on $\R_+\times\mathbb{S}\cong\R^{p+q}\setminus\{0\}$, i.e. $\tilde{\varphi}_\lambda\in C_c^\infty(\R^{p+q}\setminus \{0\}) \subseteq C_c^\infty(\nbar)$, that depends holomorphically on $\lambda\in\C$. We then have
	\begin{equation*}
	\label{eq:absolut_value_distr_on_sphere}
	\int_{\mathbb{S}}\abs{\omega''}^\lambda \varphi(\omega,\eta)\,d\mathbb{S}(\omega,\eta)=
	\int_{\R^{p+q}}\abs{X''}^\lambda \tilde{\varphi}_\lambda(X,Z)\,d(X,Z),
	\end{equation*}
	where $X=(X',X'')\in \R^{p'}\times \R^{p''}\cong\R^p$. Now the claim follows from Lemma~\ref{lemma:absolute_value_distr_holomorphic}.
\end{proof}

\begin{proof}[Proof of Theorem~\ref{lemma:holomorphic_cont_solution}]
Let $\varphi\in C_c^\infty(\nbar)$ be a test function. 
By Lemma~\ref{lemma:polar_coordinates_H_type}
\begin{multline}
\label{eq:polar_coordinates_aplication}
 \langle u_{\lambda,\nu}^A,\varphi \rangle = \frac{1}{\Gamma(\frac{\lambda+\rho-\nu-\rho'}{2})\Gamma(\frac{\lambda+\rho+\nu-\rho'}{2})}\int_{\R_+}r^{\lambda+\rho-\nu-\rho'-1}\\
 \times\int_{\mathbb{S}}\abs{\omega''}^{\lambda-\rho+\nu+\rho'}\varphi((r\omega,r^2\eta))\, d\mathbb{S}(\omega,\eta)\,dr.
\end{multline}
which extends holomorphically to an entire function in $\lambda$ and $\nu$ by Lemma~\ref{lemma:absolute_value_distr_holomorphic} and Lemma~\ref{lemma:poles_on_sphere} since the inner integral is an even function of $r$.
By Proposition~\ref{prop:solutions}, the restriction of $u_{\lambda,\nu}^A$ to $\nbar\setminus\{0\}$ is contained in $\mathcal{D}'(\nbar\setminus\{0\})_{\lambda,\nu}$. But for $\Re\nu \ll 0$ and $\Re(\lambda+\nu)\gg0$ the distribution is sufficiently regular, so that the differential equations hold on all of $\nbar$. Then $u_{\lambda,\nu}^A\in\mathcal{D}'(\nbar)_{\lambda,\nu}$ for $\Re\nu \ll 0$ and $\Re(\lambda+\nu)\gg0$ and hence for all $(\lambda,\nu)\in \C^2$ by analytic continuation.
\end{proof}

\begin{remark}
The normalizing factor $\Gamma(\frac{\lambda+\rho-\nu-\rho'}{2})\Gamma(\frac{\lambda+\rho+\nu-\rho'}{2})$ is chosen such that it has a simple pole at $(\lambda,\nu)\in (\GlobalPoles \cup \LocalPoles ) \setminus \IntersectionPoles$ and a pole of second order at $(\lambda,\nu) \in \IntersectionPoles$.
\end{remark}

\begin{remark}
That the distributions $N(X,Z)^{-2(\nu+\rho')}\abs{X''}^{\lambda-\rho+\nu+\rho'}$ extend meromorphically in $(\lambda,\nu)\in \C^2$ also follows from \cite{MOO16}, but Theorem~\ref{lemma:holomorphic_cont_solution} gives the precise normalization factor which makes the family of distributions depend holomorphically on $(\lambda,\nu)\in\C^2$.
\end{remark}

In addition, from \eqref{eq:polar_coordinates_aplication} as well as Remark~\ref{remark:residue_solution_outside_origin} and Lemma~\ref{lemma:absolute_value_distr_holomorphic} we immediately obtain:

\begin{corollary}
\label{corollary:support_in_0}
For $(\lambda,\nu)\in\C^2\setminus\GlobalPoles$ we have
$$\Supp u_{\lambda,\nu}^A= \begin{cases}
\nbar & \text{for } (\lambda,\nu)\in \C^2\setminus(\GlobalPoles \cup \LocalPoles),\\
\nbar' & \text{for } (\lambda,\nu)\in \LocalPoles \setminus \IntersectionPoles.
\end{cases}$$
For $(\lambda,\nu)\in \GlobalPoles$ we have $\Supp u_{\lambda,\nu}^A \subseteq \{0\}$.
\end{corollary}

In the following we write $A_{\lambda,\nu}\in\Hom_{G'}(\pi_\lambda|_{G'},\tau_\nu)$ for the symmetry breaking operator corresponding to $u_{\lambda,\nu}^A$ via Theorem~\ref{theorem:KS_kernel_open_cell}.

\section{Residues at the origin}\label{sec:differential_operators}

Let $(\lambda,\nu)\in \GlobalPoles$ with $\lambda+\rho-\nu-\rho'=-2k\in-2\Z_{\geq 0}$. In Section~\ref{sec:DiffOpClass1} we found a holomorphic family of polynomials $\widehat{u}_{\lambda,\nu}^C\in\C[\nbar]_{\lambda,\nu}$. The corresponding distributions $u_{\lambda,\nu}^C$ are for $m>0$ given by
\begin{equation}
\label{eq:definition_u^C1}
u^C_{\lambda,\nu} = \sum_{h+i+2j=k} \frac{2^{-2i-2h}\Gamma(\frac{2\nu+p'+2}{4})\Gamma(\frac{\lambda+\rho+\nu-\rho'}{2}+i)}{h!i!j!\Gamma(\frac{2\nu+p'+2}{4}-j)\Gamma(\frac{p''}{2}+i)\Gamma(\frac{\lambda+\rho+\nu-\rho'}{2})} \Delta_{\mathfrak{v}'}^h \Delta_{\mathfrak{v}''}^i \square^j\delta,
\end{equation}
and for $m=0$ by
\begin{equation}\label{eq:definition_u^C2}
 u^C_{\lambda,\nu} = \sum_{i+2j=k} \frac{2^{-i}\Gamma(\frac{\nu}{2}-j)}{i!j!\Gamma(\frac{\nu}{2}-\lfloor\frac{k}{2}\rfloor)\Gamma(\frac{p}{2}+i)} \Delta_{\mathfrak{v}}^i \square^j\delta.
\end{equation}

In this section we obtain $u_{\lambda,\nu}^C$ as a renormalization of the holomorphic family $u_{\lambda,\nu}^A$ and show that it is the only differential symmetry breaking operator that occurs as a renormalization of the family $A_{\lambda,\nu}$. Note that in Part~\ref{part:diff}, $u_{\lambda,\nu}^C$ has been obtained combinatorially as solution of the differential equations of Lemma~\ref{lemma:diff_equations}.

\begin{theorem}
\label{thm:u^C}
Let $(\lambda,\nu)\in \GlobalPoles$ and $k\in \Z_{\geq 0}$ given by $\lambda+\rho-\nu-\rho'=-2k$.
\begin{enumerate}[label=(\roman{*}), ref=\thetheorem(\roman{*})]
\item 
\label{prop:diff_op_residue}
$u_{\lambda,\nu}^C\in \mathcal{D}'(\nbar )_{\lambda,\nu}$ and the following residue formula holds:
$$ 
u_{\lambda,\nu}^A = \frac{(-1)^k k! \pi^{\frac{p+q}{2}}}{\Gamma(\frac{\nu+\rho'}{2})}
 \times
 \begin{cases} \frac{\Gamma(\frac{2\nu+p'}{4})}{\Gamma(\frac{2\nu+p'}{2})}u^C_{\lambda,\nu}&\mbox{for $m>0$,}\\\frac{\Gamma(\frac{\nu}{2}-\lfloor\frac{k}{2}\rfloor)}{2^k\Gamma(\nu-k)}u^C_{\lambda,\nu}&\mbox{for $m=0$.} \end{cases} $$
\item \label{corollary:diff_op_non_zero}
The distribution $u_{\lambda,\nu}^C$ is non-zero for all $(\lambda,\nu)\in\GlobalPoles$ and $\Supp u_{\lambda,\nu}^C=\{0\}$.
\end{enumerate}
\end{theorem}

\begin{proof}
First, $\Supp u_{\lambda,\nu}^C \subseteq \{0\}$ by definition. 
To prove $u^C_{\lambda,\nu}\neq 0$ we consider the case $m>0$ first.
Here the summand of $u^C_{\lambda,\nu}$ for $i=j=0$ is given by
$$ \frac{1}{2^{2k}k! \Gamma(\frac{p''}{2})}\Delta_{\mathfrak{v}'}^k\delta, $$
which is always non-zero. 
Since the distributions $\Delta_{\mathfrak{v}'}^h \Delta_{\mathfrak{v}''}^i \square^j\delta$, $h,i,j\in \Z_{\geq 0}$, are linearly independent, it follows that $u^C_{\lambda,\nu} \neq 0$. 
For $m=0$ one shows that the summand of $u^C_{\lambda,\nu}$ for $j=\lfloor\frac{k}{2}\rfloor$ does not vanish, so we have proven (ii).

To show (i) we abbreviate $\gamma := \lambda-\rho+\nu+\rho'$ and $\gamma':=-2(\nu+\rho')$. We prove the claimed residue formula for $\Re (\lambda+\nu) \gg 0$, the general case then follows by analytic continuation.

By Lemma~\ref{lemma:integral_formula_polar_coordinates} and Lemma~\ref{lemma:homogenous_distributions} we have for $2\rho+\gamma+\gamma'=-2k\in -2\Z_{\geq 0}$:
\begin{multline*}
(-1)^k\frac{2(2k)!}{k!\Gamma(\frac{2\rho+\gamma+\gamma'}{2})}\int_{\nbar} N(X,Z)^{\gamma'}\abs{X''}^\gamma \varphi(X,Z) \,d\nbar\\
=  
(-1)^k\frac{4(2k)!}{k!\Gamma(\frac{2\rho+\gamma+\gamma'}{2})} \int_{\R_+}r^{2\rho+\gamma+\gamma'-1} \int_{0}^{1}x^{p+\gamma-1}(1-x^4)^{\frac{q}{2}-1} \\ \times \int_{S^{p-1}}\abs{\omega ''}^\gamma\int_{S^{q-1}}  \varphi(rx\omega, r^2\sqrt{1-x^4}  \eta)  \; d\eta \;d\omega \,dx \, dr.
\end{multline*}
Since the integral over $x,\omega$ and $\eta$ is an even function in $r$, by Lemma~\ref{lemma:absolute_value_distr_holomorphic}:
$$
=  
\frac{d^{2k}}{dr^{2k}}\biggr|_{r=0} 2 \int_0^{1}x^{p+\gamma-1} (1-x^4)^{\frac{q}{2}-1}\int_{S^{p-1}}\abs{\omega ''}^\gamma\int_{S^{q-1}}  \varphi(rx\omega, r^2 \sqrt{1-x^4} \eta)  \, d\eta \,d\omega \,dx.
$$
By Lemma~\ref{lemma:combinatorics_deriviative} and by substituting $x^4=y$:
\begin{multline*}
=
\sum_{i+2j=2k} \frac{(2k)!}{2}\sum_{\abs{\alpha}=i}\frac{1}{\alpha!}\sum_{\abs{\beta}=j}\frac{1}{\beta!} \int_0^{1}y^{\frac{\gamma+p+i}{4}-1} (1-y)^{\frac{q+j}{2}-1}\,dy
\\ \times\int_{S^{p-1}}\omega^\alpha\abs{\omega ''}^\gamma \,d\omega
 \int_{S^{q-1}} \eta^\beta  \,d\eta
\frac{\partial^{i+j}}{\partial X^\alpha \partial Z^\beta} \varphi(0,0).
\end{multline*}
Since the integrals over the spheres vanish for odd length multi-indices $\alpha$ and $\beta$ we obtain by evaluation of the $x$ integral:
\begin{multline*}
=
\sum_{i+2j=k}\frac{(2k)!}{2}\sum_{\abs{\alpha}=i}\frac{1}{(2\alpha)!}\sum_{\abs{\beta}=j}\frac{1}{(2\beta)!} B\left( \frac{\gamma+p+2i}{4},\frac{q}{2} +j\right)
\\  
\times\int_{S^{p-1}}\omega^{2\alpha}\abs{\omega ''}^\gamma \,d\omega
 \int_{S^{q-1}} \eta^{2\beta}  \,d\eta
\frac{\partial^{2i+2j}}{\partial X^{2\alpha} \partial Z^{2\beta}} \varphi(0,0).
\end{multline*}
Evaluating the integrals over the spheres with Lemma~\ref{lemma:int_techniques} we obtain:
\begin{align*}
 ={}& \sum_{i+2j=k}2^{-2j-2i+1}(2k)!\sum_{\abs{\alpha}=i}\sum_{\abs{\beta}=j} B\left( \frac{\gamma+p+2i}{4},\frac{q}{2} +j\right)\frac{\pi^{\frac{p}{2}}\Gamma(\frac{\gamma+p''}{2}+\abs{\alpha''})}{\alpha!\Gamma(\frac{\gamma+p}{2}+i)\Gamma(\frac{p''}{2}+\abs{\alpha''})}\\ 
 & \hspace{8cm}\times \frac{\pi^{\frac{q}{2}}}{\beta!\Gamma(\frac{q}{2}+j)}\frac{\partial^{2i+2j}}{\partial X^{2\alpha} \partial Z^{2\beta}}\varphi(0,0)\\
 ={}& \pi^{\frac{p+q}{2}}\sum_{i+2j=k}\sum_{\abs{\alpha}=i}\sum_{\abs{\beta}=j} \frac{2^{-2j-2i+1}(2k)!\Gamma( \frac{\gamma+p+2i}{4})\Gamma(\frac{\gamma+p''}{2}+\abs{\alpha''})}{\alpha!\beta!\Gamma(\frac{\gamma+p}{2}+i)\Gamma(\frac{p''}{2}+\abs{\alpha''})\Gamma(\frac{\gamma+p+2q+2k}{4})}\frac{\partial^{2i+2j}}{\partial X^{2\alpha} \partial Z^{2\beta}} \varphi(0,0)\\
 ={}& \pi^{\frac{p+q}{2}}\sum_{h+i+2j=k}\sum_{\abs{\alpha'}=h}\sum_{\abs{\alpha''}=i}\sum_{\abs{\beta}=j} \frac{2^{-2j-2i-2h+1}(2k)!\Gamma( \frac{\gamma+p+2i+2h}{4})\Gamma(\frac{\gamma+p''}{2}+i)}{\alpha'!\alpha''!\beta!\Gamma(\frac{\gamma+p}{2}+i+h)\Gamma(\frac{p''}{2}+i)\Gamma(\frac{\gamma+p+2q+2k}{4})}\\ 
 & \hspace{8.3cm} \times\frac{\partial^{2h+2i+2j}}{\partial X'^{2\alpha'} \partial X''^{2\alpha''} \partial Z^{2\beta}} \varphi(0,0)\\
 ={}& \left\langle \pi^{\frac{p+q}{2}}\sum_{h+i+2j=k} \frac{2^{-2j-2i-2h+1}(2k)!\Gamma( \frac{\gamma+p+2i+2h}{4})\Gamma(\frac{\gamma+p''}{2}+i)}{h!i!j!\Gamma(\frac{\gamma+p}{2}+i+h)\Gamma(\frac{p''}{2}+i)\Gamma(\frac{\gamma+p+2q+2k}{4})} \Delta_{\mathfrak{v}'}^h \Delta_{\mathfrak{v}''}^i \square^j\delta,\varphi \right\rangle,
\end{align*}
where we have used the Multinomial Theorem~\ref{theorem:multinomial_theorem} in the last step. Using the duplication identity $\Gamma(z)\Gamma(z+\frac{1}{2})=\sqrt{\pi}2^{1-2z}\Gamma(2z)$ the final expression can be brought into the claimed form.\qedhere
\end{proof}

\begin{remark}
Although we exclude the case $\F=\R$ (i.e. $q=0$) in our considerations, the computations in Theorem~\ref{prop:diff_op_residue} also remain valid in this case. The resulting operators $C_{\lambda,\nu}: C^\infty(\R^{n})\to C^\infty(\R^{n-1})$ are the conformally covariant differential operators first dicovered by A. Juhl~\cite{juhl_2009} which have previously been obtained as residues of $A_{\lambda,\nu}$ in \cite{kobayashi_speh_2015}. However, even in the case $q=0$ the proof of Theorem~\ref{prop:diff_op_residue} gives a new and more direct approach to obtaining Juhl's operators as residues of a holomorphic family of distributions than that of \cite{kobayashi_speh_2015}.
\end{remark}

Corollary~\ref{corollary:support_in_0} together with Theorem~\ref{prop:diff_op_residue} implies the following:

\begin{corollary}
\label{cor:kernel_zero_set}
$u_{\lambda,\nu}^A=0$ if and only if $(\lambda,\nu)\in L$.
\end{corollary}

\begin{proof}
By Corollary~\ref{corollary:support_in_0} we know that $u_{\lambda,\nu}^A=0$ implies $(\lambda,\nu)\in \GlobalPoles$. We first consider the case $m>0$. Then by Theorem~\ref{prop:diff_op_residue} we have for $(\lambda,\nu)\in \GlobalPoles$ with $\lambda+\rho-\nu-\rho'=-2k\in -2\Z_{\geq 0}$:
$$ u_{\lambda,\nu}^A = (-1)^k\frac{\pi^{\frac{p+q}{2}}k!\Gamma(\frac{2\nu+p'}{4})}{\Gamma(\frac{\nu+\rho'}{2})\Gamma(\frac{2\nu+p'}{2})}u^C_{\lambda,\nu}. $$
Since $u_{\lambda,\nu}^C \neq 0$ by Theorem~\ref{corollary:diff_op_non_zero}, we have $u_{\lambda,\nu}^A=0$ if and only if the gamma factor vanishes which is by the duplication formula for the gamma function equivalent to $\nu\in-\rho'+q-1-2\Z_{\geq 0}$. Since $(\lambda,\nu)\in \GlobalPoles$ this holds if and only if $(\lambda,\nu)\in L$.\\
Now let $m=0$, then by the same arguments $u_{\lambda,\nu}^A=0$ if and only if the gamma factor
$$ \frac{\Gamma(\frac{\nu}{2}-\lfloor\frac{k}{2}\rfloor)}{\Gamma(\nu-k)\Gamma(\frac{\nu+\rho'}{2})} $$
vanishes. Using the duplication formula this is a constant multiple of
$$ \frac{\Gamma(\frac{\nu}{2}-\lfloor\frac{k}{2}\rfloor)}{\Gamma(\frac{\nu}{2}-\frac{k}{2})\Gamma(\frac{\nu}{2}-\frac{k-1}{2})\Gamma(\frac{\nu+\rho'}{2})}. $$
Now for $k$ even we have $\Gamma(\frac{\nu}{2}-\lfloor\frac{k}{2}\rfloor)=\Gamma(\frac{\nu}{2}-\frac{k}{2})$, so the gamma factor vanishes if and only if $\nu\in(-\rho'-2\Z_{\geq0})\cup((k-1)-2\Z_{\geq0})=(k-1)-2\Z_{\geq0}$. This is equivalent to $(\lambda,\nu)=(-\rho+q-1-(k+2\ell),k-2\ell-1)=(-\rho+q-1-2i,\pm(\rho'-q+1+2j))$ with $2i=k+2\ell\in2\Z_{\geq0}$ and $0\leq j\leq i$. For $k$ odd the arguments are similar.
\end{proof}

\section{Singular symmetry breaking operators}

Let $(\lambda,\nu) \in \LocalPoles$ and $l \in \Z_{\geq 0}$ given by $\lambda+\rho+\nu-\rho'=-2l$. We define the following renormalization factors:
\begin{equation}
\label{eq:singular_renorm_parameter}
c^B(\lambda,\nu):=\begin{cases}
\frac{1}{\Gamma(\frac{\lambda+\rho-\nu-\rho'}{2})} & \text{for $m>0$ and } l \leq\frac{p'}{2},\\
\frac{\Gamma(\frac{2\nu+p'+2}{4})}{\Gamma(\frac{2\nu+p'+2}{4}+ \lfloor \frac{2l-p'+2}{4}\rfloor)\Gamma(\frac{\lambda+\rho-\nu-\rho'}{2})} & \text{for $m>0$ and } l > \frac{p'}{2}, \\
\frac{1}{\Gamma(-\frac{\nu}{2}-\lfloor \frac{l}{2} \rfloor)} & \text{for $m=0$.}
\end{cases}
\end{equation}
For $\Re \nu \ll 0$ we define a family of distributions on $\nbar$ by
$$u_{\lambda,\nu}^B:=c^B(\lambda,\nu)N(X,Z)^{-2(\nu+\rho')}\Delta_{\mathfrak{v}''}^l\delta(X'').$$

\begin{theorem}
\label{thm:singular_family}
Let $(\lambda,\nu)\in \LocalPoles$ and $l \in \Z_{\geq 0}$ given by $\lambda-\rho+\nu+\rho'=-2l$.
\begin{enumerate}[label=(\roman{*}), ref=\thetheorem(\roman{*})]
\item \label{thm:singular_family:i}
$u_{\lambda,\nu}^B$ extends to a family of distributions that depends holomorphically on $\nu\in\C$ and $u_{\lambda,\nu}^B\in \mathcal{D}'(\nbar)_{\lambda,\nu}$ for all $(\lambda,\nu)\in\LocalPoles$.
\item \label{thm:singular_family:ii}
The following residue formula holds:
$$ u_{\lambda,\nu}^A = \frac{(-1)^l\pi^{\frac{p''}{2}}}{2^l\Gamma(\frac{p''}{2}+l)}\times \begin{cases}
u_{\lambda,\nu}^B & \text{for $m>0$ and } l \leq \frac{p'}{2}, \\
\frac{\Gamma(\frac{2\nu+p'+2}{4}+ \lfloor \frac{2l-p'+2}{4}\rfloor)}{\Gamma(\frac{2\nu+p'+2}{4})}u_{\lambda,\nu}^B & \text{for $m>0$ and } l> \frac{p'}{2}, \\
\frac{\Gamma(-\frac{\nu}{2}-\lfloor \frac{l}{2} \rfloor)}{\Gamma(\frac{\lambda+\rho-\nu-\rho'}{2})}
u_{\lambda,\nu}^B & \text{for $m=0$.}
\end{cases} $$
\item \label{thm:singular_family:iii}
The following identity holds:
\begin{multline*}
u_{\lambda,\nu}^B = c^B(\lambda,\nu)\sum_{k=0}^l\frac{2^{2l-2k}l!(k+\frac{p''}{2})_{l-k}}{k!}\Bigg(\sum_{i+2j=l-k}\frac{(-1)^{i+j}2^i\Gamma(\frac{\nu+\rho'}{2}+i+j)}{i!j!\Gamma(\frac{\nu+\rho'}{2})}\\
 \times\abs{X'}^{2i}N(X',Z)^{-2(\nu+\rho')-4i-4j}\Bigg)\Delta_{\mathfrak{v}''}^k\delta(X'').
\end{multline*}
\item \label{thm:singular_family:iv}
$u_{\lambda,\nu}^B$ is non-zero for all $(\lambda,\nu) \in \LocalPoles$, more precisely:
$$\Supp u_{\lambda,\nu}^B=\begin{cases}
\{0\} & \text{for } (\lambda,\nu)\in \IntersectionPoles \setminus L, \\
\nbar' & \text{otherwise.} 
\end{cases}
$$
\end{enumerate}
\end{theorem}

\begin{proof}
By Remark~\ref{remark:residue_solution_outside_origin} it is clear that the residue formula (ii) holds for $\Re\nu\ll0$, and then Theorem~\ref{lemma:holomorphic_cont_solution} implies that $u_{\lambda,\nu}^B$ extends meromorphically in $\nu$. We now show the identity (iii) and afterwards deduce (i) and (iv) from it.

Let $\Re \nu \ll 0$ and write $N(X,Z)^{-2(\nu+\rho')}=f(\abs{X''}^2)$ with $f(x)=(x^2+2a x + b)^{-\frac{\nu+\rho'}{2}}$ and $a = \abs{X'}^2$, $b=N(X',Z)^4$. By the Multinomial Theorem~\eqref{theorem:multinomial_theorem} we have for $\varphi \in C_c^\infty(\nbar)$:
$$ \Delta^l_{\mathfrak{v}''}\left(f(\abs{X''}^2)\varphi(X,Z)\right) \biggr|_{X''=0} = \sum_{\substack{\alpha\in \Z_{\geq 0}^{p''}, \\ \abs{\alpha}=l }} \frac{l!}{\alpha!}\left(\frac{\partial}{\partial X''}\right)^{2\alpha} \left(f(\abs{X''}^2)\varphi(X,Z)\right)\biggr|_{X''=0}. $$
Since odd powers of partial derivatives applied to $f(\abs{X''}^2)$ vanish at $X''=0$, we see by successively applying the product rule
$$ =\sum_{\substack{\alpha,\beta \in \Z_{\geq 0}^{p''}, \\ \abs{\alpha}=l, \beta \leq \alpha }} \frac{l! (2\alpha)!}{\alpha! (2\beta)! (2\alpha-2\beta)!} \left(\frac{\partial}{\partial X''}\right)^{2 \beta} f(\abs{X''}^2)\cdot\left(\frac{\partial}{\partial X''}\right)^{2\alpha-2 \beta}\varphi(X,Z)\biggr|_{X''=0} . $$
Then successively applying Faà di Bruno's Formula~\eqref{eq:faa_di_bruno} yields
$$
=\sum_{\substack{\alpha,\beta \in \Z_{\geq 0}^{p''}, \\ \abs{\alpha}=l, \beta \leq \alpha }} \frac{l! (2\alpha)!}{\alpha! \beta! (2\alpha-2\beta)!}
f^{(\abs{\beta})}(0)\left(\frac{\partial}{\partial X''}\right)^{2\alpha-2 \beta}\varphi(X,Z)\biggr|_{X''=0}.
$$
Again by Faà di Bruno's Formula~\eqref{eq:faa_di_bruno} we have 
	$$f^{({\abs{\beta}})}(x)\Big|_{x=0} = \sum_{i+2j={\abs{\beta}}}(-1)^{i+j} 2^i\frac{\abs{\beta}!}{i!j!} \frac{\Gamma(\frac{\nu+\rho'}{2}+i+j)}{\Gamma(\frac{\nu+\rho'}{2})} b^{-\frac{\nu+\rho'}{2}-i-j}a^i,  $$
so we obtain
\begin{multline*}
u_{\lambda,\nu}^B
=c^B(\lambda,\nu)\sum_{\substack{\alpha,\beta \in \Z_{\geq 0}^{p''}, \\ \abs{\alpha}=l, \beta \leq \alpha }} \frac{l! (2\alpha)!\abs{\beta}!}{\alpha!\beta! (2\alpha-2\beta)!} \\ \times
\sum_{i+2j=\abs{\beta}}\frac{(-1)^{i+j}2^i\Gamma(\frac{\nu+\rho'}{2}+i+j)}{i!j!\Gamma(\frac{\nu+\rho'}{2})}\abs{X'}^{2i}N(X',Z)^{-2(\nu+\rho')-4i-4j}\delta^{(2\alpha-2\beta)}(X'').
\end{multline*}
Writing $\alpha=\beta+\gamma$ and regrouping gives
\begin{multline*}
 \!\!u_{\lambda,\nu}^B
=c^B(\lambda,\nu)\sum_{k=0}^l\Bigg(\sum_{i+2j=l-k}\frac{(-1)^{i+j}2^i\Gamma(\frac{\nu+\rho'}{2}+i+j)}{i!j!\Gamma(\frac{\nu+\rho'}{2})}\abs{X'}^{2i}N(X',Z)^{-2(\nu+\rho')-4i-4j}\Bigg)\\
\times\sum_{\abs{\gamma}=k}\frac{l!(l-k)!}{(2\gamma)!}\Bigg(\sum_{\abs{\beta}=l-k}\frac{(2\beta+2\gamma)!}{(\beta+\gamma)!\beta!}
\Bigg)\delta^{(2\gamma)}(X'').
\end{multline*}
Note that
$$ \frac{(2\beta+2\gamma)!}{(\beta+\gamma)!} = \frac{(2\gamma)!}{\gamma!}\prod_{h=1}^{p''}\frac{(2\gamma_h+1)_{2\beta_h}}{(\gamma_h+1)_{\beta_h}} = \frac{(2\gamma)!}{\gamma!}\prod_{h=1}^{p''}2^{2\beta_h}(\gamma_h+\tfrac{1}{2})_{\beta_h} = \frac{2^{2|\beta|}(2\gamma)!(\gamma+\frac{1}{2})_\beta}{\gamma!}, $$
where we have used the Pochhammer symbol $(a)_n=a(a+1)\cdots(a+n-1)$ and its multi-index analog $(\alpha)_\beta=(\alpha_1)_{\beta_1}\cdots(\alpha_{p''})_{\beta_{p''}}$. Then the sum over $\beta$ can be computed using Lemma~\ref{lem:MultiindexSum} and we obtain
\begin{align*}
u_{\lambda,\nu}^B ={}& c^B(\lambda,\nu)\sum_{k=0}^l\Bigg(\sum_{i+2j=l-k}\frac{(-1)^{i+j}2^i\Gamma(\frac{\nu+\rho'}{2}+i+j)}{i!j!\Gamma(\frac{\nu+\rho'}{2})}\abs{X'}^{2i}N(X',Z)^{-2(\nu+\rho')-4i-4j}\Bigg)\\
& \hspace{5cm}\times\sum_{\abs{\gamma}=k}\frac{2^{2l-2k}l!(k+\frac{p''}{2})_{l-k}}{\gamma!}\delta^{(2\gamma)}(X'')\\
={}& c^B(\lambda,\nu)\sum_{k=0}^l\Bigg(\sum_{i+2j=l-k}\frac{(-1)^{i+j}2^i\Gamma(\frac{\nu+\rho'}{2}+i+j)}{i!j!\Gamma(\frac{\nu+\rho'}{2})}\abs{X'}^{2i}N(X',Z)^{-2(\nu+\rho')-4i-4j}\Bigg)\\
& \hspace{5cm}\times\frac{2^{2l-2k}l!(k+\frac{p''}{2})_{l-k}}{k!}\Delta_{\mathfrak{v}''}^k\delta(X''),
\end{align*}
where we have again used the Multinomial Theorem~\eqref{theorem:multinomial_theorem} in the second step. This shows (iii) for $\Re\nu\ll0$ and the general case follows by analytic continuation. Note that the distributions $\Delta_{\mathfrak{v}''}^k\delta(X'')$ ($0\leq k\leq l$) are linearly independent.

We now show (i) and (iv) for $m>0$. First note that $(\lambda,\nu)\in L$ if and only if $\nu=-\rho'+q-1-2j$ with $2j\leq l-\frac{p'}{2}-1$ which can only happen if $l>\frac{p'}{2}$. Since both $u_{\lambda,\nu}^A$ and the gamma factor in the residue formula (ii) vanish for precisely those values of $\nu$, it follows from Theorem~\ref{lemma:holomorphic_cont_solution} that $u_{\lambda,\nu}^B$ is holomorphic in $\nu$, so we have shown (i). Now for $(\lambda,\nu)\notin L$ the gamma factor in the residue formula (ii) is non-zero, so $u_{\lambda,\nu}^B$ is a non-zero multiple of $u_{\lambda,\nu}^A$ and therefore $\Supp u_{\lambda,\nu}^B=\Supp u_{\lambda,\nu}^A$, so that (iv) follows from Corollaries~\ref{corollary:support_in_0} and \ref{cor:kernel_zero_set} in this case. For $(\lambda,\nu)\in L$ with $\nu=-\rho'+q-1-2j$, $2j\leq l-\frac{p'}{2}-1$, the normalization factor $c^B(\lambda,\nu)$ is regular. Now, the term for $k=l$ in the identity (iii) equals
$$ c^B(\lambda,\nu)N(X',Z)^{-2(\nu+\rho')}\Delta_{\mathfrak{v}''}^l\delta(X''). $$
Since also $N(X',Z)^{-2(\nu+\rho')}$ is regular with $\Supp(N(X',Z)^{-2(\nu+\rho')})=\nbar'$ by Corollary~\ref{cor:HeisenbergNormPowersHolomorphic} and all other terms for $0\leq k<l$ have support contained in $\nbar'$, this implies (iv) for $(\lambda,\nu)\in L$.

Finally we show (i) and (iv) for $m=0$. Here $(\lambda,\nu)\in L$ if and only if $\nu$ is an odd integer $\geq-l$. These are precisely the poles of $\Gamma(-\frac{\nu-1}{2}-\lfloor\frac{l+1}{2}\rfloor)$. The gamma factor in the residue formula (ii) can with the duplication formula be rewritten as
$$ \frac{\Gamma(-\frac{\nu}{2}-\lfloor\frac{l}{2}\rfloor)}{\Gamma(\frac{\lambda+\rho-\nu-\rho'}{2})} = \frac{\Gamma(-\frac{\nu}{2}-\lfloor\frac{l}{2}\rfloor)}{\Gamma(-\nu-l)} = \frac{2^{\nu+l+1}\sqrt{\pi}\Gamma(-\frac{\nu}{2}-\lfloor\frac{l}{2}\rfloor)}{\Gamma(-\tfrac{\nu}{2}-\tfrac{l}{2})\Gamma(-\tfrac{\nu}{2}-\tfrac{l-1}{2})}. $$
Since $\{\frac{l}{2},\frac{l-1}{2}\}=\{\lfloor\frac{l+1}{2}\rfloor-\frac{1}{2},\lfloor\frac{l}{2}\rfloor\}$ this equals
$$ = \frac{2^{\nu+l+1}\sqrt{\pi}}{\Gamma(-\tfrac{\nu-1}{2}-\lfloor\tfrac{l+1}{2}\rfloor)} $$
and therefore the gamma factor vanishes precisely where $u_{\lambda,\nu}^A$ vanishes. It follows that $u_{\lambda,\nu}^B$ is holomorphic on $\LocalPoles$, so we have shown (i). For $(\lambda,\nu)\notin L$ the same argument as for $m>0$ shows (iv). For $(\lambda,\nu)\in L$ the term for $k=l$ in the identity (iii) equals
$$ c^B(\lambda,\nu)|Z|^{-(\nu+\rho')}\Delta_{\mathfrak{v}}^l\delta(X). $$
Since $\nu$ is an odd integer, the normalization factor $c^B(\lambda,\nu)$ is regular, and since $\rho'=q=\dim\mathfrak{z}$, we have $\Supp(|Z|^{-(\nu+\rho')})=\nbar'$ by Lemma~\ref{lemma:absolute_value_distr_holomorphic}. This implies (iv) for $(\lambda,\nu)\in L$ and the proof is complete.
\end{proof}

In the following we write $B_{\lambda,\nu}$ for the symmetry breaking operator corresponding to $u_{\lambda,\nu}^B$.

\section{Classification of symmetry breaking operators}

In this section we give a full classification of symmetry breaking operators between spherical principal series of $G$ and $G'$ in terms of the operators $A_{\lambda,\nu}$ and $B_{\lambda,\nu}$ and the previously classified differential symmetry breaking operators (see Part~\ref{part:diff}).

\begin{theorem}
\label{theorem:classification}
For all strongly spherical pairs of the form $(G,G')=(\Urm(1,n+1;\F),\Urm(1,m+1;\F)\times F)$ with $\F=\C,\HH,\mathbb{O}$, $0\leq m<n$ and $F<\Urm(n-m;\F)$ we have
$$\mathcal{D}'(\nbar)_{\lambda,\nu}=
\begin{cases}
\C u_{\lambda,\nu}^A  &\text{if } (\lambda,\nu) \in \C^2 \setminus \GlobalPoles, \\
\mathcal{D}'_{\{0\}}(\nbar)_{\lambda,\nu} & \text{if } (\lambda,\nu) \in \GlobalPoles \setminus L, \\
\C u^B_{\lambda,\nu} \oplus \mathcal{D}'_{\{0\}}(\nbar)_{\lambda,\nu} & \text{if } (\lambda,\nu) \in L .
\end{cases} $$
\end{theorem}
The rest of this section is devoted to the proof of this statement.

\subsection{Continuation of solutions outside of the origin}
\label{sec:continuation_of_solutions}

We first study for which parameters the distribution solutions outside the origin in $\nbar$ can be extended to the whole $\nbar$, i.e. for which $(\lambda,\nu)\in\C^2$ the restriction map $\mathcal{D}'(\nbar)_{\lambda,\nu} \to \mathcal{D}'(\nbar\setminus \{0\})_{\lambda,\nu}$ is onto. Since the latter space is one-dimensional, we can equivalently determine the cases where the restriction map is trivial.

\begin{prop}
\label{prop:continuation_solutions}
The following are equivalent:
\begin{enumerate}[label=(\roman{*})]
\item $(\lambda,\nu)\in \GlobalPoles \setminus L$,
\item The restriction $\mathcal{D}'(\nbar)_{\lambda,\nu}\to \mathcal{D}'(\nbar\setminus \{0\})_{\lambda,\nu}$ is identical zero,
\item  $\mathcal{D}'(\nbar)_{\lambda,\nu}=\mathcal{D}'_{\{0\}}(\nbar)_{\lambda,\nu}$.
\end{enumerate}
\end{prop}

The proof of Proposition~\ref{prop:continuation_solutions} works analogously as the proof of \cite[Proposition 11.7]{kobayashi_speh_2015}. Therefore we need to use the following facts (see \cite[Lemma 11.10, Lemma 11.11]{kobayashi_speh_2015}):

\begin{lemma}
\label{lemma:KS_continuation}
\begin{enumerate}[label=(\roman{*}), ref=\thetheorem(\roman{*})]
\item 
\label{lemma:KS_continuation:i}
Let $d_\mu$ be a differential operator on $\R^{n}$ which is holomorphic in $\mu \in \C$ and let $v_\mu \in \mathcal{D}'(\R^{n})$ be a distribution which is meromorphic in $\mu$, such that for differential operators $d_i$ on $\R^n$ and distributions $v_j \in \mathcal{D}'(\R^n)$ we have the expansions
\begin{align*}
d_\mu=d_0 + \mu d_1+\mu^2 d_2 + \cdots \qquad \mbox{and} \qquad v_\mu=\mu^{-1}v_{-1} +v_0 + \mu v_1 + \cdots.
\end{align*}
If there exists an $\varepsilon > 0$ such that that $d_\mu v_\mu =0$ for all $\mu \in \C$  with $0< \abs{\mu} < \varepsilon$,  then
$$ d_0v_{-1}=0 \qquad \text{and} \qquad d_0v_0+d_1v_{-1}=0.$$

\item 
\label{lemma:KS_continuation:ii}
Let $E$ be the weighted Euler-operator on $\R^{p+q}$ which is in coordinates $(x_1, \ldots , x_p,\linebreak z_1 , \ldots, z_q)$ given by 
$$\sum_{i=1}^{p}x_i \frac{\partial}{\partial x_i} + 2\sum_{j=1}^{q}z_j\frac{\partial}{\partial z_j}.$$
Let $v \in \mathcal{D}'(\R^{p+q})$ be a distribution with $\Supp v= \{0\}$, then for every integer $k \in \Z$, $(E+k)v=0$ if and only if $(E+k)^2=0$.
\end{enumerate}
\end{lemma}

\begin{proof}
(i) follows immediatiely from the Laurent expansion of $d_\mu v_\mu$.

Ad (ii): Since $v_\mu$ is supported in the origin, we have $$v= \;  \sum_{\mathclap{\substack{ \alpha \in \Z_{\geq 0}^p, \beta \in \Z_{\geq 0}^q}}}  \;w_{\alpha,\beta} \delta^{((\alpha,0)+(0,\beta))}$$
for scalars $w_{\alpha,\beta} \in \C$ which are almost all zero. Then by homogeneity and since different derivatives of the Dirac delta are linearly independent it follows that
\begin{align*}
(E+k)v&=\;  \sum_{\mathclap{\substack{ \alpha \in \Z_{\geq 0}^p, \beta \in \Z_{\geq 0}^q}}} \; (k-p-2q-\abs{\alpha}-2\abs{\beta})w_{\alpha,\beta} \delta^{((\alpha,0)+(0,\beta))}=0, \\
\intertext{if and only if}
(E+k)^2v&= \; \sum_{\mathclap{\substack{ \alpha \in \Z_{\geq 0}^p, \beta \in \Z_{\geq 0}^q}}}  \;(k-p-2q-\abs{\alpha}-2\abs{\beta})^2w_{\alpha,\beta} \delta^{((\alpha,0)+(0,\beta))}=0.\qedhere
\end{align*}
\end{proof}

\begin{proof}[Proof of Proposition~\ref{prop:continuation_solutions}]
The equivalence of (ii) and (iii) follows directly from the exactness of the sequence
	$$
	\begin{tikzcd}
	0 \arrow[r] & \mathcal{D}'_{\{0\}}(\nbar)_{\lambda,\nu} \arrow[r] & \mathcal{D}'(\nbar)_{\lambda,\nu} \arrow[r]     &\mathcal{D}'(\nbar\setminus \{0\})_{\lambda,\nu},
	\end{tikzcd}
	$$
so it remains to show the equivalence of (i) and (ii). By Proposition~\ref{prop:solutions}, $u_{\lambda,\nu}^A|_{\nbar \setminus \{0\}}$ vanishes if and only if $\Gamma(\frac{\lambda+\rho-\nu-\rho'}{2})$ has a pole, i.e. if and only if $(\lambda,\nu)\in \GlobalPoles$. 
Hence $u_{\lambda,\nu}^A|_{\nbar \setminus \{0\}}$ is a non-zero element of $\mathcal{D}'(\nbar\setminus \{0\})_{\lambda,\nu}$ for all $(\lambda,\nu) \notin \GlobalPoles$.
For
$(\lambda,\nu) \in L$, the distribution $u^B_{\lambda,\nu}|_{\nbar \setminus \{0\}}$ is a non-zero element of $\mathcal{D}'(\nbar\setminus \{0\})_{\lambda,\nu}$ by Theorem~\ref{thm:singular_family:iv}.
So we are left to show that for $(\lambda_0,\nu_0)\in \GlobalPoles\setminus L$ the restriction map $\mathcal{D}'(\nbar)_{\lambda_0,\nu_0}\to\mathcal{D}'(\nbar\setminus\{0\})_{\lambda_0,\nu_0}$ is trivial, i.e. every $w\in\mathcal{D}'(\nbar)_{\lambda_0,\nu_0}$ has $\Supp w\subseteq\{0\}$. Write $\lambda_0+\rho-\nu_0-\rho'=-2k$ with $k\in \Z_{\geq 0}$. For all $(\lambda,\nu) \in \C$ with $\lambda+\nu=\lambda_0+\nu_0$, consider the parameter $\mu = \lambda +\rho-\nu-\rho'+2k$. Then $v_{\lambda,\nu}:=\Gamma(\frac{\lambda+\rho-\nu-\rho'}{2})u_{\lambda, \nu}^A$ depends meromorphically on $\mu$ and has a simple pole at $\mu=0$, such that near $(\lambda_0,\nu_0)$:
$$v_{\lambda,\nu}=\mu^{-1}v_{-1}+v_0+\mu v_1 + \ldots,$$
for distributions $v_{i}\in\mathcal{D}'(\nbar)$. By Corollary~\ref{cor:kernel_zero_set} the distribution $v_{-1}$ is non-trivial for $(\lambda_0,\nu_0) \notin L$. Further,
$$ v_{\lambda,\nu}|_{\nbar\setminus\{0\}}=\frac{1}{\Gamma(\frac{\lambda+\rho+\nu-\rho'}{2})}N(X,Z)^{-2(\nu+\rho')}\abs{X''}^{\lambda-\rho+\nu+\rho'}, $$
which depends holomorphically on $(\lambda,\nu)\in\C^2$, so $v_{-1}|_{\nbar\setminus\{0\}}=0$ and $v_0|_{\nbar\setminus\{0\}}\neq0$ spans $\mathcal{D}'(\nbar\setminus\{0\})_{\lambda_0,\nu_0}$ by Proposition~\ref{prop:solutions}. If now $w \in \mathcal{D}'(\nbar)_{\lambda_0,\nu_0}$, then there exists a constant $c\in \C$ such that
$$ w|_{\nbar\setminus \{0\}} = cv_0|_{\nbar \setminus\{0\}}. $$
In particular the distribution $w':=w-cv_0\in\mathcal{D}'(\nbar)$ has support contained in $\{0\}$. Now, by Lemma~\ref{lemma:diff_equations:i} we have
$$ ((E+2\rho+2k)-\mu \cdot \mathbf{1})v_{\lambda,\nu} = (E-(\lambda-\rho-\nu-\rho'))v_{\lambda,\nu} = 0, $$
so Lemma~\ref{lemma:KS_continuation:i} implies
$$ (E+2\rho+2k)v_{-1}=0 \qquad \mbox{and} \qquad (E+2\rho+2k)v_0-v_{-1} =0. $$
Hence
$$ (E+2\rho+2k)^2w'=(E+2\rho+2k)^2w-(E+2\rho+2k)^2cv_0=0, $$
because also $(E+2\rho+2k)w=0$. Lemma~\ref{lemma:KS_continuation:ii} then implies $(E+2\rho+2k)w'=0$, so $c(E+2\rho+2k)v_0=0$ and therefore $cv_{-1}=0$. Since $v_{-1}\neq 0$ we must have $c=0$ and $\Supp w= \Supp w'\subseteq \{0\}$, i.e. $w|_{\nbar\setminus\{0\}}=0$.
\end{proof}

\subsection{Proof of the main theorem}
We can finally prove Theorem~\ref{theorem:classification}.
Using the exact sequence
	$$
	\begin{tikzcd}
	0 \arrow[r] & \mathcal{D}'_{\{0\}}(\nbar)_{\lambda,\nu} \arrow[r] & \mathcal{D}'(\nbar)_{\lambda,\nu} \arrow[r]     &\mathcal{D}'(\nbar\setminus \{0\})_{\lambda,\nu},
	\end{tikzcd}
	$$
Corollary~\ref{cor:DiffOpsOnlyInGlobalPoles} and Proposition~\ref{prop:continuation_solutions} imply the statement for $(\lambda,\nu)\notin L$. For $(\lambda,\nu) \in L$, Proposition~\ref{prop:solutions} implies that $\dim \mathcal{D}'_{\{0\}}(\nbar)_{\lambda,\nu} \leq \dim \mathcal{D}'(\nbar)_{\lambda,\nu} \leq \dim \mathcal{D}'_{\{0\}}(\nbar)_{\lambda,\nu}+1$. 
Since $u_{\lambda,\nu}^B \notin \mathcal{D}'_{\{0\}}(\nbar)_{\lambda,\nu}$ for $(\lambda,\nu) \in L$ due to its support,
Theorem~\ref{thm:singular_family} implies the statement for $(\lambda,\nu) \in L$. \hfill \qedsymbol

Combining Theorem~\ref{theorem:diff_solutions}, Corollary~\ref{corollary:support_in_0}, Corollary~\ref{cor:kernel_zero_set} and Theorem~\ref{thm:singular_family}:
\begin{corollary}
\label{cor:supports}
\begin{enumerate}[label=(\roman{*})]
\item For $(\lambda,\nu)\in\C^2$:
$$\Supp u^A_{\lambda,\nu}=\begin{cases}
\emptyset & \text{for } (\lambda,\nu) \in L, \\
\{0\} & \text{for } (\lambda,\nu) \in \GlobalPoles \setminus L, \\
\nbar' & \text{for } (\lambda,\nu) \in \LocalPoles \setminus \IntersectionPoles, \\
\nbar & \text{otherwise.} 
\end{cases}$$
\item For $(\lambda,\nu)\in \LocalPoles$:
$$\Supp u^B_{\lambda,\nu}=
\begin{cases}
\{0\}  & \text{for } (\lambda,\nu) \in \IntersectionPoles \setminus L, \\
\nbar' & \text{otherwise.} \\
\end{cases}
$$
\item For $(\lambda,\nu)\in \GlobalPoles$:
$$ \Supp u_{\lambda,\nu}^C=\{0\}. $$
\end{enumerate}
\end{corollary}

\section{Functional equations}\label{sec:FunctionalEquations}

Let $\T_\lambda: \pi_{\lambda} \to \pi_{-\lambda}$ be the normalized Knapp--Stein intertwining operator for $G$ given in the non-compact picture on $\nbar$ by convolution with
$$ \frac{1}{\Gamma(\lambda)} N(X,Z)^{2(\lambda-\rho)}. $$
Similarly, we denote by $\T'_{\nu}:\tau_\nu \to \tau_{-\nu}$ the normalized Knapp--Stein intertwining operator for $G'$ which is defined for $m>0$ by convolution with
$$\frac{1}{\Gamma(\nu)} N(X,Z)^{2(\nu-\rho')}, $$
and for $m=0$ by convolution with
$$\frac{1}{\Gamma(\frac{\nu}{2})} \abs{Z}^{\nu-\rho'}. $$
By Lemma~\ref{lemma:absolute_value_distr_holomorphic} and Corollary~\ref{cor:HeisenbergNormPowersHolomorphic} both $\T_\lambda$ and $\T_\nu'$ are holomorphic in the parameters $\lambda$ and $\nu$ and non-trivial for all $\lambda,\nu\in\C$. Since the normalized Knapp--Stein operators are intertwining, they define maps between spaces of symmetry breaking operators by composition:
\begin{align*}
&\Hom_{G'}(\pi_\lambda|_{G'},\tau_\nu) \to \Hom_{G'}(\pi_{-\lambda}|_{G'},\tau_\nu), && A \mapsto A \circ \T_{-\lambda},\\
&\Hom_{G'}(\pi_\lambda|_{G'},\tau_\nu) \to \Hom_{G'}(\pi_{\lambda}|_{G'},\tau_{-\nu}), && A \mapsto \T_{\nu}' \circ A.
\end{align*}

Formulas expressing the image of a symmetry breaking operator under such a composition map in terms of another symmetry breaking operator are called functional equations. These equations are particularly interesting if the Knapp--Stein operator or the symmetry breaking operator is a differential operator. For $T_\lambda$ this happens precisely for $\lambda =-k \in -\Z_{\geq 0}$, where $\T_\lambda$ is given by convolution with (see Corollary~\ref{cor:HeisenbergNormPowersHolomorphic})
$$ \frac{(-1)^k2^{-k}k!\pi^{\frac{p+q}{2}}}{\Gamma(\frac{\rho+k}{2})} \sum_{i+2j=k} \frac{2^{-i}\Gamma(\frac{p+2i}{4})}{i!j!\Gamma(\frac{p+2i}{2})}\Delta_{\mathfrak{v}}^i\square^j\delta(X,Z). $$

\begin{remark}
Similarly for $m>0$ and $\nu=-k\in -\Z_{\geq 0}$, $T'_\nu$ is given by convolution with 
$$ \frac{(-1)^k2^{-k}k!\pi^{\frac{p'+q}{2}}}{\Gamma(\frac{\rho'+k}{2})} \sum_{i+2j=k} \frac{2^{-i}\Gamma(\frac{p'+2i}{4})}{i!j!\Gamma(\frac{p'+2i}{2})}\Delta_{\mathfrak{v}'}^i\square^j\delta(X',Z). $$
For $m=0$ and $\nu=-2k \in -2\Z_{\geq 0}$, Lemma~\ref{lemma:absolute_value_distr_holomorphic} implies that the operator $T'_\nu$ is given by convolution with
$$ (-1)^k \frac{\pi^\frac{q}{2}}{2^{2k}\Gamma(\frac{\rho'+2k}{2})}\square^k\delta(Z). $$
\end{remark}

We choose the maximal compact subgroup $K=\Urm(1;\F)\times\Urm(n+1;\F)$ of $G$. The spherical principal series representations $\pi_\lambda$ contain a unique (up to scalar multiples) $K$-spherical vector $\mathbf{1}_\lambda\in I_\lambda$ which we normalize by $\mathbf{1}_\lambda(0)=1$. We first find an explicit formula for $\mathbf{1}_\lambda$:

\begin{lemma}
For $(X,Z)\in\nbar$ we have
$$ \mathbf{1}_\lambda(X,Z)=((1+\abs{X}^2)^2+\abs{Z}^2)^{-\frac{\lambda+\rho}{2}}. $$
\end{lemma}

\begin{proof}
By the Iwasawa decomposition $G=KAN$ we can decompose $\bar{n}_{(X,Z)}=kan$ with $k\in K$, $a=\exp(tH)\in A$ and $n\in N$. An easy computation shows that
$$ e^{2t}=(1+\abs{X}^2)^2+\abs{Z}^2, $$
then the claim follows.
\end{proof}

The intersection $K'=K\cap G'=\Urm(1;\F)\times\Urm(m+1;\F)\times F$ is a maximal compact subgroup of $G'$ and we denote by
$$ \mathbf{1}'_\nu(X',Z) = ((1+\abs{X'}^2)^2+\abs{Z}^2)^{-\frac{\nu+\rho'}{2}} $$
the unique $K'$-spherical vector of $\tau_\nu$, normalized by $\mathbf{1}'_\nu(0)=1$. Note that for $m=0$ this equals $(1+\abs{Z}^2)^{-\frac{\nu+\rho'}{2}}$.

We now evaluate the Knapp--Stein operators $T_\lambda$ and $T_\nu'$ and the symmetry breaking operators $A_{\lambda,\nu}$ on these spherical vectors. The relevant integral in this context was computed by Frahm--Su~\cite{FraSu_2018}. To keep this paper self-contained, we include a full proof.

\begin{prop}[{\cite[Section 4.2]{FraSu_2018}}]
\label{prop:integral_spherical_vector}
For $(\lambda,\nu)\in\C^2$ with $\abs{\Re\nu}<\Re\lambda+\frac{p''}{2}$:
\begin{multline*}
\int_{\nbar} N(X,Z)^{-2(\nu+\rho')} \abs{X''}^{\lambda-\rho+\nu+\rho'} ((1+\abs{X}^2)^2+\abs{Z}^2)^{-\frac{\lambda+\rho}{2}} \,d(X,Z)\\= \frac{\pi^{\frac{p+q}{2}} \Gamma(\frac{2\lambda+p}{4}) \Gamma(\frac{\lambda+\rho-\nu-\rho'}{2})\Gamma(\frac{\lambda+\rho+\nu-\rho'}{2}) }
{\Gamma(\frac{p''}{2}) \Gamma(\frac{2\lambda+p}{2}) \Gamma(\frac{\lambda+\rho}{2})   }.
\end{multline*}
\end{prop}

\begin{proof}
We abbreviate $a:=-2(\nu+\rho')$, $b:=\lambda-\rho+\nu+\rho'$ and $c:=-2(\lambda+\rho)$.
Using polar coordinates we find
\begin{multline*}
\int_{\nbar} N(X,Z)^{a} \abs{X''}^b ((1+\abs{X}^2)^2+\abs{Z}^2)^{\frac{c}{4}} \,d(X,Z) \\=
\frac{8\pi^{\frac{p+q}{2}}}{\Gamma(\frac{p'}{2})\Gamma(\frac{p''}{2})\Gamma(\frac{q}{2})}
\int_{\R_+^3}((r^2+s^2)^2+t^2)^{\frac{a}{4}}((1+r^2+s^2)^2+t^2)^{\frac{c}{4}}r^{p'-1}s^{p'+b-1}t^{q-1}\,dr\,ds\,dt.
\end{multline*}
We integrate over $t$ first using \cite[3.259 (3)]{Tables_Integrals_Prodicts}
\begin{multline*}
=\frac{4\pi^{\frac{p+q}{2}}\Gamma(-\frac{a+4b+2q}{4})}{\Gamma(\frac{p'}{2})\Gamma(\frac{p''}{2})\Gamma(-\frac{a+4b}{4})} 
\int_{\R_+^2} (1+r^2+s^2)^{\frac{c+2q}{2}} (r^2+s^2)^{\frac{a}{2}} r^{p'-1} s^{p''+b-1} 
\\ \times {_2F_1}\left(  -\frac{a}{4}, \frac{q}{2}; -\frac{a+c}{4};1-\frac{(1+r^2+s^2)^2}{(r^2+s^2)^2}  \right) \,dr\,ds.
\end{multline*}
Using polar coordinates on $\R^2$ we find
\begin{multline*}
=\frac{4\pi^{\frac{p+q}{2}}\Gamma(-\frac{a+4b+2q}{4})}{\Gamma(\frac{p'}{2})\Gamma(\frac{p''}{2})\Gamma(-\frac{a+4b}{4})} 
\int_{0}^{\frac{\pi}{2}} \cos^{p'-1} \phi \sin^{p''+b-1}\phi \,d\phi
\int_{0}^{\infty}x^{p+b+a-1}(1+x^2)^{\frac{c+2q}{2}} 
\\ \times{_2F_1}\left(  -\frac{a}{4}, \frac{q}{2}; -\frac{a+c}{4};1-\frac{(1+x^2)^2}{x^4}  \right) \,dx.
\end{multline*}
Calculating the first integral and using the Euler identity \cite[1, (2.2.7)]{SpecialFunct}
\begin{multline*}
=\frac{2\pi^{\frac{p+q}{2}}\Gamma(\frac{p''+b}{2})\Gamma(-\frac{a+4b+2q}{4})}{\Gamma(\frac{p''}{2})\Gamma(\frac{p+b}{2})\Gamma(-\frac{a+4b}{4})} 
\int_{0}^{\infty} x^{p+2q+a+b+c-1} 
\\ \times {_2F_1}\left(
-\frac{c}{4}, -\frac{a+c+2q}{4};-\frac{a+c}{4}, 1-\frac{(1+x^2)^2}{x^4}
\right) \,dx.
\end{multline*}
Substituting $y=\frac{1+x^2}{x^2}$
\begin{multline*}
=\frac{\pi^{\frac{p+q}{2}}\Gamma(\frac{p''+b}{2})\Gamma(-\frac{a+4b+2q}{4})}{\Gamma(\frac{p''}{2})\Gamma(\frac{p+b}{2})\Gamma(-\frac{a+4b}{4})} 
\int_{1}^{\infty} (y-1)^{-\frac{p+2q+a+b+c}{2}-1}
\\ \times {_2F_1}\left(
-\frac{c}{4}, -\frac{a+c+2q}{4};-\frac{a+c}{4}, 1-y^2
\right) \,dy.
\end{multline*}
Expanding the hypergeometric function, using the Euler integral representation \cite[1, (2.3.17)]{SpecialFunct} yields
\begin{equation*}
=\frac{\pi^{\frac{p+q}{2}}\Gamma(\frac{p''+b}{2})}{\Gamma(\frac{p''}{2})\Gamma(\frac{q}{2})\Gamma(\frac{p+b}{2}) } 
\int_{0}^{\infty}t^{-\frac{a+c+2q}{4}-1}(1+t)^\frac{a}{4}  \int_{1}^{\infty}(y-1)^{-\frac{p+2q+a+b+c}{2}-1}(1+y^2t)^{\frac{c}{4}} \,dy\,dt.
\end{equation*}
Evaluating the inner integral using \cite[3.254 (2)]{Tables_Integrals_Prodicts} we find
\begin{multline*}
=\frac{\pi^{\frac{p+q}{2}}\Gamma(\frac{p''+b}{2})\Gamma(-\frac{p+2q+a+b+c}{2})\Gamma(\frac{p+2q+a+b}{2})}{\Gamma(\frac{p''}{2})\Gamma(\frac{q}{2})\Gamma(\frac{p+b}{2}) \Gamma(-\frac{c}{2})} 
\int_{0}^{\infty} t^{-\frac{a+2q}{4}-1}(1+t)^{\frac{a}{4}} \\
\times{_2F_1} \left(
\frac{p+2q+a+b}{4}, \frac{p+2q+a+b+2}{4};\frac{2-c}{4};-t^{-1}
\right) \,dt.
\end{multline*}
Substituting $x=t^{-1}$ we have
\begin{multline*}
=\frac{\pi^{\frac{p+q}{2}}\Gamma(\frac{p''+b}{2})\Gamma(-\frac{p+2q+a+b+c}{2})\Gamma(\frac{p+2q+a+b}{2})}{\Gamma(\frac{p''}{2})\Gamma(\frac{q}{2})\Gamma(\frac{p+b}{2}) \Gamma(-\frac{c}{2})} 
\int_{0}^{\infty}
x^{\frac{q}{2}-1}(1+x)^\frac{a}{4}
\\ \times {_2F_1} \left(
\frac{p+2q+a+b}{4}, \frac{p+2q+a+b+2}{4};\frac{2-c}{4};-x
\right) \,dx.
\end{multline*}
Then \cite[7.512 (5)]{Tables_Integrals_Prodicts} yields
\begin{multline*}
=\frac{\pi^{\frac{p+q}{2}}\Gamma(\frac{p''+b}{2})\Gamma(\frac{p+b}{4})\Gamma(-\frac{p+2q+a+b+c}{2})\Gamma(\frac{p+2q+a+b}{2})}{\Gamma(\frac{p''}{2})\Gamma(\frac{p+b}{2})\Gamma(-\frac{c}{2})\Gamma(\frac{p+2q+b}{4})} \\ \times {_3F_2}\left( 
\frac{p+2q+a+b}{4},-\frac{p+2q+a+b+c}{4},\frac{q}{2}; \frac{2-c}{4},\frac{p+2q+b}{4};1
\right).
\end{multline*}
Since $-p+2q-a-b-c=\lambda+\rho+\nu+\rho'=p+2q+b$, this is
$$
=\frac{\pi^{\frac{p+q}{2}}\Gamma(\frac{p''+b}{2})\Gamma(\frac{p+b}{4})\Gamma(-\frac{p+2q+a+b+c}{2})\Gamma(\frac{p+2q+a+b}{2})}{\Gamma(\frac{p''}{2})\Gamma(\frac{p+b}{2})\Gamma(-\frac{c}{2})\Gamma(\frac{p+2q+b}{4})}   {_2F_1}\left( 
\frac{p+2q+a+b}{4},\frac{q}{2}; \frac{2-c}{4};1
\right).
$$
Then evaluating the hypergeometric function with \cite[Theorem 2.2.2]{SpecialFunct} and applying the duplication formula of the gamma function yields the desired formula.
\end{proof}

For the case $m=0$ we further need the following integral:

\begin{prop}
\label{prop:spherical_vector_KS_m=0}
For $\nu \in \C$ with $\Re \nu>0$ we have
$$\int_{\mathfrak{z}}\abs{Z}^{\nu-\rho'}(1+\abs{Z}^2)^{-\frac{\nu+\rho'}{2}} \, dZ=\frac{\pi^{\frac{q}{2}}\Gamma(\frac{\nu}{2})}{\Gamma(\frac{\nu+\rho'}{2})}.$$
\end{prop}

\begin{proof}
Using polar coordinates on $\R^q$ the integral is equal to
$$\frac{2\pi^{\frac{q}{2}}}{\Gamma(\frac{q}{2})}\int_{0}^{\infty}r^{\nu-1}(1+r^2)^{-\frac{\nu+\rho'}{2}}\,dr.$$
Then substituting $r^2=t$ this is equal to
\begin{equation*}
\frac{\pi^{\frac{q}{2}}}{\Gamma(\frac{q}{2})}B\left(\frac{\nu}{2},\frac{\rho'}{2} \right)=\frac{\pi^{\frac{q}{2}}\Gamma(\frac{\nu}{2})}{\Gamma(\frac{\nu+\rho'}{2})}. \qedhere
\end{equation*}
\end{proof}

\begin{corollary}\label{lemma:spherical_vector_action}
For the spherical vectors $\mathbf{1}_\lambda$ and $\mathbf{1}'_\nu$ we have
$$ \T_\lambda \mathbf{1}_\lambda=\frac{\pi^{\frac{p+q}{2}}\Gamma(\frac{2\lambda+p}{4})}{\Gamma(\frac{2\lambda+p}{2})\Gamma(\frac{\lambda+\rho}{2})} \mathbf{1}_{-\lambda}, \qquad
\T'_\nu \mathbf{1}'_\nu=\frac{\pi^{\frac{p'+q}{2}}}{\Gamma(\frac{\nu+\rho'}{2})}\times \begin{cases}
\frac{\Gamma(\frac{2\nu+p'}{4})}{\Gamma(\frac{2\nu+p'}{2})} \mathbf{1}'_{-\nu} & \text{for $m>0$,}\\
\mathbf{1}'_{-\nu} & \text{for $m=0$}
\end{cases} $$
and
$$ A_{\lambda,\nu}\mathbf{1}_\lambda  = \frac{\pi^{\frac{p+q}{2}} \Gamma(\frac{2\lambda+p}{4})  }
{\Gamma(\frac{p''}{2}) \Gamma(\frac{2\lambda+p}{2}) \Gamma(\frac{\lambda+\rho}{2})   }\mathbf{1}'_\nu. $$
\end{corollary}

\begin{proof}
Since $A_{\lambda,\nu}:\pi_\lambda \to \tau_\nu$ is intertwining and since $K'\subseteq K$, the image $A_{\lambda,\nu}\mathbf{1}_\lambda$ of the spherical vector $\mathbf{1}_\lambda$ must be $K'$-invariant, i.e. $K'$-spherical. Hence
it is enough to compute $A_{\lambda,\nu} \mathbf{1}_\lambda (0)=\langle u_{\lambda,\nu}^A,\mathbf{1}_\lambda \rangle$. The last identity follows from Proposition~\ref{prop:integral_spherical_vector}. The remaining two identities follow similarly by setting $\lambda-\rho+\nu+\rho'=0$ in Proposition~\ref{prop:integral_spherical_vector} for $m>0$, and for $m=0$ by Proposition~\ref{prop:spherical_vector_KS_m=0}.
\end{proof}

Since the compositions of symmetry breaking operators and Knapp--Stein operators are again symmetry breaking operators and these are generically unique, we can read off the following functional equations by evaluating at the spherical vector and employing Corollary~\ref{lemma:spherical_vector_action}.

\begin{theorem}
\label{thm:functional_equations}
For $(\lambda,\nu)\in \C^2$ we have
\begin{align*}
 \T'_\nu \circ A_{\lambda,\nu} &= \frac{\pi^{\frac{p'+q}{2}}}{\Gamma(\frac{\nu+\rho'}{2})}\times \begin{cases}
 \frac{\Gamma(\frac{2\nu+p'}{4})}{\Gamma(\frac{2\nu+p'}{2})} A_{\lambda,-\nu} & \text{for $m>0$,}\\
 A_{\lambda,-\nu} & \text{for $m=0$,}\\
 \end{cases} \\
 A_{\lambda,\nu}\circ \T_{-\lambda} &= \pi^{\frac{p+q}{2}}\frac{\Gamma(\frac{2\lambda+p}{4})}{\Gamma(\frac{2\lambda+p}{2})\Gamma(\frac{\lambda+\rho}{2})}A_{-\lambda,\nu}.
\end{align*}
\end{theorem}

\begin{remark}
Theorem~\ref{thm:functional_equations} can be combined with the residue formulas Theorem~\ref{prop:diff_op_residue} and Theorem~\ref{thm:singular_family:ii} to obtain formulas for the composition of $B_{\lambda,\nu}$ and $C_{\lambda,\nu}$ with the Knapp--Stein operators $T_\lambda$ and $T'_\nu$. For example for $m>0$, 
$(\lambda,\nu)\in \IntersectionPoles$ and
$l,k \in \Z_{\geq 0}$ given by $\lambda+\rho+\nu-\rho'=-2l$, $\lambda+\rho-\nu-\rho'=-2k$ with $l \geq k$ we obtain the following identity involving only differential operators:
$$\T'_{k-l} \circ C_{-k-l-\frac{p''}{2},k-l}=(-1)^{l+k}\pi^{\frac{p'+q}{2}} \frac{l!\Gamma(\frac{2l-2k+p'}{4})}{k!\Gamma(\frac{2l-2k+p'}{2})\Gamma(\frac{l-k+\rho'}{2})} C_{-k-l-\frac{p''}{2},l-k},$$
and for $(\lambda,\nu)\in L$ with $\nu\in -\rho'-2\Z_{\geq 0}$ we obtain 
$\T'_\nu \circ B_{\lambda,\nu} = 0.$
\end{remark}

\newpage

\part*{Appendix}
\addcontentsline{toc}{part}{Appendix}

\appendix

\section{Homogeneous generalized functions}

Let $n\geq1$. For $\lambda\in\C$ with $\Re\lambda>-n$ the function
\begin{equation*}
u_{\lambda}(x)=\frac{\abs{x}^\lambda}{\Gamma(\frac{\lambda+n}{2})}
\end{equation*}
is locally integrable and hence defines a distribution $u_\lambda \in \mathcal{D}'(\R^n)$ which is homogeneous of degree $\lambda$.

\begin{lemma}[\cite{Gelfand_Shilov_1} Chapter I, 3.5 and 3.9]
\label{lemma:absolute_value_distr_holomorphic}
The family of distributions $u_{\lambda}\in \mathcal{D}'(\R^{n})$ extends holomorphically to an entire function in $\lambda \in \C$. For $\lambda\notin-n-2\Z_{\geq0}$ we have $\Supp u_\lambda=\R^n$ and for $\lambda=-n-2N\in-n-2\Z_{\geq 0}$ we have
$$ u_{-n-2N}(x) = (-1)^N\frac{\pi^{\frac{n}{2}}}{2^{2N}\Gamma(\frac{n+2N}{2})}(\Delta^N\delta)(x), $$
where $\Delta=\sum_{i=1}^n \frac{\partial^2}{\partial x_i^2}$ is the Laplacian on $\R^n$.
\end{lemma}

The following result follows immediately from the classification of homogeneous distributions on $\R$ in \cite[Chapter I, 3.11]{Gelfand_Shilov_1}:

\begin{lemma}
\label{lemma:homogenous_distributions}
Let $F < \Orm(n)$ be a compact subgroup which acts transitively on the unit sphere $S^{n-1} \subseteq \R^n$. If  $u\in \mathcal{D}'(\R^n)$ is homogeneous of degree $\lambda$ and invariant under the action of $F$, then $u$ is a scalar multiple of $u_\lambda$.
\end{lemma}

\section{Integral formulas}

For $\alpha\in\Z_{\geq0}^p$ we write $|\alpha|=\alpha_1+\cdots+\alpha_p$ and $\alpha!=\alpha_1!\cdots\alpha_p!$.

\begin{lemma}
\label{lemma:int_techniques}
\begin{enumerate}[label=(\roman{*}), ref=\thetheorem(\roman{*})]
\item 
\label{lemma:int_techniques:i}
Let $S^{p-1}\subseteq\R^p$ be the unit sphere and $\alpha \in \Z_{\geq 0}^p$, then
\begin{align*}
\int_{S^{p-1}}\omega_{1}^{2\alpha_1}\dots  \omega_{p}^{2\alpha_{p}}\,d\omega 
=2^{-2\abs{\alpha}+1}\frac{(2\alpha)!\pi^{\frac{p}{2}}}{\alpha!\Gamma(\frac{p}{2}+\abs{\alpha})}.
\end{align*}
\item 
\label{lemma:int_techniques:ii}
For $p=p'+p''$ and $\omega\in S^{p-1}$, $\alpha\in\Z_{\geq0}^p$ we write $\omega''=(\omega_{p'+1}, \ldots \omega_p)$ and $\alpha''=(\alpha_{p'+1},\ldots,\alpha_p)$. Then for $\gamma \in \C$ with $\Re \gamma > -p''-2|\alpha''|$ we have
\begin{equation*}
\int_{S^{p-1}}\omega_1^{2\alpha_1}\dots \omega_p^{2\alpha_p}\abs{\omega''}^\gamma \,d\omega =
2^{-2\abs{\alpha}+1}\frac{\pi^{\frac{p}{2}}(2\alpha)!\Gamma(\frac{\gamma+p''}{2}+\abs{\alpha''})}{\alpha!\Gamma(\frac{\gamma+p}{2}+\abs{\alpha})\Gamma(\frac{p''}{2}+\abs{\alpha''})}.
\end{equation*}
\end{enumerate}
\end{lemma}
\begin{proof}
Ad (i):
By the repeated use of the coordinates $(-1,1)\times S^{m-1}\to S^{m}$, $(r,\eta)\mapsto\omega=(r,\sqrt{1-r^2}\eta)$ with $d\omega=(1-r^2)^{\frac{m-2}{2}}\,dr\,d\eta$ and the beta integral we find
\begin{multline*}
\int_{S^{p-1}}\omega_{1}^{2\alpha_1}\dots  \omega_{p}^{2\alpha_{p}}\,d\omega=2\prod_{l=1}^{p-1}\frac{\Gamma(\alpha_l+\frac{1}{2})\Gamma(\frac{p-l}{2}+\alpha_{l+1}+ \dots + \alpha_{p})}{\Gamma(\frac{p+1-l}{2}+\alpha_l+\dots + \alpha_{p})} \\
=\frac{2}{\Gamma(\frac{p}{2}+\abs{\alpha})}\prod_{l=1}^{p}\Gamma(\alpha_l+\tfrac{1}{2})= 2^{-2\abs{\alpha}+1}\frac{(2\alpha)!\pi^{\frac{p}{2}}}{\alpha!\Gamma(\frac{p}{2}+\abs{\alpha})}.
\end{multline*}

Ad (ii): Similar as in (i) we find
\begin{align*}
& \int_{S^{p-1}}\omega_1^{2\alpha_1}\dots \omega_p^{2\alpha_p}\abs{\omega''}^\gamma \,d\omega\\
& \hspace{1cm} =\int_{-1}^{1}r^{2\alpha_1}(1-r^2)^{\frac{\gamma +p-3}{2}+\alpha_2+\dots+\alpha_p}dr \int_{S^{p-2}}\omega_2^{2\alpha_2}\dots \omega_p^{2\alpha_p}\abs{\omega''}^\gamma \,d\omega\\
& \hspace{1cm} = B\left(\alpha_1+\tfrac{1}{2},\tfrac{\gamma+p-1}{2}+\alpha_2+\dots+\alpha_p\right)\int_{S^{p-2}}\omega_2^{2\alpha_2}\dots \omega_p^{2\alpha_p}\abs{\omega''}^\gamma \,d\omega\\
& \hspace{1cm} = \ldots = \prod_{l=1}^{p'}B\left( \alpha_l+\tfrac{1}{2},\tfrac{\gamma+p-l}{2}+\alpha_{l+1}+\dots \alpha_{p} \right) \int_{S^{p''-1}}\omega_{p'+1}^{2\alpha_{p'+1}}\dots\omega_p^{2\alpha_p}\,d\omega.
\end{align*}
The remaining integral can be evaluated with (i).
\end{proof}

Fix $p,q\geq1$ and let $\mathbb{S}:=\{ (\omega,\eta)\in \R^p\times\R^q, \abs{\omega}^4+\abs{\eta}^2=1 \}\subseteq\R^{p+q}$. By $d(X,Z)=dX\,dZ$ we denote the normalized Lebesgue measure on $\R^p\times\R^q\cong\R^{p+q}$.

\begin{lemma}[{\cite[Chapter I]{cowling_1982}}]
\label{lemma:polar_coordinates_H_type}
There exists a unique smooth measure $d\mathbb{S}$ on $\mathbb{S}$ such that for every $\varphi \in C^\infty_c(\R^{p+q})$:
$$\int_{\R^{p+q}}\varphi(X,Z)\,d(X,Z)=\int_{\R_+}r^{p+2q-1}\int_{\mathbb{S}}\varphi((r\omega,r^2\eta))\,d\mathbb{S}(\omega,\eta)\,dr.$$
\end{lemma}

\begin{lemma}
For the measure $d\mathbb{S}$ of Lemma \ref{lemma:polar_coordinates_H_type} the following integral formula holds for all $\varphi\in C_c^\infty(\mathbb{S})$:
\label{lemma:integral_formula_polar_coordinates}
$$\int_{\mathbb{S}}\varphi(\omega,\eta)\,d\mathbb{S}(\omega,\eta) = 2\int_{0}^{1} x^{p-1} (1-x^4)^{\frac{q}{2}-1} \int_{S^{p-1}}\int_{S^{q-1}}\varphi(x\omega,\sqrt{1-x^4} \eta)\,d\eta\,d\omega\,dx.$$
\end{lemma}

\begin{proof}
Let $\psi \in C^\infty_c(\R^{p+q})$. Using polar coordinates we find
\begin{equation*}
 \int_{\R^{p+q}} \psi(X,Z)\,d(X,Z) = \int_{\R_+}s^{p-1}\int_{S^{p-1}}\int_{\R_+}t^{q-1}\int_{S^{q-1}} \psi(s\omega,t\eta)\,d\eta\,dt\,d\omega\,ds.
\end{equation*}
Now choosing coordinates $(s,t)=(rx,r^2\sqrt{1-x^4})$ we obtain
$$
= 2\int_{\R_+}r^{p+2q-1}
\int_{0}^{1} x^{p-1} (1-x^4)^{\frac{q}{2}-1} \int_{S^{p-1}}\int_{S^{q-1}}\psi(rx\omega,r^2\sqrt{1-x^4} \eta)\,d\eta\,d\omega\,dx\,dr.
$$
Then Lemma~\ref{lemma:polar_coordinates_H_type} implies the statement.
\end{proof}

For $\lambda\in\C$ with $\Re\lambda>-(p+2q)$ the function
$$ v_\lambda(X,Z) = \frac{1}{\Gamma(\frac{\lambda+p+2q}{2})}(\abs{X}^4+\abs{Z}^2)^{\frac{\lambda}{4}} $$
is locally integrable and hence defines a distribution $v_\lambda\in\mathcal{D}'(\R^{p+q})$.

\begin{corollary}\label{cor:HeisenbergNormPowersHolomorphic}
The family of distributions $v_\lambda\in\mathcal{D}'(\R^{p+q})$ extends holomorphically to an entire function in $\lambda\in\C$. For $\lambda\notin-(p+2q)-2\Z_{\geq0}$ we have $\Supp v_\lambda=\R^{p+q}$ and for $\lambda=-(p+2q)-2N\in-(p+2q)-2\Z_{\geq0}$ we have
$$ v_{-(p+2q)-2N}(X,Z) = \frac{(-1)^N2^{-N}N!\pi^{\frac{p+q}{2}}}{\Gamma(\frac{2N+p+2q}{4})} \sum_{i+2j=N} \frac{2^{-i}\Gamma(\frac{p+2i}{4})}{i!j!\Gamma(\frac{p+2i}{2})}\Delta^i\square^j\delta(X,Z), $$
where $\Delta=\sum_{i=1}^p\frac{\partial^2}{\partial X_i^2}$ and $\Box=\sum_{j=1}^q\frac{\partial^2}{\partial Z_j^2}$ are the Laplacians on $\R^p$ and $\R^q$.
\end{corollary}

\begin{proof}
Using the polar coordinates $(r\omega,r^2\eta)$, $(\omega,\eta)\in \mathbb{S}$, of Lemma~\ref{lemma:polar_coordinates_H_type}, the distribution $v_\lambda$ is given by
$$ \frac{1}{\Gamma(\frac{\lambda+p+2q}{2})}r^{\lambda+p+2q-1}, $$
which is holomorphic in $\lambda\in\C$ by Lemma~\ref{lemma:absolute_value_distr_holomorphic}. For $\lambda\notin-(p+2q)-2\Z_{\geq0}$ we have $\Supp r^{\lambda+p+2q-1}=\R_+$ and hence $\Supp v_\lambda=\R^{p+q}$. For $\lambda=-(p+2q)-2N$, $N\in\Z_{\geq0}$, Lemma~\ref{lemma:absolute_value_distr_holomorphic} shows that
$$ \frac{1}{\Gamma(\frac{\lambda+p+2q}{2})}r^{\lambda+p+2q-1} = (-1)^N\frac{N!}{2(2N)!}\delta^{(2N)}(r). $$
Let $\varphi \in C_c^\infty(\nbar)$. Then by Lemma~\ref{lemma:integral_formula_polar_coordinates} and Lemma~\ref{lemma:combinatorics_deriviative} we have
\begin{multline*}
\frac{d^{2N}}{dt^{2N}}\biggr|_{t=0} \int_{\mathbb{S}}\varphi(r\omega,r^2\eta) \,d\mathbb{S} =
2\sum_{i+2j=2N} \sum_{\abs{\alpha}=i} \sum_{\abs{\beta}=j} \frac{(2N)!}{\alpha!\beta!}\int_{0}^{1}x^{p+i-1}(1-x^4)^{\frac{q+j}{2}-1} \,dx\\ \times \int_{S^{p-1}}\omega^\alpha \,d\omega \int_{S^{q-1}} \eta^\beta \,d\eta \left( \frac{\partial}{\partial X} \right)^\alpha \left( \frac{\partial}{\partial Z} \right)^\beta \varphi(0,0).
\end{multline*}
Now the integrals over the spheres vanish for odd-length multi-indices and the remaining integrals can be computed using Lemma~\ref{lemma:int_techniques}. We obtain
$$
=\sum_{i+2j=N} \sum_{\abs{\alpha}=i} \sum_{\abs{\beta}=j} 2^{-2i-2j+1} \pi^{\frac{p+q}{2}}\frac{(2N)!B(\frac{p+2i}{4},\frac{q}{2}+j)}{\alpha! \beta! \Gamma(\frac{p}{2}+i)\Gamma(\frac{q}{2}+j)} \left( \frac{\partial}{\partial X} \right)^{2\alpha }\left( \frac{\partial}{\partial Z} \right)^{2\beta} \varphi(0,0),
$$
which is by the Multinomial Theorem~\eqref{theorem:multinomial_theorem} equal to
\begin{equation*}
 = \left \langle \sum_{i+2j=N} 2^{-2i-2j+1} \pi^{\frac{p+q}{2}}\frac{(2N)!\Gamma(\frac{p+2i}{4})}{\Gamma(\frac{2N+p+2q}{4})i!j! \Gamma(\frac{p}{2}+i)}\Delta_{\mathfrak{v}}^i\square^j\delta,  \varphi \right \rangle.\qedhere
\end{equation*}
\end{proof}

\section{Combinatorial identities}

We recall Faà di Bruno's Formula for the higher derivatives of the composition of two real-valued functions $f$ and $g$ on $\R$:
\begin{equation}\label{eq:faa_di_bruno}
 \frac{d^n}{dx^n}\Big[(f\circ g)(x)\Big] = \sum_{k_1+2k_2+\cdots+nk_n=n}\frac{n!}{k_1!\cdots k_n!}\Bigg[\frac{d^{k_1+\cdots+k_n}f}{dx}\circ g\Bigg](x)\prod_{m=1}^n\Bigg(\frac{1}{m!}\frac{d^mg}{dx^m}\Bigg)^{k_m}.
\end{equation}

\begin{lemma}
	\label{lemma:combinatorics_deriviative}
	Let $f \in C^\infty(\R^2)$, then $$ \frac{d^k}{d t^k}\biggr|_{t=0}f(tx, t^2y)= \sum_{i+2j=k} \frac{k!}{i!j!} x^iy^j\frac{\partial^{i+j}f}{\partial x^i \partial y^j}(0,0). $$
\end{lemma}
\begin{proof}
	It is clear that the left hand side is an expression of the form
	$$ \sum_{i+2j=k}c_{ijk}x^iy^j\frac{\partial^{i+j}f}{\partial x^i \partial y^j}(0,0) $$ 
	with certain non-negative integer coefficients $c_{ijk}$.
	By riffle shuffle permutation the coefficients $c_{ijk}$ are divisible by $\binom{k}{i}$ and the normalized coefficients $\tilde{c}_{ijk}:=\binom{k}{i}^{-1}c_{ijk}$ satisfy $$\frac{d^{2j}}{d t^{2j}}\biggr|_{t=0}\varphi(t^2y)=\tilde{c}_{ijk}\frac{\partial^j\varphi}{\partial y^j}(0) \qquad \forall\,\varphi\in C^\infty(\R).$$
	By Faà di Bruno's Formula~\eqref{eq:faa_di_bruno} we have $\tilde{c}_{ijk}=\frac{(2j)!}{j!}$. Hence $c_{ijk}= \frac{(i+2j)!}{i!j!}$.
\end{proof}

For commuting variables $x_1,\ldots,x_n$ and $m\in\Z_{\geq0}$ the Multinomial Theorem holds:
\begin{equation}
 (x_1+\cdots+x_n)^m=\sum_{\substack{\alpha \in \Z_{\geq 0}^n\\\abs{\alpha}=m}}\frac{m!}{\alpha_1!\cdots \alpha_n!}x_1^{\alpha_1}\cdots x_n^{\alpha_n}.\label{theorem:multinomial_theorem}
\end{equation}

For $a\in\R$ and $n\in\Z_{\geq0}$ let $(a)_n=a(a+1)\cdots(a+n-1)$ denote the Pochhammer symbol. For multi-indices $\alpha\in\N^n$ and $x\in\R^n$ we further write
$$ (x)_\alpha = (x_1)_{\alpha_1}\cdots(x_n)_{\alpha_n}. $$

\begin{lemma}\label{lem:MultiindexSum}
For any $x\in\R^n$ and $m\in\N$ we have
$$ \sum_{|\alpha|=m}\frac{(x)_\alpha}{\alpha!} = \frac{(x_1+\cdots+x_n)_m}{m!}. $$
\end{lemma}

\begin{proof}
The proof is by induction on $n$. For $n=1$ the statement is obvious. For the induction step write $x=(x',x_n)\in\R^{n-1}\times\R=\R^n$ and $\alpha=(\beta,k)$ with $\beta\in\Z_{\geq0}^{n-1}$ and $0\leq k\leq m$, then
\begin{align*}
 \sum_{|\alpha|=m}\frac{(x)_\alpha}{\alpha!} &= \sum_{k=0}^m\frac{(x_n)_k}{k!}\sum_{|\beta|=m-k}\frac{(x')_{\beta}}{\beta!}\\
 &= \sum_{k=0}^m\frac{(x_n)_k(x_1+\cdots+x_{n-1})_{m-k}}{k!(m-k)!}\\
 &= \frac{(x_1+\cdots+x_{n-1}+x_n)_m}{m!},
\end{align*}
where, in the last step, we have used the identity
$$ \sum_{k=0}^m\binom{m}{k}(a)_k(b)_{m-k} = (a+b)_m, $$
which is equivalent to the Vandermonde identity
\begin{equation*}
 \sum_{k=0}^m\binom{a}{k}\binom{b}{m-k} = \binom{a+b}{m},
\end{equation*}
since $\frac{1}{k!}(a)_k=(-1)^k\binom{-a}{k}$.
\end{proof}

\newpage
\addtocontents{toc}{\protect\setcounter{tocdepth}{4}}
\addcontentsline{toc}{subsubsection}{References}

\providecommand{\bysame}{\leavevmode\hbox to3em{\hrulefill}\thinspace}
\providecommand{\href}[2]{#2}

\vspace{1cm}

\textsc{JF: Department of Mathematics, Aarhus University, Ny Munkegade 118, 8000 Aarhus C, Denmark}\par
\textit{E-Mail address:} \texttt{frahm@math.au.dk}\\\par

\textsc{CW: Department of Mathematics, Aarhus University, Ny Munkegade 118, 8000 Aarhus C, Denmark}\par
\textit{E-Mail address:} \texttt{weiske@math.au.dk}

\end{document}